\documentclass[11pt]{amsart}
\usepackage{amssymb,graphicx,amscd}

\begin{document}

\newtheorem{thm}{Theorem}

\theoremstyle{definition}
\newtheorem{quest}[thm]{Question}
\newtheorem{ex}[thm]{Example}
\newtheorem{lemma}[thm]{Lemma}
\newtheorem{remark}[thm]{Remark}
\newtheorem{cor}[thm]{Corollary}
\newtheorem{defn}[thm]{Definition}
\newcommand{\linnum}{\stepcounter{thm}\tag{\thethm}}

\newcommand{\Chat}{\widehat{\mathbb C}}
\newcommand{\Qhat}{\overline{\mathbb Q}}
\newcommand{\Otilde}{\widetilde{\mathcal{O}}}
\newcommand{\bC}{\mathbb C}
\newcommand{\bH}{\mathbb H}
\newcommand{\bN}{\mathbb N}
\newcommand{\bQ}{\mathbb Q}
\newcommand{\Qbar}{\overline{\mathbb Q}}
\newcommand{\bR}{\mathbb R}
\newcommand{\bZ}{\mathbb Z}
\newcommand{\lcm}{\mbox{lcm}}
\newcommand{\co}{\colon\thinspace}

\newcommand{\cB}{\mathcal{B}}
\newcommand{\cE}{\mathcal{E}}
\newcommand{\cF}{\mathcal{F}}
\newcommand{\cH}{\mathcal{H}}
\newcommand{\cL}{\mathcal{L}}
\newcommand{\cO}{\mathcal{O}}
\newcommand{\cP}{\mathcal{P}}
\newcommand{\cQ}{\mathcal{Q}}
\newcommand{\cR}{\mathcal{R}}
\newcommand{\cS}{\mathcal{S}}
\newcommand{\cT}{\mathcal{T}}
\newcommand{\cU}{\mathcal{U}}
\newcommand{\SR}{S_{\mathcal{R}}}
\newcommand{\SRp}{S_{\mathcal{R}'}}
\newcommand{\tR}{t_{\mathcal{R}}}
\newcommand{\tRp}{t_{\mathcal{R}'}}
\newcommand{\cl}{\overline}
\newcommand{\bdry}{\partial}
\newcommand{\Fix}{\mathop{\textrm Fix}}
\newcommand{\expm}{\varphi}
\newcommand{\subm}{\sigma_{\mathcal{R}}}
\newcommand{\submp}{\sigma_{\mathcal{R}'}}
\newcommand{\tw}{g}
\newcommand{\Star}{\textrm{Star}}
\renewcommand{\star}{\textrm{star}}
\newcommand{\teichm}{\Sigma}
\newcommand{\fslope}{\tau_s}
\newcommand{\Imag}{\mbox{\rm Im}}
\newcommand{\Real}{\mbox{\rm Re}}
\newcommand{\zn}{\nu}
\newcommand{\zt}{\tau}
\newcommand{\zo}{\omicron}
\newcommand{\zp}{\pi}
\newcommand{\zL}{\Lambda}
\newcommand{\zG}{\Gamma}
\newcommand{\zF}{\Phi}
\newcommand{\zl}{\lambda}
\newcommand{\zm}{\mu}
\newcommand{\zg}{\gamma}
\newcommand{\zd}{\delta}
\newcommand{\za}{\alpha}
\newcommand{\zh}{\eta}
\newcommand{\zs}{\sigma}
\newcommand{\zS}{\Sigma}
\newcommand{\zb}{\beta}
\newcommand{\zv}{\varphi}
\newcommand{\ze}{\epsilon}
\newcommand{\zw}{\omega}
\newcommand{\zi}{\iota}
\newcommand{\zj}{\psi}
\newcommand{\zr}{\rho}
\newcommand{\zf}{\phi}
\newcommand{\zq}{\theta}

\newcommand{\sections}{\renewcommand{\thethm}{\thesection.\arabic{thm}}
           \setcounter{thm}{0}}
\newcommand{\subsections}{\setcounter{thm}{0}\renewcommand{\thethm}
           {\thesubsection.\arabic{thm}}}
\newcommand{\nosubsections}{\renewcommand{\thethm}{\thesection.\arabic{thm}}
           \setcounter{thm}{0}}

\title[nearly euclidean Thurston maps]{nearly euclidean Thurston maps}

\author{J. W. Cannon}
\address{Department of Mathematics\\ Brigham Young University\\ Provo, UT
84602\\ U.S.A.}
\email{cannon@math.byu.edu}

\author{W. J. Floyd}
\address{Department of Mathematics\\ Virginia Tech\\
Blacksburg, VA 24061\\ U.S.A.}
\email{floyd@math.vt.edu}
\urladdr{http://www.math.vt.edu/people/floyd}

\author{W. R. Parry}
\address{Department of Mathematics\\ Eastern Michigan University\\
Ypsilanti, MI 48197\\ U.S.A.}
\email{walter.parry@emich.edu}

\author{K. M. Pilgrim}
\address{Department of Mathematics\\ Indiana University\\
Bloomington, IN 47405}
\email{pilgrim@indiana.edu}

%\begin{abstract}
%\end{abstract}

%\thanks{This work was supported in part by NSF research grants.}
%\keywords{finite subdivision rule, rational map, conformality}
%\subjclass{Primary 37F10, 52C20; Secondary 57M12}
\date\today
\maketitle

In this work, we take an in-depth look at Thurston's combinatorial
characterization of rational functions for a particular class of maps
we call \emph{nearly Euclidean Thurston} (NET) maps.  Suppose $f\co
S^2\to S^2$ is an orientation-preserving branched map. Following
Thurston, we define $\nu_f\co S^2 \to \bZ_+ \cup \{\infty\}$ by
$$\nu_f(x) =
\begin{cases}
\lcm(D_f(x)) &\text{if $D_f(x)$ is finite,}\\
\infty &\text{if $D_f(x)$ is infinite,}\\
\end{cases}$$
where $D_f(x) = \{n\in\bZ_+\co$ there exists $m\in\bZ_+$ and $y\in
S^2$ such that $f^{\circ m}(y) = x$ and $f^{\circ m}$ has degree $n$
at $y\}$. The points $x\in S^2$ with $\nu_f(x) > 1$ are called {\em
postcritical points}, and the set of postcritical points is denoted by
$P_f$. The map $f$ is {\em postcritically finite} if $P_f$ is
finite. A Thurston map is an orientation-preserving branched map $f\co
S^2\to S^2$ which is postcritically finite.  In this case, we denote
by $\cT$ the Teichm\"{u}ller space of the orbifold $(S^2,\nu_f)$.  The
map $f$ induces a map $\zS_f\co \cT\to \cT$ by pulling back complex
structures.

In a CBMS Conference in 1983, Thurston \cite{Th} addressed the
problem of determining when a Thurston map $f\co S^2 \to S^2$ is
equivalent to a rational map, where $f\sim g$ if there is a
homeomorphism $h\co S^2 \to S^2$ such that $h(P_f) = P_g$, $(h\circ f)
\big|_{P_f} = (g\circ h) \big|_{P_f}$, and $h\circ f$ is isotopic, rel
$P_f$, to $g\circ h$. His main theorems were 1) that $f$ is equivalent
to a rational map exactly if $\teichm_f$ has a fixed point, and 2) if
$(S^2,\zn_f)$ is hyperbolic, then $\teichm_f$ has a fixed point
exactly if there are no Thurston obstructions (these will be defined
next). Thurston didn't publish his proofs of the theorems, but proofs
were given later by Douady and Hubbard in \cite{DH}.

Now we define Thurston obstructions. A \textit{multicurve} $\Gamma$ is
a finite collection of pairwise disjoint simple closed curves in
$S^2\setminus P_f$ such that each element of $\Gamma$ is nontrivial,
each element of $\Gamma$ is nonperipheral, and distinct elements of
$\Gamma$ are not isotopic. A multicurve $\Gamma$ is \textit{invariant}
or \textit{$f$-stable} if each element of $f^{-1}(\Gamma)$ is either
trivial, peripheral, or isotopic to an element of $\Gamma$. If $\Gamma$
is an invariant multicurve, then the Thurston matrix $A^{\Gamma}\co
\bR^{\Gamma} \to \bR^{\Gamma}$ is defined in coordinates by
$$A_{\gamma\delta}^{\Gamma} = \sum_{\alpha} \frac{1}{\deg(f\co
\alpha\to\delta)},$$ where the sum is taken over connected components
$\alpha$ of $f^{-1}(\delta)$ which are isotopic to $\gamma$ in $S^2
\setminus P_f$. If $\Gamma$ is an invariant multicurve, the spectral
radius (eigenvalue of largest norm) of $A^{\Gamma}$ is called the
Thurston multiplier of $\Gamma$.  An invariant multicurve is a
\textit{Thurston obstruction} if its Thurston multiplier is at least
one.

Unfortunately, checking whether or not Thurston obstructions exist is
very difficult primarily because there are infinitely many multicurves
to consider.  Our motivation in this work is to better understand
Thurston obstructions and the issue of conformality of finite
subdivision rules (for which see \cite{fsr}). We were led to a class
of Thurston maps which are as simple as possible but yet nontrivial in
this regard.  We call these maps nearly Euclidean Thurston maps.
These are simple generalizations of Latt\`{e}s maps.

In \cite{M2} Milnor characterizes a Latt\`{e}s map as a rational map
from the Riemann sphere to itself such that each of its critical
points is simple (local degree 2) and it has exactly four postcritical
points, none of which is also critical.  We say that a Thurston map is
Euclidean if it is a straightforward generalization of this: a
Thurston map is Euclidean if its degree is at least 2, its local
degree at every critical point is 2 and it has at most four
postcritical points, none of which is also critical.
(Lemma~\ref{lemma:arefour} shows that if there are at most four
postcritical points, then there are exactly four.)  A nearly Euclidean
Thurston (NET) map allows postcritical points to be critical: a
Thurston map is nearly Euclidean if its local degree at every critical
point is 2 and it has exactly four postcritical points. (Now at most
four does not imply four, as is the case for the map $z\mapsto z^2$.)
If $f$ is a Euclidean Thurston map, then the orbifold $(S^2,\zn_f)$ is
Euclidean.  On the other hand, if a NET map $f$ is not Euclidean, that
is, some postcritical point is a critical point, then the orbifold
$(S^2,\zn_f)$ is hyperbolic.  These are the simplest Thurston maps
with hyperbolic orbifolds and nontrivial Teichm\"{u}ller spaces.

Our ultimate goal is to thoroughly understand Thurston obstructions
for NET maps.  This paper is devoted to developing the first
properties of these maps.

Section~\ref{sec:defns} presents definitions and basic facts
concerning NET maps.  These basic facts involve lifting properties of
NET maps.  Every NET map lifts to a map from one torus to another.
From such a lift we obtain a lift from $\mathbb{R}^2$ to itself.

Section~\ref{sec:twists} deals with twists of NET maps.  We twist a NET
map by postcomposing it with a suitable homeomorphism.  We find that
every NET map is a twist of a Euclidean Thurston map.

Section~\ref{sec:examples} presents two examples.  The first of these
is our main example.  The finite subdivision rule associated to this
example is the germ of this paper.  Everything in this paper arose
from studying this example.  The second example in
Section~\ref{sec:examples} shows that the rational function which
appears in the proof of statement 2 of Theorem 1.1 of \cite{BEKP} is a
NET map.

Computing Thurston matrices involves degrees and numbers of 
components of pullbacks of invariant multicurves.  These degrees and
numbers of components are described rather completely for NET maps in
Section~\ref{sec:pullbacks}.

Every homotopy class of simple closed curves in a 4-punctured sphere
is assigned a slope in $\widehat{\mathbb{Q}}=\mathbb{Q}\cup
\{\infty\}$ which characterizes the homotopy class.  Taking pullbacks,
a NET map $f$ induces a self-map $\zs_f\co \widehat{\mathbb{Q}}\cup
\{o\}\to \widehat{\mathbb{Q}}\cup \{o\}$, where $o$ denotes the union
of the classes of inessential and peripheral curves. In Section 4, we
give an algorithm (Theorem~\ref{thm:slopefn2} and the following
discussion) for computing $\zs_f$. The slope of a Thurston obstruction
is a fixed point for $\zs_f$.

Section~\ref{sec:horoballs} begins our study of the induced map on
Teichm\"{u}ller space.  Theorem~\ref{thm:halfsp} shows how
knowledge of the pullback of a given curve under $f$ translates into
an interval of slopes in which the slope of a Thurston obstruction
cannot lie.  At the end of Section~\ref{sec:horoballs}, we use this
result to show that there are no Thurston obstructions for the main
example.

Section~\ref{sec:dehn} discusses Dehn twists in the present context.
Section~\ref{sec:rflns} discusses reflections.
Section~\ref{sec:fnleqns} presents a common framework for the results
of the previous two sections.  Taken together, these three sections
allow us to compute explicit ``functional equations'' satisfied by the
map on Teichm\"{u}ller space induced by a NET map.

Section~\ref{sec:constant} shows how the results of the previous
sections can be applied to the study of maps on Teichm\"{u}ller space
induced by NET maps.  Section~\ref{sec:constant} begins the
characterization of those NET maps whose induced maps on
Teichm\"{u}ller space are constant.  Theorem~\ref{thm:algcformn}
reduces this characterization to a purely algebraic problem concerning
finite Abelian groups generated by two elements.  We then obtain
partial results for this algebraic problem.  Saenz Maldonado extends
these results concerning this algebraic problem in his thesis
\cite{SM}, although a complete solution is not yet in hand.

\section{Definitions and lifts }\label{sec:defns}
\nosubsections

A Thurston map is an orientation-preserving branched covering map from
the 2-sphere to itself which is postcritically finite.

\begin{defn}\label{def:et} A Thurston map is \emph{Euclidean} if its
degree is at least 2, its local degree at each of its critical points
is 2, it has at most four postcritical points and none of its
postcritical points is critical.  
\end{defn}

\begin{defn}\label{def:net} A Thurston map is \emph{nearly Euclidean}
(NET) if its local degree at each of its critical points is 2 and it
has exactly four postcritical points.  
\end{defn}

Note that the definition of Euclidean used here is stronger than the
condition of the orbifold being Euclidean.  Although the definition
only requires that a Euclidean Thurston map have at most four
postcritical points, the first statement of the following lemma shows
that it actually has exactly four.  It follows that a Euclidean
Thurston map is characterized by the property that its orbifold is the
$(2,2,2,2)$-orbifold.

\begin{lemma}\label{lemma:arefour}
\begin{enumerate}
  \item Every Euclidean Thurston map has exactly four postcritical
points, and so every Euclidean Thurston map is nearly Euclidean.
  \item Let $f\co S^2\to S^2$ be a NET map with postcritical set
$P_f$.  Then $f^{-1}(P_f)$ contains exactly four points which are not
critical points.  The map $f$ is Euclidean if and only if these four
points are the points of $P_f$.
\end{enumerate}
\end{lemma}
  \begin{proof} To prove statement 1, let $f\co S^2\to S^2$ be a
Euclidean Thurston map with postcritical set $P_f$ and degree $d$.
Every point of $S^2$ has $d$ preimages under $f$ counting
multiplicity.  Hence $f^{-1}(P_f)$ has $d\left|P_f\right|$ points
counting multiplicity.  The Riemann-Hurwitz formula shows that $f$ has
$2d-2$ critical points.  These points map to $P_f$ with multiplicity 2
and they are distinct from the points of $P_f$.  Combining these facts
yields the inequality $4d-4+\left|P_f\right|\le d\left|P_f\right|$.
Hence $(4-\left|P_f\right|)(d-1)\le 0$.  Since $d>1$, we have that
$\left|P_f\right|\ge 4$.  Since $\left|P_f\right|\le 4$ by assumption,
it follows that $\left|P_f\right|=4$.  This proves statement 1 of
Lemma~\ref{lemma:arefour}.

To prove statement 2, we let $f$ now be a NET map and argue as in the
previous paragraph.  If $n$ is the number of points in $f^{-1}(P_f)$
which are not critical, then $4d-4+n=4d$.  Hence $n=4$.

This proves Lemma~\ref{lemma:arefour}.

\end{proof}

A Latt\`{e}s map as in Milnor's paper \cite{M2} (The definition in
\cite{M3} is more general.) is a rational function which is a
Euclidean Thurston map, and so NET maps are closely related to
Latt\`{e}s maps.  An important property of Latt\`{e}s maps is that
they lift to maps of tori in a special way.  The next theorem shows
that NET maps lift to maps of tori in a more general way and, in fact,
this property characterizes NET maps.  The proof uses the fact that
given four points in $S^2$, there exists a double cover (unique up to
isomorphism) of $S^2$ ramified over exactly these four points.

\begin{thm}\label{thm:eqvtdefn}  Let $f\co S^2\to S^2$ be a
Thurston map.  Then $f$ is nearly Euclidean if and only if there exist
branched covering maps $p_1\co T_1\to S^2$ and $p_2\co T_2\to S^2$
with degree 2 from tori $T_1$ and $T_2$ to $S^2$ such that the set of
branch points of $p_2$ is the postcritical set of $f$ and there exists
a continuous map $\widetilde{f}\co T_1\to T_2$ such that $p_2\circ
\widetilde{f}=f\circ p_1$.  If $f$ is nearly Euclidean, then $f$ is
Euclidean if and only if the set of branch points of $p_1$ is the
postcritical set of $f$.
\end{thm}
  \begin{proof} We begin by proving the backward implication of the
first assertion.  Let $p_1$, $p_2$ and $\widetilde{f}$ be maps as
stated.  It follows that $\widetilde{f}$ is a branched covering map.
Two applications of the Riemann-Hurwitz formula show that
$\widetilde{f}$ is unramified and that $p_1$ and $p_2$ are both
ramified at exactly four points.  Hence the postcritical set of $f$
has exactly four points.  Now we combine the equation $p_2\circ
\widetilde{f}=f\circ p_1$ with the facts that local degrees multiply
under composition of functions, that the local degree of
$\widetilde{f}$ at every point is 1 and that the local degree of $p_2$
at every point is either 1 or 2.  We conclude that the local degree of
$f$ at every point is either 1 or 2.  In other words, the local degree
of $f$ at each of its critical points is 2.  This proves the backward
implication of the first assertion.

To prove the forward implication of the first assertion, suppose that
$f\co S^2\to S^2$ is a NET map with postcritical set $P_2$.  Statement
2 of Lemma~\ref{lemma:arefour} implies that four points of
$f^{-1}(P_2)$ are not critical points of $f$.  Let $P_1$ be this set
of four points.  Now let $T_1$ and $T_2$ be tori, and let $p_1\co
T_1\to S^2$ and $p_2\co T_2\to S^2$ be branched covering maps with
degree 2 such that the set of branch points of $p_1$ is $P_1$ and the
set of branch points of $p_2$ is $P_2$.  Then the restriction of $p_2$
to $T_2\setminus p_2^{-1}(P_2)$ is a covering map to $S^2\setminus
P_2$, and the restriction of $f\circ p_1$ to $T_1\setminus
p_1^{-1}(f^{-1}(P_2))$ is a continuous map to $S^2\setminus P_2$.  The
fundamental group $\zp_1(S^2\setminus P_2)$ is generated by the
homotopy classes of four loops about the elements of $P_2$.  There
exists a group homomorphism from $\zp_1(S^2\setminus P_2)$ to $\bZ/2
\bZ$ which sends these homotopy classes to the nontrivial element of
$\bZ/2 \bZ$.  The kernel of this group homomorphism is the image of
$\zp_1(T_2\setminus p_2^{-1}(P_2))$ in $\zp_1(S^2\setminus P_2)$.  Now
we see that because the elements of $P_1$ are branch points of $p_1$
and the remaining elements of $f^{-1}(P_2)$ are critical points of
$f$, the image of $\zp_1(T_1\setminus p_1^{-1}(f^{-1}(P_2)))$ in
$\zp_1(S^2\setminus P_2)$ is contained in the image of
$\zp_1(T_2\setminus p_2^{-1}(P_2))$.  The standard lifting theorem
from covering space theory now implies that there exists a lift from
$T_1\setminus p_1^{-1}(f^{-1}(P_2))$ to $T_2\setminus p_2^{-1}(P_2)$,
and this lift extends to a lift $\widetilde{f}\co T_1\to T_2$ such
that $p_2\circ \widetilde{f}=f\circ p_1$.  This proves the forward
implication of the first assertion.

The second assertion concerning Euclidean Thurston maps is now clear.

This proves Theorem~\ref{thm:eqvtdefn}.

\end{proof}

We continue this section with a discussion of NET maps.  Let $f\co
S^2\to S^2$ be a NET map.  Let $p_1\co T_1\to S^2$ and $p_2\co T_2\to
S^2$ be covering maps with degree 2 from tori $T_1$ and $T_2$ to $S^2$
and let $\widetilde{f}\co T_1\to T_2$ be a continuous map as in
Theorem~\ref{thm:eqvtdefn} such that $p_2\circ \widetilde{f}=f\circ
p_1$.  Let $P_j$ be the set of branch points of $p_j$ in $S^2$ for
$j\in \{1,2\}$.  We sometimes use the notation $P_j(f)$ instead of
$P_j$ to avoid possible confusion when dealing with more than one NET
map. The set $P_2$ is the postcritical set of $f$.

Let $j\in \{1,2\}$.  Let $q_j\co\bR^2\to T_j$ be a universal covering
map.  The map $p_j\circ q_j\co \bR^2 \to S^2$ is a branched covering
map whose local degree at every ramified point is 2.  Let
$\zL_j\subseteq \bR^2$ be the set of these ramification points.  It is
furthermore true that $p_j\circ q_j$ is regular.  Let $\zG_j$ be its
group of deck transformations.  By choosing $q_j$ appropriately, we
may assume that $\zG_j$ is generated by the set of all Euclidean
rotations of order 2 about the points of $\zL_j$.  Given rotations
$x\mapsto 2 \zl-x$ and $x\mapsto 2 \zm-x$ of order 2 about the points
$\zl,\zm\in \zL_j$, their composition, the second followed by the
first, is the translation $x\mapsto x+2(\zl-\zm)$.  We may, and do,
normalize so that $0\in \zL_j$.  It follows that $\zL_j$ is a lattice
in $\bR^2$ and that the elements of $\zG_j$ are the maps of the form
$x\mapsto 2 \zl\pm x$ for some $\zl\in \zL_j$.

The map $\widetilde{f}$ lifts to a continuous map
$\widetilde{\widetilde{\text{$f$}}}\co \bR^2\to \bR^2$ such that
$q_2\circ \widetilde{\widetilde{\text{$f$}}}=\widetilde{f}\circ q_1$.
Since $\widetilde{f}$ is a covering map, so is
$\widetilde{\widetilde{\text{$f$}}}$.  Hence
$\widetilde{\widetilde{\text{$f$}}}$ is a homeomorphism.  We replace
$q_1$ by $q_1\circ \widetilde{\widetilde{\text{$f$}}}^{-1}$.  As a
result, $\widetilde{f}$ lifts to the identity map.  Because
$\widetilde{f}$ lifts to the identity map, $\zL_1\subseteq \zL_2$ and
$\zG_1\subseteq \zG_2$.  We obtain the standard commutative diagram in
Figure~\ref{fig:cmmvdgm}, where the map from $\bR^2$ to itself is the
identity map and the maps from $\zL_1$ and $\zL_2$ are inclusion maps.

\begin{figure}
  \begin{equation*}
\begin{CD}
\zL_1 @>incl>> \zL_2 \\
@VinclVV  @VinclVV \\
\bR^2 @>id>>\bR^2 \\
@Vq_1VV  @Vq_2VV \\
T_1 @>\widetilde{f}>> T_2 \\
@Vp_1VV @Vp_2VV \\
S^2 @>f>> S^2
\end{CD}
  \end{equation*}
 \caption{ The standard commutative diagram.}
\label{fig:cmmvdgm}
\end{figure}

The group $\zG_j$ contains the group of deck transformations of $q_j$.
It is the subgroup with index 2 consisting of translations of the form
$x\mapsto 2 \zl+x$ with $\zl\in \zL_j$.  Thus we identify $T_j$ with
$\bR^2/2\zL_j$.  The standard commutative diagram implies that
$\mathbb{R}^2/\zG_1$ and $\mathbb{R}^2/\zG_2$ are both identified with
$S^2$.  Thus there is an \emph{identification map} $\zf\co
\mathbb{R}^2/\zG_2\to \mathbb{R}^2/\zG_1$.  To be precise, to evaluate
$f$ at some point $x$, we view $x$ as an element of
$\mathbb{R}^2/\zG_1$.  We lift it to $\mathbb{R}^2$, then project it
to $\mathbb{R}^2/\zG_2$ and then apply the identification map $\zf$ to
obtain $f(x)$.  For Euclidean NET maps, we usually construct this
identification map using an affine automorphism of $\mathbb{R}^2$
which restricts to an affine isomorphism from $\zL_2$ to $\zL_1$.

In this paragraph we discuss how the identification map $\zf$ might
arise from an affine isomorphism $\zF$.  Let $\zF\co \mathbb{R}^2\to
\mathbb{R}^2$ be an affine isomorphism such that $\zF(\zL_2)=\zL_1$.
Every rotation in $\zG_2$ has the form $x\mapsto 2 \zl-x$ for some
$\zl\in \zL_2$.  Suppose that $\zF(x)=Ax+B$, where $A\in
\text{GL}(2,\mathbb{R})$ and $B\in \mathbb{R}^2$.  Let $\zl\in
\zL_2$.  Then
  \begin{equation*}
\zF(2\zl-x)=A(2\zl-x)+B=2(A \zl+B)-(Ax+B)=2\zF(\zl)-\zF(x).
  \end{equation*}
Because $\zF(\zL_2)=\zL_1$, this implies that $\zF$ induces a map
$\zf\co \mathbb{R}^2/\zG_2\to \mathbb{R}^2/\zG_1$, and one checks that
it is a homeomorphism.  In this way we obtain an identification map
$\zf$.

Because $\widetilde{f}$ lifts to the identity map, when we interpret
in terms of group theory, we see that $\widetilde{f}$ is the canonical
group homomorphism from $\bR^2/2\zL_1$ to $\bR^2/2\zL_2$.  Its kernel
is $2\zL_2/2\zL_1\cong \zL_2/\zL_1$.  Thus
$\deg(f)=\deg(\widetilde{f})=\left|\zL_2/\zL_1\right|$.

Let $i\in \{1,2\}$.  We have that $P_i$ is the set of branch points of
$p_i\circ q_i$, that $\zL_i$ is the set of ramification points of
$p_i\circ q_i$ and that $q_i^{-1}(p_i^{-1}(P_i))=\zL_i$.  In this
paragraph we show that $q_1^{-1}(p_1^{-1}(P_2))\subseteq \zL_2$.  For
this, let $x\in P_2$.  Let $y\in p_1^{-1}(x)$.  Then
$p_2(\widetilde{f}(y))=f(p_1(y))=f (x)\in P_2$.  So $\widetilde{f}(y)$
is one of the four points of $T_2$ at which $p_2$ is ramified.  We
conclude that $p_1^{-1}(P_2)$ is contained in the set of $4 \deg(f)$
points of $T_1$ which $\widetilde{f}$ maps to a ramification point of
$p_2$.  This implies that $q_1^{-1}(p_1^{-1}(P_2))\subseteq \zL_2$.
More precisely, if $P_2$ contains $m$ elements in $P_1$ and $n$
elements not in $P_1$, so that $m+n=4$, then $q_1^{-1}(p_1^{-1}(P_2))$
consists of $m+2n$ cosets of $2 \zL_1$ in $\zL_2$.

\section{Twists}\label{sec:twists}\nosubsections

In this section we consider ``twists'' of NET maps.  That is, we
consider how to obtain new NET maps from known ones by postcomposing
with homeomorphisms.  Suppose that $g\co S^2\to S^2$ is a NET map.
Also suppose that $h\co S^2\to S^2$ is an orientation-preserving
homeomorphism such that $h(P_g)\subseteq g^{-1}(P_g)$.  Then the map
$f=h\circ g$ is a NET map if it has at least four postcritical points.
Indeed, it is an orientation-preserving branched map, the local degree
at each of its critical points is 2, and its set of postcritical
points is contained in $h(P_g)$, a set with four elements.  If $f$ is
a NET map, then in the usual notation, $P_1(f)=P_1(g)=P_2(g)$ and
$P_2(f)=h(P_2(g))$.

We continue the discussion of the previous paragraph by considering
conditions under which the map $f=h\circ g$ has at least four
postcritical points.  We begin with the observation that if every
element of $P_g$ is the image under $g$ of a critical point of $g$,
then every element of $h(P_g)$ is the image under $f$ of a critical
point of $f$, and so $f$ has at least four postcritical points.
Statement 2 of Lemma~\ref{lemma:arefour} shows that $g^{-1}(P_g)$
contains exactly four points which are not critical points of $g$.  So
if some point $x$ of $P_g$ is not the image under $g$ of a critical
point of $g$, then $g^{-1}(x)$ contains at most four points and so
$\deg(g)\le 4$.  We conclude that if $\deg(g)\ge 5$, then $f$ has at
least four postcritical points.  If $\deg(g)=3$, then the preimage
under $g$ of every element of $P_g$ contains three points counting
multiplicity.  There cannot be two critical points in such a preimage
because then the preimage would have at least four points counting
multiplicity.  It easily follows that every such preimage contains one
point which is critical and one point which is not.  So if
$\deg(g)=3$, then $f$ has at least four postcritical points.  Thus if
either $\deg(g)=3$ or $\deg(d)\ge 5$, then $f=h\circ g$ has at least
four postcritical points.  This may fail if either $\deg(g)=2$ or
$\deg(g)=4$.

In this paragraph we consider the converse to the discussion of the
previous two paragraphs.  Let $f$ be a NET map with $P_1=P_1(f)$ and
$P_2=P_2(f)$ as usual.  Let $h\co S^2\to S^2$ be any
orientation-preserving homeomorphism which maps $P_1$ to $P_2$.  Let
$g=h^{-1}\circ f$.  Then $\deg(g)=\deg(f)\ge 2$, the local degree of
$g$ at each of its critical points is 2, and the postcritical points
of $g$ are contained in $h^{-1}(P_2)=P_1$, a set with four elements
containing no critical points of $g$.  This means that $g$ is a
Euclidean Thurston map.

We have proved the following theorem.

\begin{thm}\label{thm:cmpon} \begin{enumerate}
  \item If $g\co S^2\to S^2$ is a NET map and $h\co S^2\to S^2$ is an
orientation-preserving homeomorphism such that $h(P_g)\subseteq
g^{-1}(P_g)$, then $f=h\circ g$ is a NET map if it has at least four
postcritical points.
  \item Let $f$ be a NET map with $P_1=P_1(f)$ and $P_2=P_2(f)$ as
usual.  Let $h\co S^2\to S^2$ be any orientation-preserving
homeomorphism with $h(P_1)=P_2$.  Then $f=h\circ g$, where $g\co
S^2\to S^2$ is a Euclidean Thurston map with $P_g=P_1$ and
$P_2\subseteq g^{-1}(P_g)$, so that $h(P_g)\subseteq g^{-1}(P_g)$.
\end{enumerate}
\end{thm}

\section{Construction of examples }\label{sec:examples}
\nosubsections

Let $g$ be a NET map in the setting of Section~\ref{sec:defns}.  Let
$h\co S^2\to S^2$ be an orientation-preserving homeomorphism such that
$h(P_g)\subseteq g^{-1}(P_g)$.  Statement 1 of Theorem~\ref{thm:cmpon}
implies that the map $f=h\circ g$ is a NET map if it has at least four
postcritical points.  Also suppose that $g$ is the subdivision map of
a finite subdivision rule $\cQ$ and that $h$ maps the 1-skeleton of
$S^2$ into the 1-skeleton of its first subdivision $\cQ(S^2)$, taking
vertices of $S^2$ to vertices of $\cQ(S^2)$.  Then $f$ is the
subdivision map of a finite subdivision rule $\cR$.  The subdivision
complex of $\cR$ is $S^2$ with cell structure the image under $h$ of
the original cell structure.  Even though the subdivision complexes of
$\cQ$ and $\cR$ are probably different, their first subdivisions are
identical.  As noted in the introduction, if $g$ is Euclidean and if
$h$ does not stabilize the postcritical set of $g$, then the orbifold
structure of $S^2$ for $f$ is hyperbolic.  These observations allow us
to easily construct finite subdivision rules whose subdivision maps
are NET maps whose orbifolds are hyperbolic.

\begin{ex}\label{ex:main} This takes us to our main example.  It will
be a NET map of the form $f=h\circ g$, where $g$ is Euclidean. In the
process of defining $g$ and $f$, we will show that each inherits the
additional structure of being a subdivision map for a finite
subdivision rule.  To define $f$ we first define $g$ and then we
define $h$.

We begin the definition of $g$ by setting $\zL_2=\bZ^2$.  Let $\cS_2$
be the tiling of the plane by $2\times 1$ rectangles so that
four rectangles meet at every lattice point $(x,y)$ for which $x$ is
even as in Figure~\ref{fig:maintess}.  Since every such rectangle
contains six elements of $\zL_2$, every such rectangle should be
viewed as a hexagon rather than a quadrilateral.  Recall that $\zG_i$
is the group generated by 180 degree rotations about the lattice
points of $\zL_i$ for $i\in \{1,2\}$.  Every such rectangle is a
fundamental domain for the action of $\zG_2$ on $\bR^2$.  Let $F_2$ be
the rectangle which has as corners $(0,0)$, $(2,0)$ and $(0,1)$.

Let $\zL_1=\left<(2,-1),(0,5)\right>$, the sublattice of $\zL_2$
generated by $(2,-1)$ and $(0,5)$.  A fundamental domain $F_1$ for the
action of $\zG_1$ on $\bR^2$ is hatched in Figure~\ref{fig:maintess}.
We give $F_1$ a cell structure so that the boundary of $F_1$ is its
1-skeleton and its vertices are at $(0,0)$, $(2,-1)$, $(4,-2)$,
$(4,3)$, $(2,4)$ and $(0,5)$.  We regard the hatched region in
Figure~\ref{fig:maintess} as a subdivision of $F_1$.  Let $\cS_1$ be
the tiling of the plane by the images of $F_1$ under the elements of
$\zG_1$.

We next construct an identification map $\zf\co \mathbb{R}^2/\zG_2\to
\mathbb{R}^2/\zG_1$ for $g$.  Because we wish to preserve cell structure,
instead of directly using an affine automorphism of $\mathbb{R}^2$ as
in Section~\ref{sec:defns}, we proceed as follows.

Let $j\in \{1,2\}$.  Let $T_j=\bR^2/2\zL_j$, and let $q_j\co \bR^2\to
T_j$ be the canonical quotient map.  Let $p_j\co T_j\to \bR^2/\zG_j$
be the canonical quotient map.  The tiling $\cS_j$ induces a tiling of
$\bR^2/\zG_j$ with one tile.  Because $F_2$ is cellularly homeomorphic
to $F_1$ in a way which respects the edge pairings induced by $\zG_2$
and $\zG_1$, there exists an orientation-preserving cellular
homeomorphism $\zf\co \bR^2/\zG_2\to \bR^2/\zG_1$ which maps
$p_2(q_2(0,0))$ to $p_1(q_1(0,0))$.  Such a homeomorphism $\zf$ can be
constructed as follows.  Define $\zf$ to map the four points
$p_2(q_2(1,0))$, $p_2(q_2(0,0))$, $p_2(q_2(0,1))$ and $p_2(q_2(1,1))$
to the four points $p_1(q_1(2,-1))$, $p_1(q_1(0,0))$, $p_1(q_1(0,5))$
and $p_1(q_1(2,4))$ in order.  The image of $\partial F_2$ in
$\mathbb{R}^2/\zG_2$ is an arc joining the first four points in order,
and the image of $\partial F_1$ in $\mathbb{R}^2/\zG_1$ is an arc
joining the second four points in order.  We extend $\zf$ to a
homeomorphism from the first arc to the second arc.  Finally, we
extend this map to an orientation-preserving homeomorphism $\zf\co
\mathbb{R}^2/\zG_2\to \mathbb{R}^2/\zG_1$.  We use this homeomorphism
to identify these two spaces and we identify the result with $S^2$.
With this identification, the set of branch points of $p_1$ equals the
set of branch points of $p_2$.  This identification map is isotopic to
the one induced by the linear automorphism of $\mathbb{R}^2$ whose
matrix with respect to the standard basis is
$\left[\begin{smallmatrix} 2& 0 \\ -1 & 5 \end{smallmatrix}\right]$.

Let $\widetilde{g}\co T_1\to T_2$ be the canonical map, and let $g\co
S^2\to S^2$ be the map which it induces.  Then $g$ is a Euclidean
Thurston map.  Its postcritical set $P_g$ is the set of branch points
of $p_1$ and $p_2$.  It is also the subdivision map of a finite
subdivision rule $\cQ$.  The subdivision complex of $\cQ$ is $S^2$
with cell structure the push forward of $\cS_1$ under $p_1\circ q_1$.
This is the same as the push forward of $\cS_2$ under $p_2\circ q_2$.
Its first subdivision is the push forward of $\cS_2$ under $p_1\circ
q_1$.  Figure~\ref{fig:maing} indicates the action of $g$.  The right
portion of Figure~\ref{fig:maing} shows the push forward of $\cS_1$
under $p_1\circ q_1$ in $S^2$ and the left portion of
Figure~\ref{fig:maing} shows the push forward of $\cS_2$ under
$p_1\circ q_1$ in $S^2$.  Most of the vertices in the left portion are
labeled with preimages in $\bR^2$.

Thus far we have the map $g$.  For $h\co S^2\to S^2$ we choose an
orientation-preserving homeomorphism which takes the 1-skeleton of the
push forward of $\cS_1$ into the 1-skeleton of the push forward of
$\cS_2$ such that $h$ fixes the images of
$(1,0),(0,0),\dotsc,(0,5),(1,5)$ and $h$ maps the image of $(2,4)$ to
the image of $(2,5)$ and the image of $(2,-1)$ to the image of
$(2,0)$.  Let $f=h\circ g$.  The action of $f$ is indicated in
Figure~\ref{fig:mainf}.  The map $f$ preserves the edge labels which
are given.  We see that $h(P_g)\subseteq g^{-1}(P_g)$.  As discussed
in the beginning of this section, it follows that $f$ is a NET map and
it is the subdivision map of a finite subdivision rule $\cR$.  The
single tile type of $\cR$ is a hexagon. The subdivision of the hexagon
is shown in Figure~\ref{fig:fsubrule}.  It is easy to check that $\cR$
has bounded valence and that the mesh of $\cR$ approaches 0
combinatorially.  The mapping scheme of $f$ is shown in
Figure~\ref{fig:mainscheme}, where points are labeled by their
preimages in $F_1$ under $p_1\circ q_1$.  This example was designed to
make it difficult to determine the invariant multicurves for possible
Thurston obstructions.

\begin{figure}\begin{center}
\includegraphics{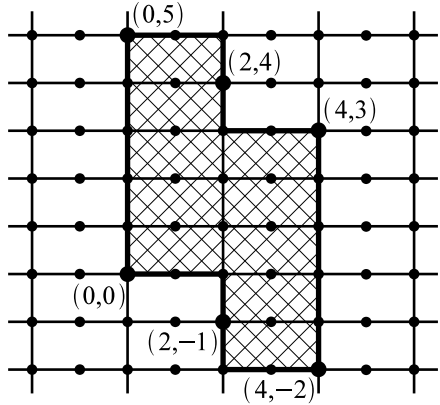} \caption{ A
fundamental domain for $\zG_1$ of the main example.}
\label{fig:maintess}
\end{center}\end{figure}

\begin{figure}\begin{center}
\includegraphics[scale=.8]{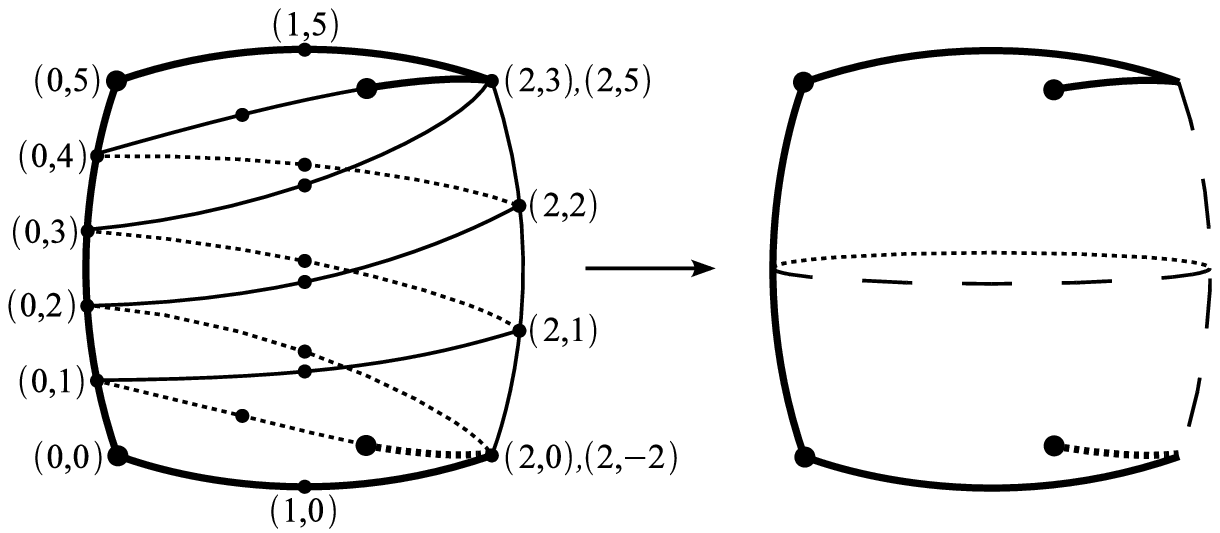} \caption{ The map $g$ of the
main example.}
\label{fig:maing}
\end{center}\end{figure}

\begin{figure}\begin{center}
\includegraphics[scale=.8]{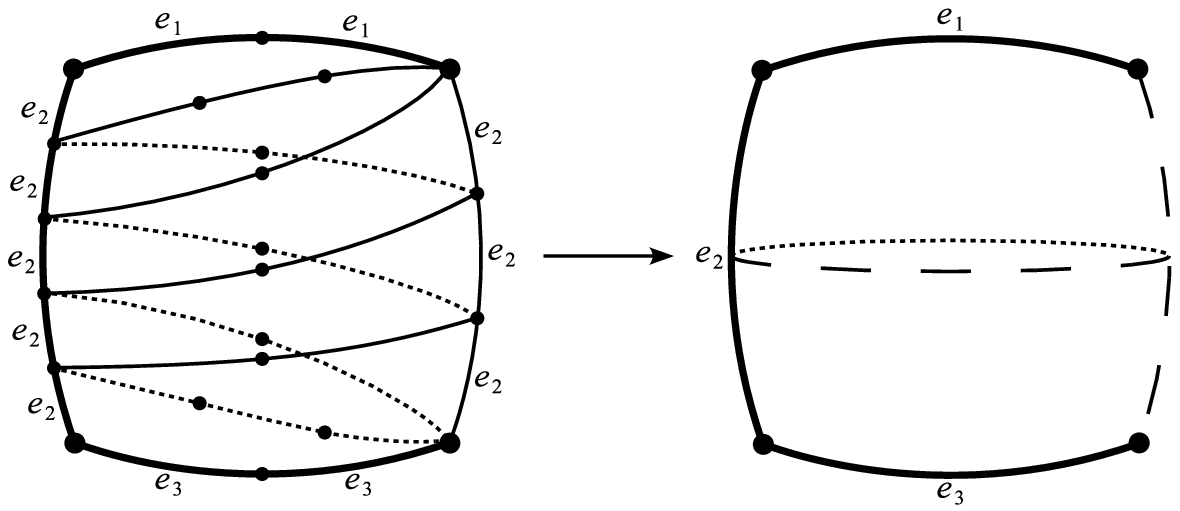} \caption{ The map $f$ of the
main example.}
\label{fig:mainf}
\end{center}\end{figure}

\begin{figure} \centerline{\includegraphics{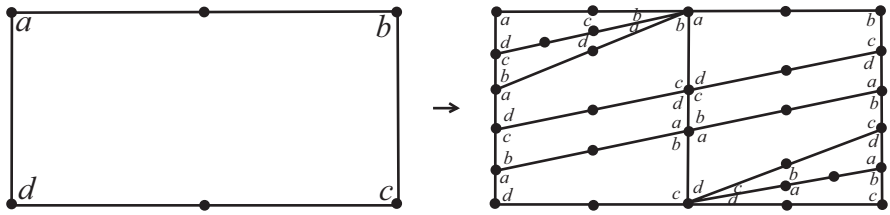}}
\caption{The subdivision of the tile type for the finite subdivision
rule $\cR$ of the main example.} \label{fig:fsubrule}
\end{figure}

\begin{figure}\begin{center}
\includegraphics{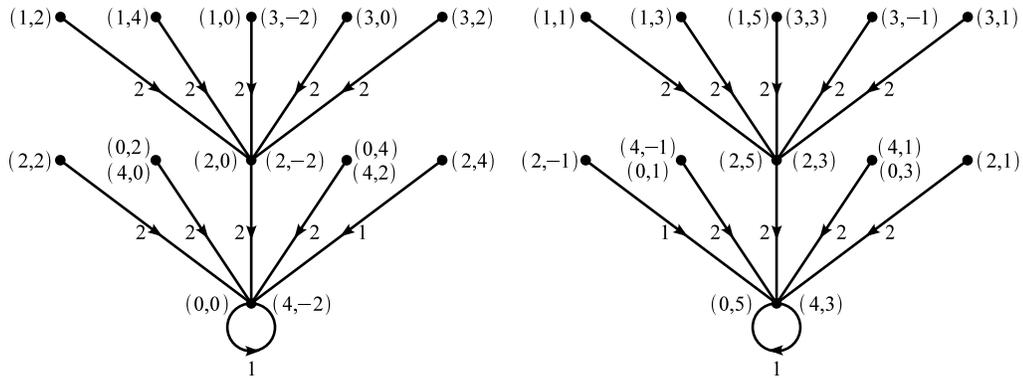} \caption{ The mapping
scheme for the main example.}
\label{fig:mainscheme}
\end{center}\end{figure}

\end{ex}

\begin{ex}\label{ex:BEKP} We begin this example by finding a
Latt\`{e}s map $g$ which as a Euclidean Thurston map has lattices
$\zL_2=\left<1,\frac{1+\sqrt{-3}}{2}\right>$, $\zL_1=\sqrt{-3}\zL_2$
and identification map induced by the linear automorphism given by
$\zF(z)=\sqrt{-3}z$.  Because $\zF$ is a conformal affine map, our
Riemann spheres $\mathbb{C}/\zG_1$ and $\mathbb{C}/\zG_2$ have the
same conformal structure.  Let $\zt=\frac{1+\sqrt{-3}}{2}$.  See
Figure~\ref{fig:BEKP}.  The parallelogram $F_2$ with vertices 0, 2,
$\zt$ and $2+\zt$ is a fundamental domain for the action of $\zG_2$.
The image $F_1$ of $F_2$ under $\zF$ is a fundamental domain for the
action of $\zG_1$.  Both $F_1$ and $F_2$ are shown in
Figure~\ref{fig:BEKP}.

\begin{figure}\begin{center}
\includegraphics{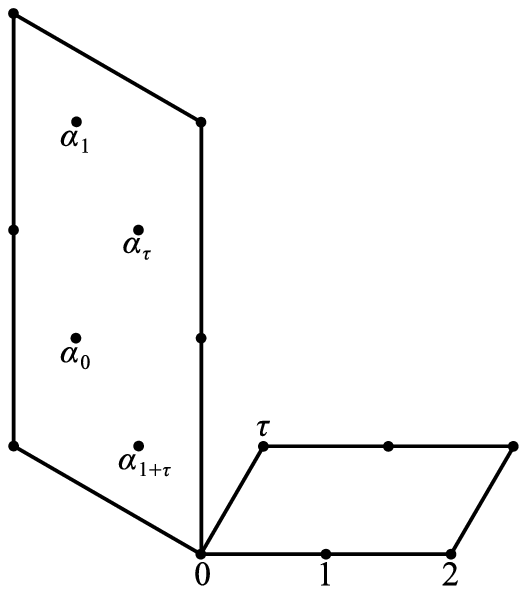} \caption{ Computing the Latt\`{e}s
map $g$.}
\label{fig:BEKP}
\end{center}\end{figure}

The matrix of $\zF$ with respect to the ordered basis $(1,\zt)$ of
$\bR^2$ is $\left[\begin{smallmatrix}-1 & -2 \\ 2 &
1 \end{smallmatrix}\right]$.  So $\zF(\za)\equiv \za\mod 2 \zL_2$ for
every $\za\in \zL_2$.  The four elements of $\zL_2$ in the interior of
$F_1$ are $\za_0=-2+2 \zt$, $\za_1=-3+4 \zt$, $\za_\zt=-2+3 \zt$ and
$\za_{1+\zt}=-1+\zt$.  The images of these four lattice points under
$p_1\circ q_1$ are the critical points of $g$.  One verifies that
$\za_\zl\equiv \zl\mod 2 \zL_2$ for every $\zl\in \{0,1,\zt,1+\zt\}$.
Using the fact that $\zF$ is the identity map modulo $2 \zL_2$, it
follows that $g(p_1(q_1(\za_\zl)))=p_2(q_2(\zl))$ for every $\zl\in
\{0,1,\zt,1+\zt\}$.

Let $\zw=\frac{-1+\sqrt{-3}}{2}$.  Because the map $z\mapsto \zw z$
stabilizes $\zL_1$, it determines an analytic homeomorphism from
$\widehat{\bC}=\bC/\zG_1$ to itself.  In the same way, the map
$z\mapsto \zw z$ determines an analytic homeomorphism from
$\widehat{\bC}=\bC/\zG_2$ to itself.  Because the map $z\mapsto \zw z$
commutes with the map $\zF$ which induces the identification map,
our two maps from $\widehat{\bC}$ to itself are equal.  Let $\zj\co
\widehat{\bC}\to \widehat{\bC}$ be this map, which is a M\"{o}bius
transformation.  The map $z\mapsto \zw z$ permutes the cosets of
$2\zL_1$ in $\zL_1$, and it permutes the cosets of $2\zL_1$ in
$\zL_2$.  Hence it permutes the cosets $\pm \za_\zl+2 \zL_1$ for
$\zl\in \{0,1,\zt,1+\zt\}$.  The congruence $\za_\zl\equiv \zl\mod 2
\zL_2$ implies that $\pm \za_\zl+2 \zL_1\subseteq \zl+2 \zL_2$ for
$\zl\in \{0,1,\zt,1+\zt\}$.  So the action of $z\mapsto \zw z$ on
these cosets of $2 \zL_1$ is the same as its action on these cosets of
$2 \zL_2$.  This gives the following congruences modulo $2 \zL_1$.
  \begin{equation*}
\zw\za_0\equiv \pm \za_0\quad \zw\za_1\equiv \pm \za_{1+\zt}\quad
\zw\za_{1+\zt}\equiv \pm \za_\zt\quad \zw\za_\zt\equiv \pm \za_1
  \end{equation*}
So $\zj$ is a M\"{o}bius transformation with order 3 which fixes
$p_1(q_1(0))$ and $p_1(q_1(\za_0))$ and cyclically permutes
$p_1(q_1(\za_1))$, $p_1(q_1(\za_{1+\zt}))$ and $p_1(q_1(\za_\zt))$.

We identify $\bC/\zG_1$ with $\widehat{\bC}$ so that $0\in \bC$ maps
to $\infty$, the point $\za_0$ maps to 0 and 1 maps to
$-\frac{1}{2}$.  Since $\zj$ is a M\"{o}bius transformation with order
3 which fixes 0 and $\infty$ and our identifications preserve
orientation, $\zj(z)=\zw z$.  Since the maps $z\mapsto \sqrt{-3}z$ and
$z\mapsto \zw z$ commute, so do $g$ and $\zj$.  Hence $g(\zw z)=\zw
g(z)$ for every $z\in \widehat{\bC}$.  Because $\sqrt{-3}\cdot 0=0$
and 0 maps to $\infty$ in $\widehat{\bC}$, we see that
$g(\infty)=\infty$.  Because $\za_0$ maps to $0\in \widehat{\bC}$, the
point 0 is a critical point of $g$ with $g(0)=\infty$.  Because 1 maps
to $-\frac{1}{2}\in \widehat{\bC}$, it follows that $-\frac{1}{2}$ is
fixed by $g$ and it is the image of a critical point of $g$.

Now we finally determine $g$.  Since the square of the modulus of
$\sqrt{-3}$ is 3, the map $g$ is a cubic rational function.  Since it
has poles at 0 and $\infty$ with 0 being a critical point, we may
assume that its denominator is $z^2$.  Because $g(\zw z)=\zw g(z)$, we
may assume that its numerator is $az^3+b$ for some $a,b\in \bC$:
$g(z)=\frac{az^3+b}{z^2}$.  Since $-\frac{1}{2}$ is fixed by $g$, the
polynomial $2az^3+z^2+2b$ has a root at $-\frac{1}{2}$:
$-\frac{a}{4}+\frac{1}{4}+2b=0$.  Hence $b=\frac{1}{8}(a-1)$ and
$2az^3+z^2+2b=(2z+1)(az^2+\frac{1}{2}(1-a)z+\frac{1}{4}(a-1))$.  Since
$-\frac{1}{2}$ is the image under $g$ of a critical point, the second
factor is a square, and so its discriminant is 0:
$0=\frac{1}{4}(a-1)^2-a(a-1)=-\frac{1}{4}(3a+1)(a-1)$.  If $a=1$, then
$b=0$, which is impossible.  Thus $a=-\frac{1}{3}$, $b=-\frac{1}{6}$
and $g(z)=-\frac{1}{6}\frac{2z^3+1}{z^2}$.

The critical points of $g$ are 0, 1, $\zw$ and $\overline{\zw}$.
These are mapped to $\infty$, $-\frac{1}{2}$, $-\frac{1}{2}\zw$ and
$-\frac{1}{2}\overline{\zw}$ by $g$ in order.  The map $h\co
\widehat{\bC}\to \widehat{\bC}$ defined by $h(z)=-\frac{1}{2z}$ maps
the latter four points to the former four points.  So
$f(z)=h(g(z))=\frac{3z^2}{2z^3+1}$ is a NET map which maps its set of
four critical points bijectively to itself.  Figure~\ref{fig:schemef}
shows the mapping scheme of $f$.  This is the rational function which
appears in the proof of statement 2 of Theorem 1.1 of \cite{BEKP}.

In \cite{L} Russell Lodge computes the slope function $\zs_f$ of $f$
introduced in Section~\ref{sec:slopefn}.  His methods are different
from those of Section~\ref{sec:slopefn}.  See
Remark~\ref{remark:virtual} for a bit more on this.

\end{ex}

  \begin{figure}\begin{center}
\includegraphics{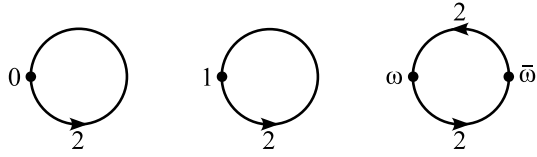} \caption{ The
mapping scheme of $f(z)=\frac{3z^2}{2z^3+1}$.}
\label{fig:schemef}\end{center}
  \end{figure}

\section{Pullbacks of simple closed curves}
\label{sec:pullbacks}\nosubsections

In this section we investigate pullbacks of simple closed curves under
NET maps.  We begin by reviewing well-known facts about simple closed
curves in tori and 4-punctured spheres.  A good reference for this
material is the book \cite{FM} by Farb and Margalit.  See Proposition
1.5 and Proposition 2.6 in \cite{FM}.

Let $f\co S^2\to S^2$ be a NET map.  We maintain the setting of
Section~\ref{sec:defns}.

Let $j\in \{1,2\}$.  Let $(\zl_j,\zm_j)$ be an ordered basis of
$\zL_j$.  Let $p$ and $q$ be relatively prime integers.  The universal
covering map $q_j$ maps every line in $\bR^2$ with parametrization of
the form $(x,y)=v+t(q \zl_j+p \zm_j)$ to a simple closed curve in
$T_j$, which is said to have slope $\frac{p}{q}\in
\widehat{\bQ}=\bQ\cup \{\infty\}$.  The resulting map from slopes to
curves establishes a bijective correspondence between $\widehat{\bQ}$
and the set of nontrivial homotopy classes of simple closed curves in
$T_j$.

A simple closed curve in $S^2\setminus P_j$ is peripheral if it is
homotopic to a very small closed curve around an element of $P_j$.  A
simple closed curve in $S^2\setminus P_j$ is essential if it is not
null homotopic.  If $\zg$ is an essential, nonperipheral simple closed
curve in $S^2\setminus P_j$, then $p_j^{-1}(\zg)$ consists of two
disjoint simple closed curves in $T_j$.  They are not null homotopic.
Being disjoint, they are homotopic to each other and hence have the
same slope.  This establishes a bijection between $\widehat{\bQ}$ and
the set of homotopy classes of essential, nonperipheral simple closed
curves in $S^2\setminus P_j$.  We must take care that this bijection
is very uncanonical.

Here is a slightly different point of view.  We use the fact that
$p_j\circ q_j\co \bR^2\setminus \zL_j\to S^2\setminus P_j$ is a
regular covering map with group of deck transformations $\zG_j$.  Let
$\za$ be an essential, nonperipheral simple closed curve in
$S^2\setminus P_j$.  Suppose that $\za$ has a lift to $\bR^2$ which
joins points $x$ and $y$.  Since $\za$ is not null homotopic, $x\ne
y$.  Because the deck transformations of $p_j\circ q_j$ are Euclidean
isometries and because this lift of $\za$ is also a lift of a closed
curve in $T_j$, we have that $y=\zg(x)$ for some translation $\zg$ in
$\zG_j$.  It follows that the slope of the line through $x$ and $y$ is
independent of the choice of lift of $\za$ to $\bR^2$.  The slope of
such a line relative to the ordered basis $(\zl_j,\zm_j)$ of $\zL_j$
is the slope of $\za$.

If $\zg$ is an essential, nonperipheral simple closed curve in
$S^2\setminus P_j$, then $\zg$ separates two points, $x$ and $y$, of
$P_j$ from the other two points of $P_j$.  We call an arc in $S^2$
joining $x$ and $y$ which is disjoint from $\zg$ a \emph{core arc}
for $\zg$.  Giving a homotopy class of essential, nonperipheral simple
closed curves in $S^2\setminus P_j$ is equivalent to giving such a
core arc.

In this paragraph we make a definition to prepare for the next
theorem.  Let $A$ be a finite Abelian group.  Let $H$ be a subset of
$A$ which is the disjoint union of four inverse pairs $\{\pm h_1\}$,
$\{\pm h_2\}$, $\{\pm h_3\}$ and $\{\pm h_4\}$.  (It is possible that
$h_i=-h_i$.)  Let $B$ be a subgroup of $A$ such that $A/B$ is cyclic,
and let $a$ be an element of $A$ whose image in $A/B$ generates $A/B$.
Let $n$ be the order of $A/B$.  For every $k\in \{1,2,3,4\}$ exactly
one coset $c a+B$ of $B$ in $A$ contains either $h_k$ or $-h_k$, where
$c$ is an integer with $0\le c\le n/2$.  Let $c_1$, $c_2$, $c_3$,
$c_4$ be these four integers ordered so that $c_1\le c_2\le c_3\le
c_4$.  (The integer $c_k$ need not correspond to $\pm h_k$.)  We call
$c_1$, $c_2$, $c_3$, $c_4$ the \emph{coset numbers} for $H$ relative
to $B$ and $a$ or relative to $B$ and the generator $a+B$ of $A/B$.
We are interested in coset numbers when $A=\zL_2/2\zL_1$, where
$\zL_1$ and $\zL_2$ are the lattices in Section~\ref{sec:defns}.  If
$\zl$ and $\zm$ form a basis of $\zL_2$, then the image of $\zl$ in
$A$ generates a cyclic subgroup $B$ and the image of $\zm$ in $A$ is
an element $a$ whose image in $A/B$ generates $A/B$.  Recall from the
end of Section~\ref{sec:defns} that $q_1^{-1}(p_1^{-1}(P_2))\subseteq
\zL_2$.  Thus we may speak of coset numbers for $H=p_1^{-1}(P_2)$
relative to $B$ and $a$.  We also call these coset numbers the coset
numbers for $q_1^{-1}(p_1^{-1}(P_2))$ relative to $\zl$ and $\zm$.
The coset number of $\zh\in q_1^{-1}(p_1^{-1}(P_2))$ is the smallest
nonnegative integer $c$ for which there exists an integer $b$ such
that $\pm \zh\in b \zl+c \zm+2 \zL_1$.

This takes us to the main result of this section.

\begin{thm}\label{thm:degscmpts} Let $f$ be a NET map in the setting
of Section~\ref{sec:defns}.  Let $\zd$ be an essential, nonperipheral
simple closed curve in $S^2\setminus P_2$ with slope $\frac{p}{q}$,
where $p$ and $q$ are relatively prime integers.  Let $\zl=q \zl_2+p
\zm_2\in \zL_2$.  Let $d$ be the order of the image of $\zl$ in
$\zL_2/\zL_1$.  Let $d'$ be the positive integer such that
$dd'=\left|\zL_2/\zL_1\right|=\deg(f)$.  Since $p$ and $q$ are
relatively prime, there exists $\zm\in \zL_2$ such that $\zl$ and
$\zm$ form another basis of $\zL_2$.  Let $c_1$, $c_2$, $c_3$, $c_4$
be the coset numbers for the elements of $q_1^{-1}(p_1^{-1}(P_2))$
relative to $\zl$ and $\zm$.  Then the following statements hold.
\begin{enumerate}
  \item Every connected component of $f^{-1}(\zd)$ maps to $\zd$ with
degree $d$.
  \item The number of essential, nonperipheral components in
$f^{-1}(\zd)$ is $c_3-c_2$.
  \item The number of peripheral components in $f^{-1}(\zd)$ is
$c_2-c_1+c_4-c_3$.
  \item The number of null homotopic components in $f^{-1}(\zd)$ is
$c_1-c_4+d'$.
  \item The lines in $\bR^2$ with slope $\frac{p}{q}$ relative to the
basis $(\zl_2,\zm_2)$ of $\zL_2$ which map under $p_1\circ q_1$ to
essential, nonperipheral simple closed curves in $S^2\setminus P_2$ are
exactly the $\zG_1$-translates of the lines with parametric forms
$(x,y)=t \zl+u \zm$ with parameter $t$ and $c_2<u<c_3$.
\end{enumerate}
\end{thm}
  \begin{proof} Because $\zd$ has slope $\frac{p}{q}$, it is homotopic
to the image in $S^2\setminus P_2$ under $p_2\circ q_2$ of a line in
$\bR^2$ with a parametrization of the form $(x,y)=v+t \zl$.  Without
loss of generality we assume that $\zd$ is the image in $S^2\setminus
P_2$ under $p_2\cdot q_2$ of such a line.  Because $p$ and $q$ are
relatively prime, the line segment joining $v$ and $v+2\zl$ maps
injectively to $T_2$ and hence to $S^2\setminus P_2$ except for its
endpoints.

\begin{figure}\begin{center}
\includegraphics{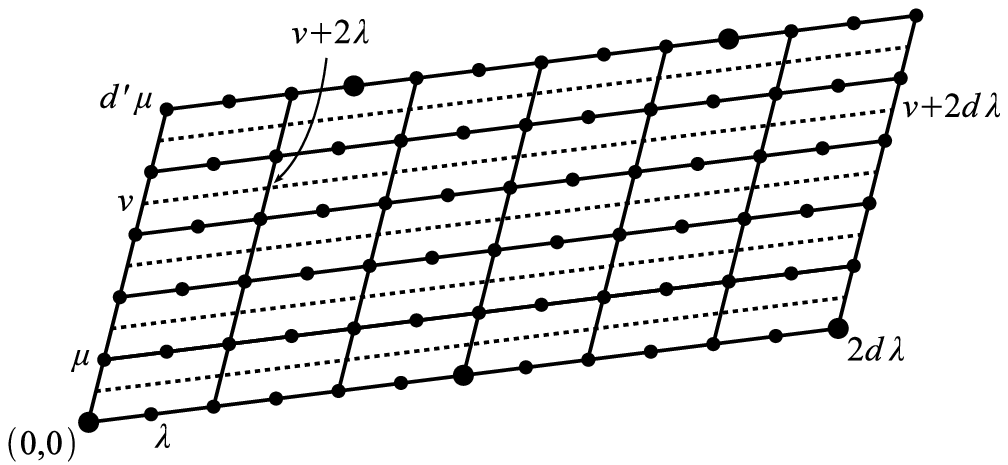}\caption{ Proving
Theorem~\ref{thm:degscmpts}.}
\label{fig:degscmpts}
\end{center}\end{figure}

Such a line segment is shown in Figure~\ref{fig:degscmpts} together
with some of its $\zG_2$-translates, drawn as dashed line segments.
The dots in Figure~\ref{fig:degscmpts} are elements of $\zL_2$ with
the larger ones being in $\zL_1$.  Each of the smallest parallelograms
bounded by solid line segments in Figure~\ref{fig:degscmpts} is a
fundamental domain for the action of $\zG_2$ on $\bR^2$.  The entire
parallelogram subdivided by these small parallelograms is a
fundamental domain $F$ for the action of $\zG_1$ on $\bR^2$.  As such,
it contains exactly one lift to $\bR^2$ under $p_1\circ q_1$ of every
connected component of $f^{-1}(\zd)$.

Because the lift of $f$ to $\bR^2$ is the identity map, the line
segment joining $v$ and $v+2\zl$ is a lift of $\zd$ in $F$ under the
map $f\circ p_1\circ q_1$.  The $\zG_2$-translates of this lift are
other lifts of $\zd$.  A concatenation of such line segments is the
lift of a closed curve in $S^2\setminus P_1$ if and only if the
difference between its endpoints is an element of $2 \zL _1$ which is
not a nontrivial multiple of an element of $2 \zL_1$.  In other words,
this difference is the smallest multiple of $2\zl$ which lies in
$2\zL_1$.  This is the order of the image of $2\zl$ in $2\zL_2/2\zL_1$,
which equals the order $d$ of the image of $\zl$ in $\zL_2/\zL_1$.
This proves the first statement of Theorem~\ref{thm:degscmpts}.

Figure~\ref{fig:degscmptsb} illustrates statements 2 through 4.  The
coset numbers $c_1$, $c_2$, $c_3$, $c_4$ determine a partition of the
line segment joining $(0,0)$ and $d' \zm$.  The components of
$f^{-1}(\zd)$ corresponding to the first and last of these subsegments
are null homotopic.  The components corresponding to the subintervals
adjacent to these are peripheral.  The remaining components are
essential and nonperipheral.

  \begin{figure}
\centerline{\includegraphics{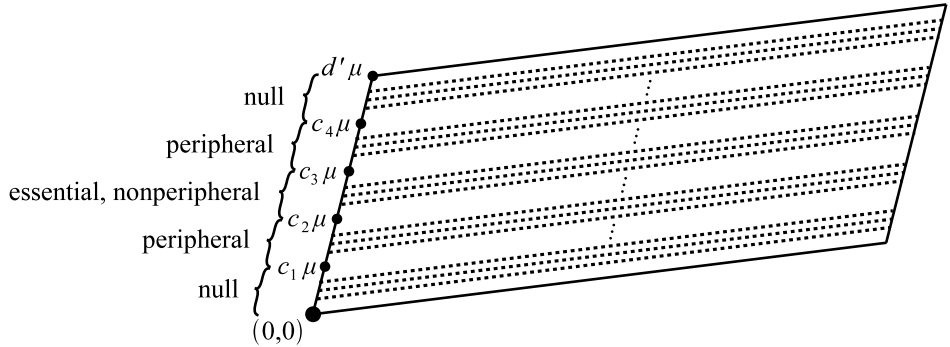}} \caption{
Illustrating Theorem~\ref{thm:degscmpts}.}
\label{fig:degscmptsb}
  \end{figure}

Now we prove these statements.  Since the restriction of $f$ to every
connected component of $f^{-1}(\zd)$ has degree $d$ and $d
d'=\deg(f)$, the number of these components is $d'$.  Let $\za$ be one
of the $d'$ line segments in $F$ which is a lift of a connected
component of $f^{-1}(\zd)$.  The action of $\zG_1$ on the boundary of
$F$ identifies two halves of the bottom of $F$ by a rotation of order
2.  In general the identification of the top of $F$ is only slightly
more complicated because the top two corners of $F$ are not
necessarily elements of $\zL_1$.  The two sides of $F$ are identified
by a translation.  Let $U$ and $V$ be the connected components of the
complement of $\za$ in $F$.  Both $U$ and $V$ map to open disks in
$S^2$ under $p_1\circ q_1$ and the image of $\za$ separates these two
disks.  So for the image of $\za$ to be essential and nonperipheral,
two elements of $U$ must map to distinct elements of $P_2$ and two
elements of $V$ must map to distinct elements of $P_2$.  Thus the
number of essential, nonperipheral components of $f^{-1}(\zd)$ is
$c_3-c_2$.  This proves statement 2.  Statements 3 and 4 can be proven
similarly.  Statement 5 is now clear.

This proves Theorem~\ref{thm:degscmpts}.

\end{proof}

The following lemma provides a way to compute the coset numbers $c_1$,
$c_2$, $c_3$, $c_4$ in Theorem~\ref{thm:degscmpts}.

\begin{lemma}\label{lemma:cmpuc}  Maintain the setting of
Theorem~\ref{thm:degscmpts}.  Let $r$ and $s$ be integers, and let
$\zh=r\zl_2+s\zm_2\in \zL_2$.  Let $b$ and $c$ be integers such that
$c\ge 0$ and $c$ is as small as possible such that $\pm \zh\in b \zl+c
\zm+2 \zL_1$.  Then $c$, the coset number of $\zh$ with respect to
$\zl$ and $\zm$, is the smallest nonnegative integer congruent to $\pm
(pr-qs)$ modulo $2d'$.
\end{lemma}
  \begin{proof} We begin by finding a useful basis of $\zL_1$.  Not
every element of $\zL_1$ is a multiple of $\zl$.  So there exist integers
$l$ and $m$ with $m>0$ such that $l \zl+m \zm\in \zL_1$.  Let $\zm'$
be an element of $\zL_1$ such that $\zm'=l \zl+m \zm$ with $m>0$ and
$m$ as small as possible.  We claim that $d \zl$ and $\zm'$ form a
basis of $\zL_1$.

To prove this, let $\zn\in \zL_1$.  It suffices to prove that $\zn$ is
an integral linear combination of $d \zl$ and $\zm'$.  There exist
integers $x$ and $y$ such that $\zn=x \zl+y \zm$.  Subtracting an
appropriate multiple of $\zm'$ from $\zn$ obtains an element $\zn'\in
\zL_1$ with $\zn'=x' \zl+y' \zm$ and $0\le y'<m$.  The choice of $m$
implies that $y'=0$.  Hence $\zn'=x' \zl$.  The choice of $d$ implies
that $d| x'$.  Thus every element of $\zL_1$ is an integral linear
combination of $d \zl$ and $\zm'$, and so $d \zl$ and $\zm'$ form a
basis of $\zL_1$.  Since the determinant of the matrix
$\left[\begin{smallmatrix}d & l \\ 0 & m \end{smallmatrix}\right]$ is
$\left|\zL_2/\zL_1\right|=\deg(f)$, it follows that $m=d'$.

Now let $x$ and $y$ be the integers such that $\zh=x \zl+y \zm$.  We
seek the nonnegative integer $c$ which is as small as possible such
that there exists an integer $b$ for which $\pm \zh\in b \zl+c \zm+2
\zL_1$.  Equivalently, $b \zl+c \zm=2\zl'\pm (x \zl+y \zm)$ for some
$\zl'\in \zL_1$.  The previous paragraph shows that $\zl'$ is an
integral linear combination of $d \zl$ and $\zm'=l \zl+d'\zm$.  It
follows that $c$ is the smallest nonnegative integer congruent to $\pm
y$ modulo $2 d'$.

So now we determine $y$.  Let $t$ and $u$ be the integers such that
$\zm=t \zl_2+u \zm_2$.  Using the fact that $\zh=r\zl_2+s\zm_2=x \zl+y
\zm$ and multilinearity of determinants, we see that
  \begin{equation*}
\left|\begin{matrix} q& r \\ p & s \end{matrix}\right|=
x\left|\begin{matrix} q& q \\ p & p \end{matrix}\right|+
y\left|\begin{matrix} q& t \\ p & u \end{matrix}\right|=\pm y,
  \end{equation*}
the last determinant being $\pm 1$ because $\zl$ and $\zm$ form a
basis of $\zL_2$.  So $y=\pm (pr-qs)$.  Therefore $c$ is the smallest
nonnegative integer congruent to $\pm (pr-qs)$ modulo $2d'$.

This proves Lemma~\ref{lemma:cmpuc}.

\end{proof}

\begin{remark}\label{remark:inttn}  The element $\zh$ in
Lemma~\ref{lemma:cmpuc} can be any element of $\zL_2$.  However, if
$r$ and $s$ are relatively prime, then we have the following
interpretation.  If $r$ and $s$ are relatively prime, then $\zh$
determines a simple closed curve $\zg$ in $T_2$ with slope
$\frac{s}{r}$.  Let $\widetilde{\zd}$ be a lift of $\zd$ to $T_2$, a
simple closed curve with slope $\frac{p}{q}$.  As in Section 1.2.3 of
\cite{FM} by Farb and Margalit, the intersection number
$\zi(\widetilde{\zd},\zg)$ is $\left|pr-qs\right|$.  So in this case
$c$ is the smallest nonnegative integer which is congruent to $\pm $
this intersection number modulo $2d'$.
\end{remark}

We continue with one more general observation about computations.
Lemma~\ref{lemma:cmpuc} implies that the coset numbers $c_ 1$, $c_2$,
$c_3$, $c_4$ depend in a simple way on $p$ and $q$ and that to
determine them it suffices to determine them modulo $2d'$.  So suppose
that the coset numbers $c_ 1$, $c_2$, $c_3$, $c_4$ arise from slope
$\frac{p}{q}$.  Let $\frac{p'}{q'}\in \widehat{\bQ}$ and suppose that,
just as for $\frac{p}{q}$, the order of the image of $q'\zl_2+p'\zm_2$
in $\zL_2/\zL_1$ is $d$.  Also suppose that there exists an integer
$u$ which is a unit modulo $2d'$ such that $p'\equiv up \mod 2d'$ and
$q'\equiv uq\mod 2d'$.  The point of this discussion is that if $c'_
1$, $c'_2$, $c'_3$, $c'_4$ are the coset numbers as in
Theorem~\ref{thm:degscmpts} for $\frac{p'}{q'}$, then $c'_1$, $c'_2$,
$c'_3$, $c'_4$ are congruent to $\pm uc_1$, $\pm uc_2$, $\pm uc_3$,
$\pm uc_4$ modulo $2d'$, not necessarily in order.

Now we apply Theorem~\ref{thm:degscmpts} to the main example.
Table~\ref{tab:mainex} displays the results.  The first column gives
$q$ modulo 20.  The second column gives $2p+q$ modulo 5.  The third
column gives the degree $d$ of the restriction of $f$ to every
connected component of the inverse image of an essential simple closed
curve in $S^2\setminus P_f$ with slope $\frac{p}{q}$.  The next column
gives the coset numbers $c_1$, $c_2$, $c_3$, $c_4$ which appear in
Theorem~\ref{thm:degscmpts}.  The last three columns give the numbers
of essential nonperipheral components, peripheral components and null
homotopic components in this inverse image.

\begin{table}
\begin{tabular}{|c|c|c|c|c|c|c|}\hline
$q$ mod 20 & $2p+q$ mod 5 & deg & $c_1$,$c_2$,$c_3$,$c_4$ &
essl & perl & null \\ \hline
$\pm 1$, $\pm 3$, $\pm 5$, $\pm 7$, $\pm 9$ & $\pm 1$, $\pm 2$ & 10 &
0, 0, 1, 1 & 1 & 0 & 0  \\ \hline
$\pm 1$, $\pm 9$ & 0 & 2 & 0, 1, 4, 5 & 3 & 2 & 0  \\ \hline
$\pm 3$, $\pm 7$  & 0 & 2 & 0, 2, 3, 5 & 1 & 4 & 0 \\ \hline
$0$, $\pm 2$, $\pm 4$, $\pm 6$, $\pm 8$, 10 & $\pm 1$, $\pm 2$ &
5 & 0, 0, 2, 2 & 2 & 0 & 0  \\ \hline
$\pm 2$ & 0 & 1 & 0, 2, 8, 10 & 6 & 4 & 0 \\ \hline
$\pm 4$ & 0 & 1 & 0, 0, 6, 6 & 6 & 0 & 4 \\ \hline
$\pm 6$ & 0 & 1 & 0, 4, 6, 10 & 2 & 8 & 0 \\ \hline
$\pm 8$ & 0 & 1 & 0, 0, 2, 2 & 2 & 0 & 8 \\ \hline
\end{tabular}
\smallskip
\caption{ Degrees, coset numbers and numbers of components for
the main example.}
\label{tab:mainex}
\end{table}

For the main example $\zL_2=\bZ^2$ and
$\zL_1=\left<(2,-1),(0,5)\right>$.  Let $p$ and $q$ be relatively
prime integers.  We begin with an essential simple closed curve $\zd$
in $S^2\setminus P_f$ with slope $\frac{p}{q}$.  Let $d$ be the degree
of the restriction of $f$ to any connected component of $f^{-1}(\zd)$.

Statement 1 of Theorem~\ref{thm:degscmpts} implies that $d$ is the
smallest positive integer such that there exist integers $x$ and $y$
for which $x(2,-1)+y(0,5)=d(q,p)$.  Solving for $x$ and $y$, we find
that
  \begin{equation*}
x=\frac{dq}{2}\quad\text{and}\quad y=\frac{d(2p+q)}{10}.
  \end{equation*}
Suppose that $q\equiv 0 \mod 2$.  Then $x$ is an integer for every $d$.
For $y$ to be an integer, we see that the only condition on $d$ is that
$d\equiv 0\mod 5$ if $2p+q\not\equiv 0\mod 5$.  Now suppose that
$q\not\equiv 0\mod 2$.  Considering $x$ shows that $d\equiv 0\mod 2$.
Considering $y$ shows that, as before, $d\equiv 0\mod 5$ if
$2p+q\not\equiv 0\mod 5$.  This leads to the values of $d$ given in
Table~\ref{tab:mainex}.

Now we determine the remaining entries of the table.  The elements
$(0,0)$, $(2,0)$, $(0,5)$ and $(2,5)$ of $\zL_2$ map to the four
elements of $P_2$ under $p_1\circ q_1$.  To apply
Lemma~\ref{lemma:cmpuc}, we calculate $pr-qs$ for these four elements
and obtain 0, $2p$, $-5q$ and $2p-5q$.  In what follows, we find the
reduced residues of $\pm 1$ times these values modulo $2d'$.  It is
then easy to determine the remaining entries in Table~\ref{tab:mainex}
using Theorem~\ref{thm:degscmpts}.

First suppose that $d=10$ and $d'=1$.  According to
Table~\ref{tab:mainex}, the integer $q$ is odd.  So reducing 0, $2p$,
$-5q$ and $2p-5q$ modulo 2 yields 0, 0, 1, 1.  This completes the
computation for $d=10$.

Next suppose that $d=5$ and $d'=2$.  Then $q$ is even and $p$ is odd.
Regardless of whether $q\equiv 0\mod 4$ or $q\equiv 2\mod 4$, our four
values reduce to 0, 0, 2, 2 modulo 4.  This completes the computation
for $d=5$.

Next suppose that $d=2$ and $d'=5$.  Table~\ref{tab:mainex} shows that
$q\not\equiv 0\mod 2$.  It also shows that $2p+q\equiv 0\mod 5$, and
so $q\not\equiv 0\mod 5$.  So $q$ is a unit modulo 10.  Suppose that
$q\equiv 1\mod 5$.  Then $p\equiv 2\mod 5$.  So up to a sign, our
values reduce to 0, 1, 4, 5 modulo 10.  According to the observation
after Lemma~\ref{lemma:cmpuc}, multiplying $q$ by a unit modulo 10
amounts to multiplying these four values by the same unit.  The units
modulo 10 are represented by $\pm 1$ and $\pm 3$.  Since
multiplication by $-1$ does nothing, we need only consider
multiplication by 3.  We obtain 0, 2, 3, 5.  This completes the
computation for $d=2$.

Finally, suppose that $d=1$ and $d'=10$.  Table~\ref{tab:mainex} shows
that $q\equiv 0\mod 2$ and $2p+q\equiv 0\mod 5$. Hence $p\not\equiv
0\mod 2$ and $q\not\equiv 0\mod 5$.  Up to multiplication by a unit,
either $q\equiv 2\mod 20$ or $q\equiv 4\mod 20$.  Suppose that
$q\equiv 2\mod 20$.  Then $p\equiv -1,9\mod 20$.  So up to a sign, our
four values are 0, 2, 8, 10.  If $q\equiv 4\mod 20$, then $p\equiv
3,13\mod 20$.  Now we obtain 0, 0, 6, 6.  The units modulo 20 are $\pm
1$, $\pm 3$, $\pm 7$, $\pm 9$.  Up to a sign, $\pm 1$ and $\pm 9$
stabilize both $\{0,2,8,10\}$ and $\{0,0,6,6\}$.  Multiplying by 3
yields $\{0,4,6,10\}$ and $\{0,0,2,2\}$.  This completes the
computation for $d=1$.

\section{The slope function }\label{sec:slopefn}\nosubsections

Let $f$ be a NET map in the setting of Section~\ref{sec:defns}.  We
let $o$ denote the union of the classes of inessential and peripheral
curves in $S^2\setminus P_2$, and we define a \emph{slope function}
$\zs_f\co \widehat{\bQ}\to \widehat{\bQ}\cup \{o\}$ as follows.  As at
the beginning of Section~\ref{sec:pullbacks}, we fix an ordered basis
$(\zl_2,\zm_2)$ of $\zL_2$ by which we define slopes of essential,
nonperipheral simple closed curves in $S^2\setminus P_2$.  Let $p$ and
$q$ be relatively prime integers, so that $\frac{p}{q}\in
\widehat{\bQ}$.  Let $\zd$ be an essential, nonperipheral simple
closed curve in $S^2\setminus P_2$ with slope $\frac{p}{q}$.  Every
connected component of $f^{-1}(\zd)$ is in $S^2\setminus P_2$.  If no
connected component of $f^{-1}(\zd)$ is essential and nonperipheral in
$S^2\setminus P_2$, then set $\zs_f(\frac{p}{q})=o$.  Suppose that
some connected component $\za$ of $f^{-1}(\zd)$ is essential and
nonperipheral in $S^2\setminus P_2$.  In this case we let
$\zs_f(\frac{p}{q})$ be the slope of $\za$ in $S^2\setminus P_2$.
This defines $\zs_f$, independent of the choices of $\zd$ and $\za$.
The main goal of this section is to describe a method to compute
$\zs_f$.

According to statement 2 of Theorem~\ref{thm:cmpon}, it is possible to
factor $f$ as a composition $f=h\circ g$ of functions, where $g\co
S^2\to S^2$ is a Euclidean Thurston map and $h\co S^2\to S^2$ is any
orientation-preserving homeomorphism such that $h(P_1)=P_2$.  We
choose $h$ so that $h$ fixes $P_1\cap P_2$.

In this paragraph we construct four arcs in $S^2$ and their inverse
images in $\bR^2$.  Suppose that $P_1=\{x_1,x_2,x_3,x_4\}$.  For every
$k\in \{1,2,3,4\}$ let $\zb_k$ be an arc in $S^2$ which joins $x_k$
and $h(x_k)$.  Because $h$ fixes $P_1\cap P_2$, we may choose these
arcs so that they are disjoint.  Every connected component of
$q^{-1}_j(p^{-1}_j(\zb_k))$ contains exactly one element of $\zL_j$
for $j\in \{1,2\}$.  If $\zb_k$ is nontrivial, then the restriction of
$p_j\circ q_j$ to such a component is a branched covering map onto
$\zb_k$ with degree 2.  We call every such connected component a
\emph{spin mirror} for $p_j\circ q_j$.  This terminology will be
explained soon.  We emphasize that these spin mirrors depend on the
choice of four arcs in $S^2$ and these arcs depend in turn on the
homeomorphism $h$.  We furthermore assume that every spin mirror for
$p_1\circ q_1$ is a piecewise linear arc in $\bR^2$.

Let $B=\zb_1\cup \zb_2\cup \zb_3\cup \zb_4$, and let
$B_j=q_j^{-1}(p_j^{-1}(B))$ for $j\in \{1,2\}$.  Recall that $p_j\circ
q_j\co \bR^2\to S^2$ is a branched covering map ramified at exactly
the points of $\zL_j$ with local degree 2 at every such point.  Hence
the restriction of $p_1\circ q_1$ to $\bR^2\setminus B_1$ is a
covering map onto $S^2\setminus B$ which is equivalent to the covering
map obtained by restricting $p_2\circ q_2$ to $\bR^2\setminus B_2$.
This means that there exists a homeomorphism $ \zw \co \bR^2\setminus
B_1\to \bR^2\setminus B_2$ such that $p_2\circ q_2\circ \zw=p_1\circ
q_1$ as functions from $\bR^2\setminus B_1$ to $S^2\setminus B$.

In general the map $\zw$ does not extend to a continuous map from
$\bR^2$ to $\bR^2$, but nonetheless it does determine a bijection from
the set of spin mirrors in $B_1$ to the set of spin mirrors in $B_2$.
Suppose that $k\in \{1,2,3,4\}$ such that $\zb_k$ is nontrivial.  Let
$M_1$ be a spin mirror in $B_1$ which maps to $\zb_k$, and let $M_2$
be the corresponding spin mirror in $B_2$.  Both $M_1$ and $M_2$ are
branched double covering spaces of $\zb_k$ with $M_1$ branched over
the point of $P_1$ in $\zb_k$ and $M_2$ branched over the point of
$P_2$ in $\zb_k$.  Figure~\ref{fig:phi} illustrates the behavior of
$\zw$ near $M_1$.  Figure~\ref{fig:phi} is an idealized drawing.  The
spin mirrors need not be line segments, and $\zw$ need not be
piecewise linear.  We might think in terms of cutting $\bR^2$ open
along $M_1$.  We obtain a hole bounded by four arcs as in the middle
of Figure~\ref{fig:phi}.  The inverse operation is to identify two
pairs of adjacent arcs in the boundary of this hole.  To obtain $M_2$,
we identify the other two pairs of adjacent arcs.  We might imagine a
photon traveling through $\bR^2$ and crossing $M_2$ as indicated by
the dashed line segment in the rightmost part of Figure~\ref{fig:phi}.
The photon's inverse image under $\zw$ is indicated by the two dashed
line segments in the leftmost part of Figure~\ref{fig:phi}.  When the
photon's inverse image strikes $M_1$, it spins about the center of
$M_1$ and thereby reverses direction.  This property gives spin
mirrors their name.

\begin{figure}\begin{center}
\includegraphics{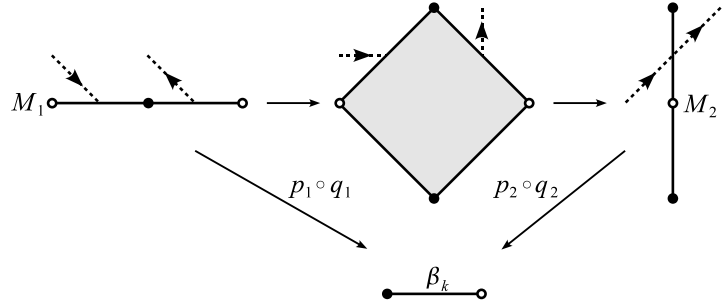} \caption{ The map $\zw$ near the
spin mirror $M_1$.}
\label{fig:phi}
\end{center}\end{figure}

Now we consider computing the slope function $\zs_f$.  Let
$\frac{p}{q}\in \widehat{\bQ}$.  Let $L$ be a line in $\bR^2\setminus
\zL_2$ with slope $\frac{p}{q}$ relative to the ordered basis
$(\zl_2,\zm_2)$ of $\zL_2$.  Then $\zd=p_2(q_2(L))$ is an essential,
nonperipheral simple closed curve in $S^2\setminus P_2$ with slope
$\frac{p}{q}$.  Theorem~\ref{thm:degscmpts} provides the means to
determine whether or not $f^{-1}(\zd)$ contains an essential,
nonperipheral component, that is, it allows us to determine whether or
not $\zs_f(\frac{p}{q})=o$.  So suppose that $f^{-1}(\zd)$
contains an essential, nonperipheral component $\za$.  Because the
lift of $f$ to $\bR^2$ is the identity map, we may, and do, assume
that $L$ is one connected component of $q_1^{-1}(p_1^{-1}(\za))$.  Let
$\zl=q \zl_2+p \zm_2$.  Let $d$ be the degree with which $f$ maps
$\za$ to $\zd$.  Theorem~\ref{thm:degscmpts} shows that $d$ is the
order of the image of $\zl$ in $\zL_2/\zL_1$.  Choose $v\in L$ so that
$v$ is not contained in a spin mirror for $p_1\circ q_1$.  Then the
segment of $L$ joining $v$ and $v+2 \zl$ is a lift to $\bR^2$ of $\zd$
under $f\circ p_1\circ q_1$ and the segment $S$ of $L$ joining $v$ and
$w=v+2d \zl$ is a lift to $\bR^2$ of $\za$ under $p_1\circ q_1$.
Recall that every spin mirror for $p_1\circ q_1$ is a piecewise linear
arc.  We choose $L$ so that its intersection with every such spin
mirror is transverse.  So $S$ meets $B_1$, the union of the spin
mirrors for $p_1\circ q_1$, transversely in finitely many points.

Let $S'$ be the lift to $\bR^2$ of $\za$ under $p_2\circ q_2$ based at
$\zw(v)$.  Suppose that $S$ meets $B_1$ in $n$ points.  Let
$S_1,\dotsc,S_{n+1}$ be the line segments in order from $v$ to $w=v+2d
\zl$ so that $S_1\cup \cdots\cup S_{n+1}=S\setminus B_1$.  Let
$S'_1,\dotsc, S'_{n+1}$ be the corresponding arcs in $S'$.  For every
$j\in\{1,\dotsc,n\}$ the closures of $S_j$ and $S_{j+1}$ meet at a
spin mirror.  Let $\zl_j\in \zL_1$ be the midpoint of this spin
mirror.

Standard covering space theory implies that $S'_1=\zw(S_1)$.  The
discussion which explains the naming of spin mirrors shows that
$S'_2=\zw(2 \zl_1-S_2)$.  Next, $S'_3=\zw(2 \zl_1-(2 \zl_2-S_3))=\zw(2
\zl_1-2 \zl_2+S_3)$.  Inductively, it follows that
  \begin{equation*}
S'_j=\zw\left((-1)^{j+1}S_j+2\sum_{i=1}^{j-1}(-1)^{i+1}\zl_i\right)
  \end{equation*}
for $j\in \{1,\dotsc,n+1\}$.  Set
$w'=(-1)^nw+2\sum_{i=1}^{n}(-1)^{i+1}\zl_i$.  It follows that
$S'$ joins $\zw(v)$ and $\zw(w')$.  So $\zs_f(\frac{p}{q})$ is the
slope of the line segment joining $\zw(v)$ and $\zw(w')$ relative to
the ordered basis $(\zl_2,\zm_2)$ of $\zL_2$.

It remains to interpret this in terms of the line segment joining $v$
and $w'$.  We use the fact that the restriction of $p_j\circ q_j$ to
$\bR^2\setminus B_j$ is a regular covering map with group of deck
transformations $\zG_j$ for $j\in \{1,2\}$.  The map $\zw$ induces a
group isomorphism from $\zG_1$ to $\zG_2$, hence a group isomorphism
from $2 \zL_1$ to $2 \zL_2$ and hence a group isomorphism from $\zL_1$
to $\zL_2$.  The map $\zw$ is not uniquely determined by the choice of
spin mirrors, but it is unique up to postcomposing with an element of
$\zG_2$.  So the isomorphism from $\zG_1$ to $\zG_2$ is unique up to
conjugation by an element of $\zG_2$.  One checks that the isomorphism
from $\zL_1$ to $\zL_2$ is therefore unique up to multiplication by
$\pm 1$.  This does not affect slopes.  So the choice of ordered basis
$(\zl_2,\zm_2)$ of $\zL_2$ together with the choice of spin mirrors
determines two ordered bases of $\zL_1$ of the form $(\zl_1,\zm_1)$
and $(-\zl_1,-\zm_1)$.  Then $\zs_f(\frac{p}{q})$ is the slope of the
line segment joining $v$ and $w'$ relative to either of these ordered
bases of $\zL_1$.  We emphasize that the correspondence between this
basis of $\zL_2$ and these two bases of $\zL_1$ involves both the
identification map $\zf\co \mathbb{R}^2/\zG_2\to \mathbb{R}^2/\zG_1$
and the choice of spin mirrors.

We have proved the following theorem.

\begin{thm}\label{thm:slopefn1} Let $f$ be a NET map in the setting of
Section~\ref{sec:defns}.  Let $\frac{p}{q}\in \widehat{\bQ}$.  Let
$\zd$ be an essential simple closed curve in $S^2\setminus P_2$ with
slope $\frac{p}{q}$ relative to the basis $(\zl_2,\zm_2)$ of $\zL_2$.
Suppose that $\za$ is an essential, nonperipheral component of
$f^{-1}(\zd)$ in $S^2\setminus P_2$.  Let $d$ be the degree with which
$f$ maps $\za$ to $\zd$.  Let $\zl=q \zl_2+p \zm_2$.  Let $v$ be any
point in $\bR^2$ such that $p_1(q_1(v))\in \za$ and $v$ is not
contained in a spin mirror for $p_1\circ q_1$.  It is possible to
choose $\zd$ so that the line segment joining $v$ and $w=v+2d \zl$ is
a lift of $\za$ to $\bR^2$ under $p_1\circ q_1$ and it intersects the
spin mirrors for $p_1\circ q_1$ transversely in finitely many points.
Let $S$ be the line segment joining $v$ and $w$.  Let
$\zl_1,\dotsc,\zl_n$ be the midpoints of the spin mirrors which meet
$S$ in order.  Then $\zs_f(\frac{p}{q})$ is the slope of the line
segment joining $v$ and $w'=(-1)^nw+2\sum_{i=1}^{n}(-1)^{i+1}\zl_i$
relative to either of the two ordered bases of $\zL_1$ determined by
$(\zl_2,\zm_2)$ and the choice of spin mirrors.
\end{thm}

\begin{remark}\label{remark:divby2}  If the spin mirrors for
$p_1\circ q_1$ are invariant under translation by the elements of
$\zL_1$, as in the main example, then the element $w$ in
Theorem~\ref{thm:slopefn1} may be taken to be $v+d \zl$ instead of
$v+2d \zl$.
\end{remark}

\begin{figure}\begin{center}
\includegraphics{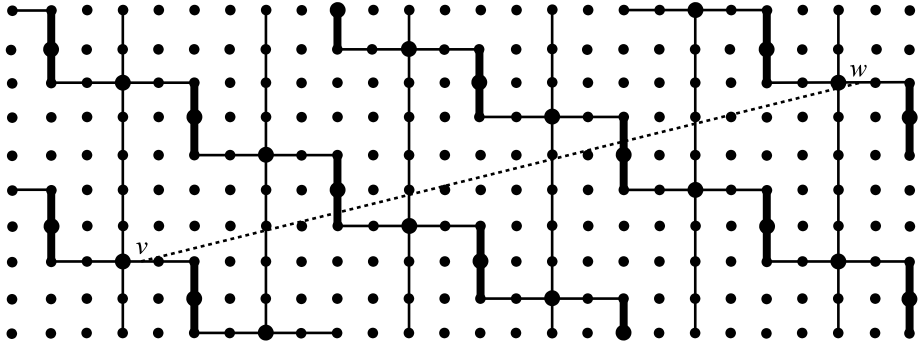}
\caption{ Spin mirrors for $p_1\circ q_1$ of the main
example.}  \label{fig:mainmirs}
\end{center}\end{figure}

\begin{figure}\begin{center}
\includegraphics{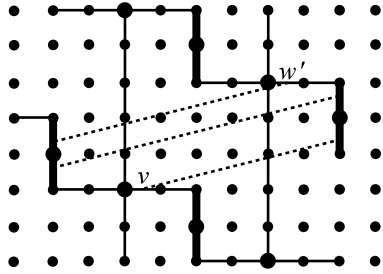} \caption{ Illustrating Theorem~\ref{thm:slopefn1}.}
\label{fig:mainzigzag}
\end{center}\end{figure}

Figures~\ref{fig:mainmirs} and \ref{fig:mainzigzag} illustrate
Theorem~\ref{thm:slopefn1} for the main example.
Figure~\ref{fig:mainmirs} shows the tiling of $\bR^2$ by the
$\zG_1$-translates of the fundamental domain for $\zG_1$ shown in
Figure~\ref{fig:maintess}.  The spin mirrors for $p_1\circ q_1$ are
drawn with thick line segments.  We take the standard basis
$\zl_2=(1,0)$ and $\zm_2=(0,1)$ for $\zL_2=\bZ^2$.  We choose
$\frac{p}{q}=\frac{1}{4}$.  Hence $\zl=(4,1)$, and we may take
$\zm=(1,0)$.  Table~\ref{tab:mainex} shows for this slope that $d=5$
and $c_1=c_2=0$.  Thus we may take $v=\frac{1}{2}\zm
=(\frac{1}{2},0)$.  By Remark~\ref{remark:divby2} we may take $w=v+d
\zl=(\frac{41}{2},5)$.  The dashed line segment in
Figure~\ref{fig:mainmirs} joins $v$ and $w$ and has slope
$\frac{1}{4}$.  It meets two spin mirrors.  The resulting spin
reflections are shown in Figure~\ref{fig:mainzigzag}.  Since the
ordered basis of $\zL_1$ consisting of $(2,-1)$ and $(0,5)$
corresponds to the ordered basis of $\zL_2$ consisting of $(1,0)$ and
$(0,1)$, it follows that $\zs_f(\frac{1}{4})$ is the slope of the line
through $v=(\frac{1}{2},0)$ and $w'=(\frac{9}{2},3)$ relative to the
basis $(2,-1)$ and $(0,5)$ of $\zL_1$.  Thus
$\zs_f(\frac{1}{4})=\frac{1}{2}$.

Theorem~\ref{thm:slopefn1} provides a way to compute the slope
function as in the previous paragraph, but the method leaves something
to be desired.  Although it might not be obvious, the next theorem
provides an improvement.

\begin{thm}\label{thm:slopefn2} Let $f$ be a NET map in the setting of
Section~\ref{sec:defns}.  Let $\frac{p}{q}\in \widehat{\bQ}$.  As in
Theorem~\ref{thm:slopefn1}, let $\zl=q \zl_2+p \zm_2$ and let $\zm$ be
an element of $\zL_2$ such that $\zl$ and $\zm$ form a basis of
$\zL_2$.  Also let $c_1$, $c_2$, $c_3$, $c_4$ be the coset numbers for
$q_1^{-1}(p_1^{-1}(P_2))$ relative to $\zl$ and $\zm$.  We assume that
$\zs_f(\frac{p}{q})\ne o$, equivalently, $c_2\ne c_3$ by
Theorem~\ref{thm:degscmpts}.  Let $L$ be a line in $\bR^2$ which has a
$\zG_1$-translate given in parametric form by either $(x,y)=t \zl+c_2
\zm$ or $(x,y)=t \zl+c_3 \zm$.  Let $v$ and $w$ be distinct elements
of $L\cap q_1^{-1}(p_1^{-1}(P_2))$ such that no element of
$q_1^{-1}(p_1^{-1}(P_1\cup P_2))$ is strictly between $v$ and $w$.
Let $S$ be the closed line segment which joins $v$ and $w$.  We assume
that the interior of $S$ intersects the spin mirrors for $p_1\circ
q_1$ transversely in finitely many points.  Let $\zl_1,\dotsc,\zl_n$
be the midpoints of these spin mirrors which meet the interior of $S$
in order.  Since $v,w\in q_1^{-1}(p_1^{-1}(P_2))$, both $v$ and $w$
are contained in spin mirrors for $p_1\circ q_1$.  Let $\zl_0$ and
$\zl_{n+1}$ be the midpoints of these two spin mirrors.  Then
$\zs_f(\frac{p}{q})$ is the slope of the line segment joining 0 and
$\sum_{i=0}^{n}(-1)^i(\zl_{i+1}-\zl_i)$ relative to either of the two
ordered bases of $\zL_1$ determined by $(\zl_2,\zm_2)$ and the choice
of spin mirrors.
\end{thm}
  \begin{proof} It suffices to prove the theorem for the case in which
$L$ is given in parametric form by either $(x,y)=t \zl+c_2 \zm$ or
$(x,y)=t \zl+c_3 \zm$, and so we assume that $L$ has this form.  Let
$\ze$ be a positive real number.  If $L$ is given in parametric form
by $(x,y)=t \zl+c_2 \zm$, then let $L_\ze$ be the line with parametric
form $(x,y)=t \zl+(c_2+\ze)\zm$.  If $L$ is given in parametric form
by $(x,y)=t \zl+c_3 \zm$, then let $L_\ze$ be the line with parametric
form $(x,y)=t \zl+(c_3-\ze)\zm$.  Statement 5 of
Theorem~\ref{thm:degscmpts} shows that if $\ze$ is small enough, then
$p_1(q_1(L_\ze))$ is an essential, nonperipheral simple closed curve
in $S^2\setminus P_2$.  Of course, $f$ maps $p_1(q_1(L_\ze))$ to an
essential simple closed curve in $S^2\setminus P_2$ with slope
$\frac{p}{q}$.  So $\zs_f(\frac{p}{q})$ is the slope of
$p_1(q_1(L_\ze))$ in $S^2\setminus P_2$.

The assumptions imply that the interior of $p_1(q_1(S))$ avoids the
branch points of $p_2\circ q_2$ and maps injectively to $S^2\setminus
P_2$.  Let $S'$ be a lift of $p_1(q_1(S))$ to $\bR^2$ under $p_2\circ
q_2$.  It follows that there exists $v_\ze\in L_\ze$ near $v$ and
$w_\ze\in L_\ze$ near $w$ such that $p_1(q_1(v_\ze))$ and
$p_1(q_1(w_\ze))$ lift under $p_2\circ q_2$ to points near the
endpoints of $S'$ and these lifts differ by the same nontrivial
element of $\zL_2$.  Expressing this element of $\zL_2$ in terms of
the basis $(\zl_2,\zm_2)$ determines $\zs_f(\frac{p}{q})$.  If $\ze$
is small enough, then computing this element of $\zL_2$ from $S$,
which joins $v$ and $w$, is the same as computing this element of
$\zL_2$ from the segment of $L_\ze$ which joins $v_\ze$ and $w_\ze$.

The discussion before Theorem~\ref{thm:slopefn1} now essentially
completes the proof of Theorem~\ref{thm:slopefn2}.  The only
difference now is that the endpoints $v$ and $w$ of $S$ lie in spin
mirrors for $p_1\circ q_1$.  Computing the relevant element of $\zL_1$
by means of the isomorphism between $\zL_1$ and $\zL_2$ requires
replacing $v$ and $w$ by the midpoints of the spin mirrors which
contain them.  Hence $\zs_f(\frac{p}{q})$ is the slope of the line
segment joining $\zl_0$ and
$(-1)^n \zl_{n+1}+2\sum_{i=1}^{n}(-1)^{n+1}\zl_i$ relative to either of
the two ordered bases of $\zL_1$ determined by $(\zl_2,\zm_2)$ and the
choice of spin mirrors.  This is equivalent to the desired conclusion.

This proves Theorem~\ref{thm:slopefn2}.

\end{proof}

One advantage of Theorem~\ref{thm:slopefn2} over
Theorem~\ref{thm:slopefn1} is that in Theorem~\ref{thm:slopefn2} both
$v$ and $w$ are in $\zL_2$, whereas in Theorem~\ref{thm:slopefn1}
neither is.  Another advantage is that in Theorem~\ref{thm:slopefn2}
the line segment joining $v$ and $w$ is shorter than the one in
Theorem~\ref{thm:slopefn1}, resulting in a shorter computation.

We next show for the main example that Theorem~\ref{thm:slopefn2}
provides an algorithm for computing the slope function which is easy
to implement by computer and which can even be used by hand in simple
cases.  This will occupy the rest of this section.

\begin{table}
\begin{tabular}{|c|c|c|c|}\hline
$q$ mod 4 & $2p+q$ mod 5 & $v$ & $w$ \\ \hline
0 & 0 & $(0,0)$ & $(q,p)$  \\ \hline
0 & $\pm 1$, $\pm 2$ & $(0,0)$ & $(5q,5p)$  \\ \hline
2 & 0 & $(2,0)$ & $(2+2q,2p)$  \\ \hline
2 & $\pm 1$ & $(0,0)$ & $(3q,3p)$  \\ \hline
2 & $\pm 2$ & $(0,0)$ & $(q,p)$  \\ \hline
$\pm 1$ & 0 & $(2,0)$ & $(2+4q,4p)$  \\ \hline
$\pm 1$ & $\pm 1$ & $(0,0)$ & $(2q,2p)$  \\ \hline
$\pm 1$ & $\pm 2$ & $(0,0)$ & $(6q,6p)$  \\ \hline
\end{tabular}
\smallskip
\caption{ Choosing $v$ and $w$ for the main example.}
\label{tab:mainvw}
\end{table}

Let $p$ and $q$ be relatively prime integers.  We wish to compute
$\zs_f(\frac{p}{q})$ for the main example using
Theorem~\ref{thm:slopefn2}.  Table~\ref{tab:mainex} shows that $c_2\ne
c_3$, so $\zs_f(\frac{p}{q})\ne o$.  We begin by choosing appropriate
lattice points $v$ and $w$ as in Theorem~\ref{thm:slopefn2}.  Since
$\zL_2=\bZ^2$ for the main example, $v$ and $w$ are simply standard
lattice points in $\bR^2$.  Table~\ref{tab:mainvw} gives our choices.
The first column gives $q$ modulo 4.  The second column gives $2p+q$
modulo 5.  The last two columns give $v$ and $w$.

We justify our choices of $v$ and $w$ beginning with this paragraph.
Our basis for $\zL_2$ is the standard basis $\zl_2=(1,0)$ and
$\zm_2=(0,1)$.  Hence $\zl=(q,p)$.  Moreover,
$\zL_1=\left<(2,-1),(0,5)\right>$.  The set $q_1^{-1}(p_1^{-1}(P_1\cup
P_2))$ is a union of cosets of $2 \zL_1$ in $\zL_2$, and the following
elements are distinct representatives for these cosets.
  \begin{equation*}
(0,0), (0,5), (2,-1), (2,4), (2,0), (2,-2), (2,3), (2,5)
  \end{equation*}
The first two coset representatives are in $q_1^{-1}(p_1^{-1}(P_1\cap
P_2))$, the next two are in $q_1^{-1}(p_1^{-1}(P_1\setminus
P_2))$ and the last four are in $q_1^{-1}(p_1^{-1}(P_2\setminus
P_1))$.

We first determine all cases in which it is possible to choose
$v=(0,0)$.  We see that $(0,0)\in q_1^{-1}(p_1^{-1}(P_2))$, as
required by Theorem~\ref{thm:slopefn2}.  We also need $(0,0)$ to be in
the line $L$ of Theorem~\ref{thm:slopefn2}.  This is equivalent to the
condition that $c_2=0$.  Table~\ref{tab:mainex} shows that this is in
turn equivalent to the condition that $q\equiv 0 \mod 4$ if
$2p+q\equiv 0\mod 5$.

With $v=(0,0)$ the element $w$ is an integer multiple of $\zl=(q,p)$,
and without loss of generality we take this integer $x$ to be
positive.  We want $w$ to be in $q_1^{-1}(p_1^{-1}(P_2))$ with no
element of $q_1^{-1}(p_1^{-1}(P_1\cup P_2))$ between $v$ and $w$.  Let
$(r,s)$ be one of our eight coset representatives.  We are interested
in the congruence $(r,s)\equiv x(q,p)\mod 2 \zL_1$.  Hence we are
interested in integers $y$ and $z$ such that
  \begin{equation*}
(r,s)=x(q,p)+y(4,-2)+z(0,10).
  \end{equation*}
The following equations give $y$ and $z$ as rational numbers.
  \begin{equation*}
y=\frac{1}{4}(r-xq)\quad\text{and}\quad z=\frac{1}{20}(2s+r-x(2p+q))
  \end{equation*}
Thus $y$ and $z$ are integers if and only if the following three
congruences are satisfied.
  \begin{equation*}\linnum\label{lin:maincnges}
r\equiv xq\mod 4\quad s\equiv xp\mod 2\quad 2s+r\equiv x(2p+q)\mod 5
  \end{equation*}
So, assuming that $q\equiv 0\mod 4$ if $2p+q\equiv 0\mod 5$, then we
may take $v=(0,0)$ and $w=x(q,p)$, where $x$ is the smallest positive
integer which satisfies line~\ref{lin:maincnges} for some choice of
$(r,s)$.

First suppose that $q\equiv 0\mod 4$.  Line~\ref{lin:maincnges}
implies that $r\equiv 0\mod 4$.  Thus either $(r,s)=(0,0)\in \zL_1$ or
$(r,s)=(0,5)\in \zL_1$.  It follows that $x$ is the order of the image
of $\zl=(q,p)$ in $\zL_2/\zL_1$.  Theorem~\ref{thm:degscmpts} and
Table~\ref{tab:mainex} now imply that $x=1$ if $2p+q\equiv 0\mod 5$
and $x=5$ if $2p+q\equiv \pm 1,\pm 2\mod 5$.  This gives the first two
lines of Table~\ref{tab:mainvw}.

Next suppose that $q\equiv 2\mod 4$.  Then $p\equiv 1\mod 2$.  We
consider solutions to line~\ref{lin:maincnges} with $x=1$.  There is
such a solution if and only if $r\equiv 2\mod 4$, $s\equiv 1\mod 2$
and $2s+r\equiv 2p+q\mod 5$.  These congruences have a solution with
$(r,s)\in q_1^{-1}(p_1^{-1}(P_2))$ if and only if $(r,s)\in
\{(2,3),(2,5)\}$ and $2p+q\equiv \pm 2\mod 5$.  This obtains line 5 of
Table~\ref{tab:mainvw}.  If $x=2$, then line~\ref{lin:maincnges} shows
that $r\equiv 0\mod 4$ and $s\equiv 0\mod 2$.  Hence $(r,s)=(0,0)$.
Hence $2(q,p)\in 2 \zL_1$, hence $(q,p)\in \zL_1$ and so
line~\ref{lin:maincnges} has a solution with $x=1$.  Thus there is no
acceptable value of $w$ with $x=2$.  Finally, we verify that
line~\ref{lin:maincnges} always has a solution for $x=3$, $q\equiv
2\mod 4$ and $2p+q\equiv \pm 1\mod 5$ by taking $(r,s)\in
\{(2,3),(2,5)\}$.  This obtains line 4 of Table~\ref{tab:mainvw}.

Next suppose that $q\equiv \pm 1\mod 4$.  If $x$ is odd, then the first
congruence in line~\ref{lin:maincnges} shows that $r$ is also odd.
This is impossible.  So $x$ is even.  We now proceed as in the last
paragraph to obtain lines 7 and 8 of Table~\ref{tab:mainvw}.

We have thus far handled every case in which it is possible to choose
$v=(0,0)$.  We are left with the values of $p$ and $q$ for which
$q\not\equiv 0\mod 4$ and $2p+q\equiv 0\mod 5$.
Table~\ref{tab:mainex} shows that these are precisely the cases in
which $c_1<c_2<c_3<c_4$.  In this situation the elements of
$q_1^{-1}(p_1^{-1}(P_1\cap P_2))$ correspond to $c_1$ and $c_4$ while
the elements of $q_1^{-1}(p_1^{-1}(P_2\setminus P_1))$ correspond to
$c_2$ and $c_3$.  Thus we may choose the line $L$ of
Theorem~\ref{thm:slopefn2} so that it contains $(2,0)$.  As for
line~\ref{lin:maincnges}, replacing $(0,0)$ by $(2,0)$ has the effect
of replacing $r$ by $r-2$ to obtain the following.
  \begin{equation*}
r-2\equiv xq\mod 4\quad s\equiv xp\mod 2\quad 2s+r-2\equiv x(2p+q)\mod 5
  \end{equation*}

First suppose that $q\not\equiv 0\mod 4$, $2p+q\equiv 0\mod 5$ and
$x\equiv 1\mod 2$.  Then $r\not\equiv 2\mod 4$.  Hence $(r,s)\in
\{(0,0),(0,5)\}\subseteq \zL_1$.  As discussed in the previous
paragraph, the elements of $\zL_1$ correspond to $c_1$ and $c_4$ not
$c_2$ or $c_3$.  Thus $x\equiv 0\mod 2$.  Now we verify that the
congruences in the last display are solved by choosing $q\equiv 2\mod
4$, $2p+q\equiv 0\mod 5$, $x=2$ and $(r,s)=(2,0)$.  On the other hand,
if $q\equiv \pm 1\mod 4$, $2p+q\equiv 0\mod 5$ and $x=2$, then
$r\equiv 0\mod 4$ and $s\equiv 0\mod 2$, hence $(r,s)=(0,0)$ and so
$2s+r-2\equiv -2\not\equiv 0\equiv x(2p+q)\mod 5$.  Thus there is no
solution in this case.  Finally, taking $q\equiv \pm 1\mod 4$,
$2p+q\equiv 0\mod 5$, $x=4$ and $(r,s)=(2,0)$ gives a solution.  This
completes the verification of Table~\ref{tab:mainvw}.

Now that we have $v$ and $w$, we find the lattice points
$\zl_1,\dotsc,\zl_n$ which appear in Theorem~\ref{thm:slopefn2}.  A
point $(x,y)\in \bR^2$ is the center of a spin mirror for $p_1\circ
q_1$ if and only if $x$ is an integer congruent to 2 modulo 4 and
there exists an integer $Q_x$ such that $x+2y=10Q_x$.  A point
$(x,y)\in \bR^2$ is in a spin mirror for $p_1\circ q_1$ if and only if
$x$ is an integer congruent to 2 modulo 4 and there exists an integer
$Q_x$ and a real number $R_x$ with $\left|R_x\right|\le 2$ such that
$x+2y=10Q_x+R_x$.

So suppose that $v=(0,0)$.  Let $w=(w_1,w_2)$.  The line segment
joining $v$ and $w$ is in the line given by $y=\frac{p}{q}x$.  For
every integer $x$ define $Q_x$ and $R_x$ so that
  \begin{equation*}
\left(1+\frac{2p}{q}\right)x=10Q_x+R_x\quad
\text{with }Q_x\in \bZ,R_x\in \bQ\text{ and }-5<R_x\le 5.
  \end{equation*}
Let $0<x_1<x_2<x_3<\cdots <x_n<w_1$ be those integers congruent to 2
modulo 4 such that $\left|R_{x_i}\right|<2$.  Set $x_0=0$ and
$x_{n+1}=w_1$.  If $\zl_i$ is the center of the $i$th spin mirror as
in Theorem~\ref{thm:slopefn2}, then
  \begin{equation*}
\zl_i=(x_i,\tfrac{10Q_{x_i}-x_i}{2})=
\tfrac{1}{2}x_i(2,-1)+Q_{x_i}(0,5)\quad\text{for }i\in
\{0,\dotsc,n+1\}.
  \end{equation*}
Set
  \begin{equation*}
N=\sum_{i=0}^{n}(-1)^i(Q_{x_{i+1}}-Q_{x_i})
  \end{equation*}
and
  \begin{equation*}
D=\frac{1}{2}\sum_{i=0}^{n}(-1)^i(x_{i+1}-x_i).
  \end{equation*}
Assembling what we have, Theorem~\ref{thm:slopefn2} implies that
  \begin{equation*}
\zs_f\left(\frac{p}{q}\right)=\frac{N}{D}.
  \end{equation*}

The situation is similar if $v=(2,0)$.  Again let $w=(w_1,w_2)$.  The
line segment joining $v$ and $w$ is in the line given by
$y=\frac{p}{q}x-\frac{2p}{q}$.  For every integer $x$ define $Q_x$ and
$R_x$ so that
  \begin{equation*}
\left(1+\frac{2p}{q}\right)x-\frac{4p}{q}=10Q_x+R_x\quad
\text{with }Q_x\in \bZ,R_x\in \bQ\text{ and }-5<R_x\le 5.
  \end{equation*}
Let $2<x_1<x_2<x_3<\cdots <x_n<w_1$ be those integers congruent to 2
modulo 4 such that $\left|R_{x_i}\right|<2$.  Set $x_0=2$ and
$x_{n+1}=w_1$.  Set
  \begin{equation*}
N=\sum_{i=0}^{n}(-1)^i(Q_{x_{i+1}}-Q_{x_i})
  \end{equation*}
and
  \begin{equation*}
D=\frac{1}{2}\sum_{i=0}^{n}(-1)^i(x_{i+1}-x_i).
  \end{equation*}
Then
  \begin{equation*}
\zs_f\left(\frac{p}{q}\right)=\frac{N}{D}.
  \end{equation*}

We illustrate this formula for $\zs_f(\frac{p}{q})$ by calculating
$\zs_f(\frac{3}{2})$.  Table~\ref{tab:mainvw} shows that $v=(0,0)$ and
$w=(2,3)$.  Since $(1+\frac{2\cdot 3}{2})\cdot 2=10\cdot 1-2$, we
simply have that $n=0$, $x_0=0$, $x_1=2$, $Q_{x_0}=0$ and
$Q_{x_1}=1$.  So $N=Q_{x_1}-Q_{x_0}=1$ and
$D=\frac{1}{2}(x_1-x_0)=1$.  Therefore
$\zs_f(\frac{3}{2})=\frac{N}{D}=1$.

The formulas just derived for the slope function $\zs_f$ of the main
example were used to create the graph of $\zs_f$ shown in
Figure~\ref{fig:slopes256}.  The prominent horizontal lines in the
graph indicate that $\zs_f$ is often infinite-to-one, which can be
proved using the functional equations in Sections~\ref{sec:dehn},
\ref{sec:rflns} and \ref{sec:fnleqns}.  The less prominent vertical
lines indicate that $\zs_f$ does not extend continuously to the
Thurston boundary, which can also be proved using functional
equations.  The prominent fuzzy line with positive slope less than 1
indicates that iteration of $\zs_f$ is probably contracting.  It
seems very possible that under iteration of $\zs_f$, every element
of $\widehat{\mathbb{Q}}$ eventually lies in a finite set of cycles,
although we do not know this.  See Remark~\ref{remark:graph} for an
equation of the prominent fuzzy line.

  \begin{figure}
\centerline{\includegraphics{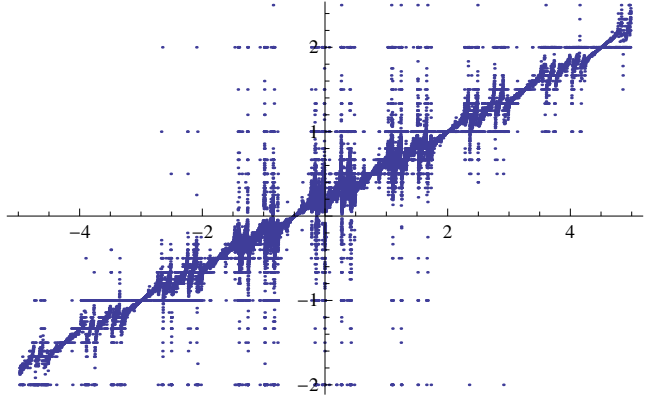}} \caption{The graph of
$\zs_f$ for the main example.}
\label{fig:slopes256}
  \end{figure}

\section{Horoballs in Teichm\"{u}ller space }
\label{sec:horoballs}\nosubsections

We maintain the setting of Section~\ref{sec:defns}.  Let $\bH$ denote
the upper half complex plane, so that $\bH=\{z\in
\bC:\text{Im}(z)>0\}$.  We identify $\bH$ with the Teichm\"{u}ller
space of $S^2\setminus P_2$ as follows.  Let $\zt\in \bH$.  Recall
that we have chosen an ordered basis $(\zl_2,\zm_2)$ of the lattice
$\zL_2$.  There exists a unique $\bR$-linear isomorphism $\zv\co
\bC\to \bR^2$ so that $\zv(1)=2 \zl_2$ and $\zv(\zt)=2 \zm_2$.  The
map $p_2\circ q_2\circ \zv\co \bC\to S^2\setminus P_2$ induces a
complex structure on $S^2\setminus P_2$.  As $\zt$ varies over $\bH$,
the resulting isotopy classes of complex structures on $S^2\setminus
P_2$ are distinct, and every complex structure on $S^2\setminus P_2$
is isotopic to one of them.  In this way we regard $\bH$ as the
Teichm\"{u}ller space of $S^2\setminus P_2$.

In this section we relate horoballs in $\bH$ to moduli of curve
families.  We turn to a determination of convenient equations for the
horocycles in $\bH$.

The horocycles in $\bH$ at $\infty$ are simply horizontal lines, which
are given by equations of the form $\text{Im}(z)=m$ for positive real
numbers $m$.  Now let $p$ and $q$ be relatively prime integers with
$q\ne 0$.  We consider horocycles in $\bH$ at $\frac{p}{q}$.
Figure~\ref{fig:horocycle} shows a horocycle in $\bH$ at $\frac{p}{q}$
with Euclidean diameter $D$.  Two similar right triangles also appear
in Figure~\ref{fig:horocycle}, from which we conclude the following.
  \begin{equation*}
\frac{B}{A}=\frac{A}{D}\Longleftrightarrow \frac{B}{A^2}=\frac{1}{D}
\Longleftrightarrow
\frac{\text{Im}(z)}{\left|z-p/q\right|^2}=\frac{1}{D}
\Longleftrightarrow \frac{\text{Im}(z)}{\left|qz-p\right|^2}=m
  \end{equation*}
Here $m=\frac{1}{q^2D}$ and $D=\frac{1}{q^2m}$.  Thus if $p$ and $q$
are relatively prime integers, then the horoballs in $\bH$ at
$\frac{p}{q}\in \widehat{\bQ}$ are the subsets of the form $\{z\in
\bH:\frac{\text{Im}(z)}{\left|qz-p\right|^2}>m\}$ for positive real
numbers $m$.

\begin{figure}\begin{center} \includegraphics{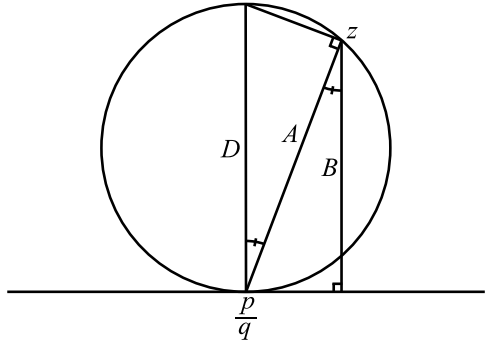}
\caption{ Finding an equation for a horocycle.}
\label{fig:horocycle}
\end{center}\end{figure}

We view $\text{PSL}(2,\bZ)$ as a subgroup of the group of
orientation-preserving isometries of $\bH$, that is,
  \begin{equation*}
\text{PSL}(2,\bZ)=\{z\mapsto \frac{az+b}{cz+d}:\left[\begin{matrix}a & b \\ c & d \end{matrix}\right]\in \text{SL}(2,\bZ)\}.
  \end{equation*}
If $\left[\begin{smallmatrix}a & b \\ c & d \end{smallmatrix}\right]$
is an element of $\text{GL}(2,\bZ)$ with determinant $-1$, then
$z\mapsto \frac{a\overline{z}+b}{c \overline{z}+d}$ is an
orientation-reversing isometry of $\bH$.  This allows us to view
$\text{PGL}(2,\bZ)$ as a subgroup of the group of isometries of $\bH$.
We next consider the action of $\text{PGL}(2,\bZ)$ on the horoballs of
$\bH$.

\begin{lemma}\label{lemma:horocycle}  Let $\zv\in
\text{PGL}(2,\bZ)$, let $\frac{p}{q}\in \widehat{\bQ}$ and suppose
that $\zv(\frac{p}{q})=\frac{p'}{q'}\in \widehat{\bQ}$.  Then
  \begin{equation*}
\frac{\text{Im}(\zv(z))}{\left|q'
\zv(z)-p'\right|^2}=\frac{\text{Im}(z)}{\left|qz-p\right|^2}.
  \end{equation*}
\end{lemma}
  \begin{proof} If the lemma is true for $\zv_1,\zv_2\in
\text{PGL}(2,\bZ)$, then it is true for the composition $\zv_1\circ
\zv_2$.  It follows that since $z\mapsto -\overline{z}$, $z\mapsto
z+1$ and $z\mapsto -\frac{1}{z}$ generate $\text{PGL}(2,\bZ)$, to
prove the lemma, it suffices to prove it for these three
transformations.

First suppose that $\zv(z)=-\overline{z}$.  Then
$\frac{p'}{q'}=-\frac{p}{q}$, and so we may assume that $p'=p$ and
$q'=-q$. We easily see that the lemma is true in this case.  Next
suppose that $\zv(z)=z+1$.  Then $\frac{p'}{q'}=\frac{p}{q}+1$, and so
we may assume that $p'=p+q$ and $q'=q$.  Since
$\text{Im}(z+1)=\text{Im}(z)$, the lemma is true in this case.  If
$\zv(z)=-\frac{1}{z}$, then $\frac{p'}{q'}=-\frac{q}{p}$, and so we
may assume that $p'=-q$ and $q'=p$.  Since
$\text{Im}(-\frac{1}{z})=\frac{\text{Im}(z)}{\left|z\right|^2}$, the
lemma is true in this case too.

This proves Lemma~\ref{lemma:horocycle}.

\end{proof}

\begin{cor}\label{cor:horocycle}  Let $\zv\in
\text{PGL}(2,\bZ)$, let $\frac{p}{q}\in \widehat{\bQ}$ and suppose
that $\zv(\frac{p}{q})=\frac{p'}{q'}\in \widehat{\bQ}$.  Then $\zv$
maps the horoball $\{z\in
\bH:\frac{\text{Im}(z)}{\left|qz-p\right|^2}>m\}$ bijectively to the
horoball $\{z\in \bH:\frac{\text{Im}(z)}{\left|q'z-p'\right|^2}>m\}$
for every positive real number $m$.
\end{cor}

Now we turn our attention to moduli of curve families.  Let $\zt\in
\bH$, let $\zL_\zt=\left<1,\zt\right>$, and let $T_\zt=\bC/\zL_\zt$.
We view $T_\zt$ as a torus with complex structure.  We calculate
slopes of simple closed curves in $T_\zt$ using the ordered basis
$(1,\zt)$ of $\zL_\zt$.  We denote by $\zG_{\frac{p}{q},\zt}$ the set
of simple closed curves in $T_{\zt}$ with slope $\frac{p}{q}$,
although in this paragraph we abbreviate $\zG_{\frac{p}{q},\zt}$ to
$\zG_{\frac{p}{q}}$.  By abuse of notation we write $dz$ for the
1-form on $T_\zt$ induced by the standard 1-form on $\bC$.  For a
nonnegative Borel measurable function $\zr$ on $T_\zt$, define
  \begin{equation*}
L_\zr(\zG_{\frac{p}{q}})=\inf_{\zg\in \zG_{\frac{p}{q}}}
\int_{\zg}^{}\zr\left|dz\right|\quad\text{and}\quad
A_\zr=\iint\limits_{T_\zt}\zr^2\left|dz\right|^2.
  \end{equation*}
The extremal length of the curve family $\zG_{\frac{p}{q}}$ on $T_\zt$
is
  \begin{equation*}
\sup_\zr \frac{L_\zr^2(\zG_{\frac{p}{q}})}{A_\zr}.
  \end{equation*}
The modulus of the curve family $\zG_{\frac{p}{q}}$ on $T_\zt$ is the
reciprocal of this:
  \begin{equation*}
\text{mod}_\zt(\tfrac{p}{q})=\inf_\zr\frac{A_\zr}{L_\zr^2(\zG_{\frac{p}{q}})}.
  \end{equation*}

The modulus is a conformal invariant: if $\zf:T_\zt\to T_{\zt'}$ is a
conformal isomorphism sending $\zG_{\frac{p}{q},\zt}$ to
$\zG_{\frac{p'}{q'},\zt'}$, then
  \begin{equation*}
\text{mod}_\zt(\tfrac{p}{q})=\text{mod}_{\zt'}(\tfrac{p'}{q'}).
  \end{equation*}

\begin{lemma}\label{lemma:mod}  For every $\frac{p}{q}\in
\widehat{\bQ}$ and $\zt\in \bH$ we have that
  \begin{equation*}
\text{mod}_\zt(\tfrac{p}{q})=\frac{\text{Im}(\zt)}{\left|p
\zt+q\right|^2}.
  \end{equation*}
\end{lemma}
  \begin{proof} We first prove the lemma for $\frac{p}{q}=0$.  Let
$\zt\in \bH$.  Let $h=\text{Im}(\zt)$.  The parallelogram $P$ with
vertices 0, 1, $\zt$ and $1+\zt$ is a fundamental domain for the
action of $\zL_\zt$ on $\bC$.  Every element of $\zG_{0,\zt}$ is
homotopic to the image in $T_\zt$ of a horizontal line segment joining
the left and right sides of this fundamental domain.  Let $\zr$ be the
Borel function on $T_\zt$ which equals 1 everywhere.  Then $A_\zr$ is
the Euclidean area of $P$, and so $A_\zr=\text{Im}(\zt)=h$.  Now we
consider $L_\zr(\zG_{0,\zt})$.  Let $\zg\in \zG_{0,\zt}$.  Every lift
of $\zg$ to $\bC$ has the property that its endpoints have equal
imaginary parts and even that the Euclidean distance between these
endpoints is 1.  Since $\int_{\zg}^{}\zr\left|dz\right|$ is the
Euclidean length of every such lift, it follows that
$L_\zr(\zG_{0,\zt})=1$.  So $\text{mod}_\zt(0)\le
\frac{A_\zr}{L_\zr^2(\zG_{0,\zt})}=h$.

Now let $\zr$ be any nonnegative Borel measurable function on $T_\zt$.
Let $\widetilde{\zr}$ be the lift of $\zr$ to $\bC$.  Let
$l=L_\zr(\zG_{0,\zt})$.  We use the fact that the rectangle $R$ whose
vertices are 0, 1, $hi$ and $1+hi$ is a fundamental domain for the
action of $\zL_\zt$ on $\bC$.  Horizontal line segments in $R$ give
rise to elements of $\zG_{0,\zt}$, that is, for every $y\in [0,h]$ the
curve $\zg_y$ parametrized by $\zg(x)=x+yi \mod \zL_\zt$ from the
closed interval $[0,1]$ to $T_\zt$ is an element of $\zG_{0,\zt}$.
Hence $l\le \int_{0}^{1}\widetilde{\zr}(x,y)\,dx$ for every $y\in
[0,h]$, and so
  \begin{equation*}
hl\le \int_{0}^{h}\int_{0}^{1}\widetilde{\zr}\,dxdy\le
\left(\iint\limits_{R}dxdy\iint\limits_{R}\widetilde{\zr}^2dxdy\right)^{1/2}
=(hA_\zr)^{1/2}.
  \end{equation*}
Thus
  \begin{equation*}
\frac{A_\zr}{L_\zr^2(\zG_{0,\zt})}=\frac{A_\zr}{l^2}\ge h.
  \end{equation*}
Since $\text{mod}_\zt(0)$ is the infimum of such terms,
$\text{mod}_\zt(0)\ge h$.  This inequality and the conclusion of the
previous paragraph imply that $\text{mod}_\zt(0)=h=\text{Im}(\zt)$.
This proves Lemma~\ref{lemma:mod} if $\frac{p}{q}=0$.

Now suppose that $\frac{p}{q}\ne 0$.  Since $p$ and $q$ are relatively
prime, there exist integers $r$ and $s$ such that $rq-sp=1$.  Hence
$(p \zt+q,r \zt+s)$ is an ordered basis of $\zL_\zt$.  Setting
$\zt'=\frac{r \zt+s}{p \zt+q}\in \bH$, we find that the conformal map
$\widetilde{\zf}:\bC\to \bC$ given by
  \begin{equation*}
\widetilde{\zf}(z)=\frac{z}{p \zt+q}
  \end{equation*}
sends the lattice $\zL_\zt$ to $\zL_{\zt'}$ and sends the ordered basis
$(p \zt+q,r \zt+s)$ to $(1,\zt')$.  Hence it descends to a conformal
isomorphism $\zf:T_\zt\to T_{\zt'}$.  Moreover,
  \begin{equation*}
\zf(\zG_{\frac{p}{q},\zt})=\zG_{0,\zt'},
  \end{equation*}
and so since the modulus is a conformal invariant, we have that
  \begin{equation*}
\text{mod}_\zt(\tfrac{p}{q})=\text{mod}_{\zt'}(0).
  \end{equation*}
This and the previous paragraph imply that
  \begin{equation*}
\text{mod}_\zt(\tfrac{p}{q})=\text{Im}(\zt').
  \end{equation*}
Finally, Lemma~\ref{lemma:horocycle} with $\frac{p}{q}$ there replaced
by $-\frac{q}{p}$ shows that
$\text{mod}_\zt(\frac{p}{q})=\frac{\text{Im}(\zt)}{\left|p
\zt+q\right|^2}$.

This proves Lemma~\ref{lemma:mod}.

\end{proof}

For every $\frac{p}{q}\in \widehat{\bQ}$ and every positive real
number $m$, we set
  \begin{equation*}
B_m(\tfrac{p}{q})=\{\zt\in \bH:\text{mod}_\zt(\tfrac{p}{q})>m\}.
  \end{equation*}

\begin{cor}\label{cor:mod}  If $\frac{p}{q}\in \widehat{\bQ}$ and
if $m$ is a positive real number, then
  \begin{equation*}
B_m(\tfrac{p}{q})=\{\zt\in \bH:\tfrac{\text{Im}(\zt)}{\left|p
\zt+q\right|^2}>m\},
  \end{equation*}
a horoball in $\bH$ at $-\frac{q}{p}$.
\end{cor}

Let $S_\zt$ be the quotient space of $T_\zt$ determined by the map
$z\mapsto -z$. Let $p_\zt\co T_\zt\to S_\zt$ be the corresponding
degree 2 branched covering map.  Let $P_\zt$ be the set of four branch
points of $p_\zt$ in $S_\zt$.  Let $\frac{p}{q}\in \widehat{\bQ}$.  By
definition, the set of essential simple closed curves in
$S_\zt\setminus P_\zt$ with slope $\frac{p}{q}$ lifts under $p_\zt$ to
$\zG_{\frac{p}{q},\zt}$.  We define the modulus of this family of
curves just as we defined $\text{mod}_\zt(\frac{p}{q})$. Because
lengths of curves do not change when pulling back from $S_\zt$ to
$T_\zt$ but area doubles, this new modulus is
$\frac{1}{2}\text{mod}_\zt(\frac{p}{q})$.

Let $f$ be a NET map.  We define a function $\zd_f\co \widehat{\bQ}\to
\bQ$ as follows.  Let $\frac{p}{q}\in \widehat{\bQ}$.  Let $\zg$ be an
essential simple closed curve in $S^2\setminus P_2$ with slope
$\frac{p}{q}$.  Let $d$ be the degree with which $f$ maps every
connected component of $f^{-1}(\zg)$ to $\zg$, and let $c$ be the
number of these connected components which are essential and
nonperipheral.  Then $\zd_f(\frac{p}{q})=\frac{c}{d}$.  The multicurve
$\zG$ whose only element is $\zg$ is $f$-stable if and only if either
$\zs_f(\frac{p}{q})=\frac{p}{q}$ or $\zs_f(\frac{p}{q})=o$.  If $\zG$
is $f$-stable, then the Thurston matrix $A^\zG$ is $1\times 1$ with
entry $\zd_f(\frac{p}{q})$.  Thus $\zG$ is a Thurston obstruction if
and only if $\frac{p}{q}\in \text{Fix}(\zs_f)$ and
$\zd_f(\frac{p}{q})\ge 1$.

We maintain the setting of the previous paragraph.  Recall from the
introduction that $\zS_f\co \mathbb{H}\to \mathbb{H}$ is the map on
Teichm\"{u}ller space induced by $f$.  Let $\zt\in \bH$, let
$\zt'=\zS_f(\zt)$, let $\frac{p'}{q'}=\zs_f(\frac{p}{q})$ and let
$\zd=\zd_f(\frac{p}{q})$.  Then an argument based on the subadditivity
of moduli proves that
  \begin{equation*}
\text{mod}_{\zt'}(\tfrac{p'}{q'})\ge \zd\text{mod}_\zt(\tfrac{p}{q}).
  \end{equation*}
Hence
  \begin{equation*}\linnum\label{lin:contain}
\zS_f(B_m(\tfrac{p}{q}))\subseteq B_{\zd m}(\tfrac{p'}{q'})
  \end{equation*}
for every positive real number $m$.

\begin{remark}\label{remark:selinger} For a sphere with four marked
points, the Teichm\"{u}ller and hyperbolic metrics coincide.  There is
another natural metric on Teichm\"{u}ller space, the so-called
Weil-Petersson (WP) metric. The WP metric is incomplete. For a sphere
with four marked points, the WP boundary, as a topological space, is
discrete and as a set is $\widehat{\mathbb{Q}}$.  A neighborhood
basis element for a boundary point $-\frac{q}{p}$ is a horoball
tangent to $-\frac{q}{p}$ union the singleton $\{-\frac{q}{p}\}$.
Selinger \cite{Se} shows that $\zS_f$ is $\sqrt{\deg(f)}$-Lipschitz
with respect to the WP metric, and so extends to the WP completion. It
easily follows that for NET maps, computation of $\zs_f$ is the
computation of the boundary values of $\zS_f$ on the WP completion of
Teichm\"{u}ller space.

\end{remark}

We next prove the following theorem in the above setting.  Let
$d(\cdot ,\cdot )$ denote the hyperbolic metric for $\bH$.  If $H$ is
a half-space in $\mathbb{H}$, then we let $\partial _{\infty}H$ denote
the set of points in the boundary of $\mathbb{H}$ which are limits of
sequences in $H$.

\begin{thm}\label{thm:halfsp}
\begin{enumerate}
  \item If $\frac{p}{q}\ne \frac{p'}{q'}$, then for every sufficiently
large $m$, the closed horoballs $B=\overline{B_m(\frac{p}{q})}$ and
$B'=\overline{B_{\zd m}(\frac{p'}{q'})}$ are disjoint.  When they are
disjoint the set $H=\{\zt\in \bH:d(\zt,B)<d(\zt,B')\}$ is an open
hyperbolic half-space which is independent of $m$.
  \item  If $\frac{r}{s}\in \text{Fix}(\zs_f)$ and
$-\frac{s}{r}\in \partial _\infty H$, then $\zd_f(\frac{r}{s})<1$,
that is, there is no Thurston obstruction with slope $\frac{r}{s}$.
  \item If $\zt_0\in \text{Fix}(\zS_f)$, then $\zt_0\notin H$.
\end{enumerate}
\end{thm}
  \begin{proof} We first prove statement 1.  It is clear from the
description of horoballs above that if $m$ is sufficiently large, then
$B$ and $B'$ are disjoint.  We assume that $m$ is this large.  For
every $m>0$, every $\frac{a}{b}\in \widehat{\bQ}$ and every $t>0$, the
hyperbolic distance between the horocycles $\partial B_m(\frac{a}{b})$
and $\partial B_{tm}(\frac{a}{b})$ is equal to $\left|\ln(t)\right|$;
in particular this is independent of $m$.  Let $l$ be the hyperbolic
geodesic joining $-\frac{q}{p}$ and $-\frac{q'}{p'}$, let
$l_m\subseteq l$ be the closure of the geodesic segment lying outside
$B\cup B'$ and let $l_m^\perp$ be its perpendicular bisector.  Then
$l_m^\perp$ is independent of $m$, and so $H$ is an open hyperbolic
half-space independent of $m$.  This proves statement 1.

To prove statement 2, suppose that $\frac{r}{s}\in \text{Fix}(\zs_f)$,
that $-\frac{s}{r}\in \partial _\infty H$ and to the contrary that
$\zd_f(\frac{r}{s})\ge 1$.  We now choose $m$ so that $B$ and $B'$
intersect in a single point.  Figure~\ref{fig:halfsp} shows $B$ and
$B'$ with $\frac{p}{q}=0$.  It also shows $l_m^\perp$.  Now we choose
$m^*>0$ so that $B$ and $\overline{B_{m^*}(\frac{r}{s})}$ also
intersect in a single point $\zh$.  The assumption that
$-\frac{s}{r}\in \partial_ \infty H$ implies that
$\overline{B_{m^*}(\frac{r}{s})}\cap B'=\emptyset$.
Line~\ref{lin:contain} implies that $\zS_f(B)\subseteq B'$ and
$\zS_f(\overline{B_{m^*}(\frac{r}{s})})\subseteq
\overline{B_{m^*}(\frac{r}{s})}$ because $\zd_f(\frac{r}{s})\ge 1$.
But then $\zS_f(\zh)\in B'\cap \overline{B_{m^*}(\frac{r}{s})}$, which
is impossible.  This proves statement 2.

\begin{figure}\begin{center}
\includegraphics{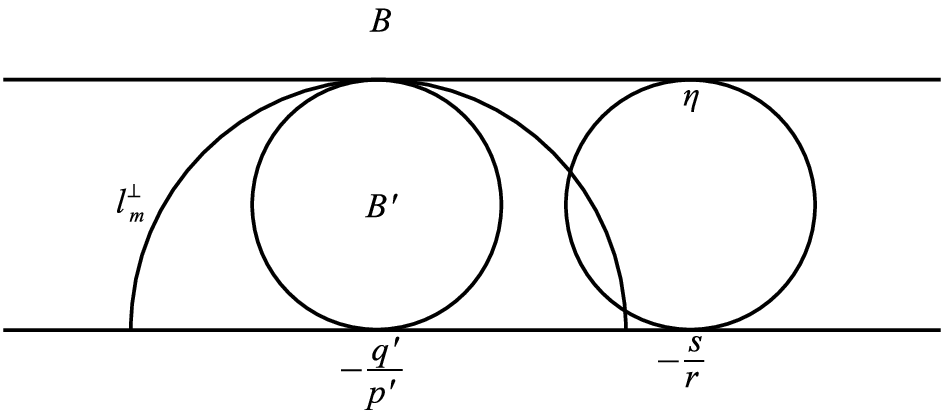} \caption{ Proving Theorem~\ref{thm:halfsp}.}
\label{fig:halfsp}
\end{center}\end{figure}

To prove statement 3, suppose that $\zt_0\in \text{Fix}(\zS_f)$.  Let
$\zt_1\in B$ realize the distance between $\zt_0$ and $B$.
Line~\ref{lin:contain} implies that $\zS_f(B)\subseteq B'$ and $\zS_f$
is distance nonincreasing, so
  \begin{equation*}
d(\zt_0,B)=d(\zt_0,\zt_1)\ge d(\zt_0,\zS_f(\zt_1))\ge d(\zt_0,B').
  \end{equation*}
Hence $\zt_0\notin H$, proving statement 3.

This proves Theorem~\ref{thm:halfsp}.

\end{proof}

  \begin{figure}
\centerline{\includegraphics{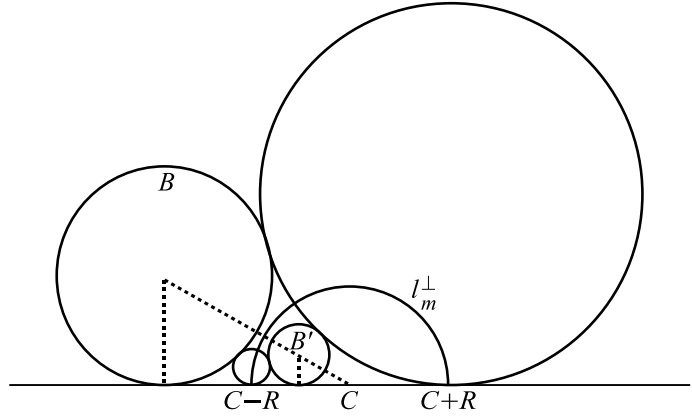}} \caption{ The
generic situation in Theorem~\ref{thm:halfsp}.}
\label{fig:halfspb}
  \end{figure}

Figure~\ref{fig:halfspb} shows the generic situation in
Theorem~\ref{thm:halfsp}.  The horocycle boundaries of $B$ and $B'$
are shown, tangent to the real line at $-\frac{q}{p}$ and
$-\frac{q'}{p'}$, respectively.  The geodesic $l_m^\perp$ in the proof
of Theorem~\ref{thm:halfsp} is shown, with Euclidean center $C$ and
Euclidean radius $R$.  Two other horocycles are shown.  The larger one
is tangent to $B$ and the real line at $C+R$.  The smaller one is
tangent to $B$ and the real line at $C-R$.  Because $l$ is the
geodesic joining $-\frac{q}{p}$ and $-\frac{q'}{p'}$ and $l_m^\perp$ is
the perpendicular bisector of the segment of $l$ joining $B$ and $B'$,
inversion about $l_m^\perp$ stabilizes these two horocycles and
interchanges $B$ and $B'$.  So these two horocycles are also tangent
to $B'$.

The proof of the main part of Theorem~\ref{thm:halfsp} can be
restated as follows.  Suppose that $f$ has a Thurston obstruction with
slope $r$.  Line~\ref{lin:contain} implies that $\zS_f$ maps every
horoball at $r$ into itself.  Let $B_r$ be the closed horoball at $r$
tangent to $B$.  Then $\zS_f(B)\cap B_r\ne \emptyset$, because $\zS_f$
maps $B_r$ into itself.  But then $B'\cap B_r\ne \emptyset$.  So given
$B$ and $B'$, it follows that $C-R\le r\le C+R$.

To apply Theorem~\ref{thm:halfsp}, it is useful to have concrete
formulas for $C$ and $R$.  We find such formulas now.  The dotted line
segments in Figure~\ref{fig:halfspb} identify two similar right
triangles.  The height of the larger one times 2 is the Euclidean
diameter of $B$, which by the argument involving
Figure~\ref{fig:horocycle} is $D=\frac{1}{mp^2}$.  The width of this
right triangle is $C+\frac{q}{p}$.  Corresponding statements hold for
the smaller right triangle, and so
  \begin{equation*}
\frac{C+\frac{q}{p}}{\frac{1}{mp^2}}=
\frac{C+\frac{q'}{p'}}{\frac{1}{\zd mp'^2}}\Longleftrightarrow
p^2C+pq=\zd p'^2C+\zd p'q'\Longleftrightarrow
C=\frac{-pq+\zd p'q'}{p^2-\zd p'^2}.
  \end{equation*}

Because $-\frac{q}{p}$ is gotten from $-\frac{q'}{p'}$ by inversion
about $l_m^\perp$, it follows that
$R^2=(C+\frac{q}{p})(C+\frac{q'}{p'})$.  The last display shows that
$C+\frac{q'}{p'}=\frac{p^2}{\zd p'^2}(C+\frac{q}{p})$.  So
$R^2=\frac{p^2}{\zd p'^2}(C+\frac{q}{p})^2$.  We calculate:
  \begin{equation*}
C+\frac{q}{p}=\frac{-pq+\zd p'q'}{p^2-\zd p'^2}+\frac{q}{p}=
\frac{-p^2q+\zd pp'q'+p^2q-\zd p'^2q}{p(p^2-\zd p'^2)}=
\frac{\zd p'(pq'-p'q)}{p(p^2-\zd p'^2)}.
  \end{equation*}
So
  \begin{equation*}
R=\left|\frac{(pq'-p'q)\sqrt{\zd}}{p^2-\zd p'^2 }\right|.
  \end{equation*}
These formulas for $C$ and $R$ hold if neither $\frac{p}{q}$ nor
$\frac{p'}{q'}$ is 0 and the Euclidean radii of $B$ and $B'$ are
unequal.  If either $\frac{p}{q}=0$ or $\frac{p'}{q'}=0$, then one may
either directly verify that these formulas still hold or one may apply
a continuity argument.  We note that if $\zd>\frac{p^2}{p'^2}$, then
$B$ has larger Euclidean radius than $B'$ as in
Figure~\ref{fig:halfspb} and $H$ is the region in $\mathbb{H}$ outside
the Euclidean circle with center $C$ and radius $R$.  If
$\zd<\frac{p^2}{p'^2}$, then $H$ is the region in $\mathbb{H}$ within
this circle.

Finally, we consider the case in which $\zd=\frac{p^2}{p'^2}$.  In
this case $B$ and $B'$ have the same Euclidean radius and $l_m^\perp$
is a Euclidean half-line.  One endpoint of $l_m^\perp$ is $\infty$ and
the other is $-\frac{1}{2}(\frac{q}{p}+\frac{q'}{p'})$.  So if
$\zd=\frac{p^2}{p'^2}$, then
  \begin{equation*}
H=\{\zt\in \bH:\text{Re}(\zt)<-\frac{1}{2}(\frac{q}{p}+\frac{q'}{p'})\}
\text{ if }\frac{p}{q}<\frac{p'}{q'}
  \end{equation*}
and
  \begin{equation*}
H=\{\zt\in \bH:\text{Re}(\zt)>-\frac{1}{2}(\frac{q}{p}+\frac{q'}{p'})\}
\text{ if }\frac{p}{q}>\frac{p'}{q'}
  \end{equation*}
with the convention that every rational number is less than $\infty$.

\begin{ex}\label{ex:mainhoro}  In this example we apply
these ideas to the Thurston map $f$ of the main example.
Table~\ref{tab:mainhoro} contains values of $\frac{p}{q}$,
$\frac{p'}{q'}$, $\zd$, $C$, and $R$ and the last column states
whether or not $H$ is bounded in the Euclidean metric.  The values of
$p'$ and $q'$ can be computed using the results at the end of
Section~\ref{sec:slopefn}.  The values of $\zd$ come from the values
of $c$ and $d$ in Table~\ref{tab:mainex}.

\begin{table}\renewcommand{\arraystretch}{1.4}
\begin{tabular}{|c|c|c|c|c|c|}\hline
$\frac{p}{q}$ & $\frac{p'}{q'}$ & $\zd$  & $C$ & $R$ & $H$ bounded?
  \\ \hline
$-\frac{1}{2}$  & 0 & 6 & 2 & $\sqrt{6}$ & yes \\ \hline
$-\frac{1}{4}$  & $\frac{1}{6}$ & $\frac{2}{5}$ & $\frac{32}{3}$ &
  $\frac{10\sqrt{10}}{3}$ & yes \\ \hline
$\frac{1}{8}$  & $\frac{1}{4}$ & 2 & 0 &
  $4\sqrt{2}$ & no \rule{0ex}{2.5ex}\\ \hline
$\frac{1}{4}$  & $\frac{1}{2}$ & $\frac{2}{5}$ & $-\frac{16}{3}$ &
  $\frac{2\sqrt{10}}{3}$ & yes \\ \hline
$\frac{1}{3}$  & 0 & $\frac{1}{2}$ & $-3$ &
  $\frac{1}{\sqrt{2}}$ & yes \\ \hline
$\frac{7}{16}$  & $\frac{1}{4}$ & 6 & $-\frac{88}{43}$ &
  $\frac{12\sqrt{6}}{43}$ & yes \\ \hline
$\frac{1}{2}$  & $\frac{1}{3}$ & $\frac{2}{5}$ & $-\frac{4}{3}$ &
  $\frac{\sqrt{10}}{3}$ & yes \\ \hline
$\frac{3}{4}$  & $\frac{1}{2}$ & 6 & 0 &
  $\frac{2\sqrt{6}}{3}$ & yes \\ \hline
\end{tabular}
\smallskip
\caption{ Half-space data for the main example.}
\label{tab:mainhoro}
\end{table}

Figure~\ref{fig:halfspinttn} shows the corresponding half-spaces $H$
in $\bH$.  The intersection of their complements is shaded.  Statement
2 of Theorem~\ref{thm:halfsp} and the fact that this intersection is a
bounded subset of $\mathbb{H}$ imply that $f$ has no Thurston
obstructions.  Thus the Teichm\"{u}ller map of $f$ has a fixed point
in $\bH$, the map $f$ is equivalent to a rational map and the finite
subdivision rule of the main example is combinatorially conformal.
Statement 3 of Theorem~\ref{thm:halfsp} implies that the fixed point
of the Teichm\"{u}ller map of $f$ is in the shaded region of
Figure~\ref{fig:halfspinttn}.  More half-spaces are drawn in
Figure~\ref{fig:halfspinttn} than is necessary to obtain these
results; the half-spaces corresponding to $\frac{7}{16}$ and
$-\frac{1}{2}$ are not necessary.  These two half-spaces are included
because computations suggest that the half-space corresponding to each
value of $\frac{p}{q}\in \widehat{\bQ}$ is contained in one of the
eight shown.

\end{ex}

\begin{figure}\begin{center}
\includegraphics{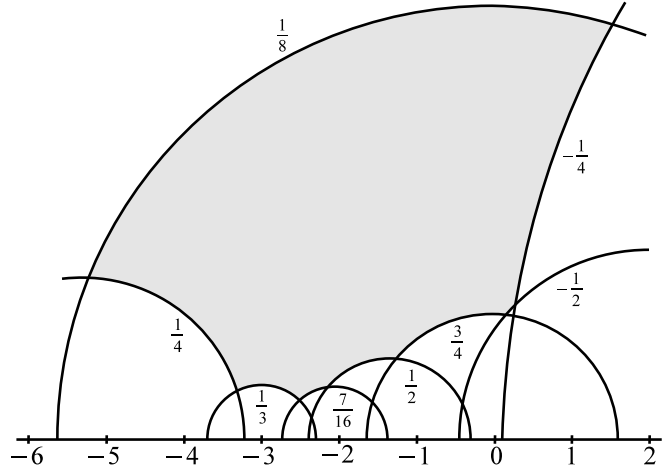} \caption{ Some
half-spaces for the main example.}
\label{fig:halfspinttn}
\end{center}\end{figure}

\section{Dehn twists }\label{sec:dehn}\nosubsections

We fix our conventions concerning Dehn twists in this paragraph.  Let
$\zg$ be a simple closed curve in an oriented surface $S$.  Let $A$ be
a closed regular neighborhood of $\zg$ in $S$, so that there exists an
orientation-preserving homeomorphism $g\co \{z\in \bC:1\le
\left|z\right|\le 2\}\to A$.  A (right-handed) Dehn twist about $\zg$
is a homeomorphism from $S$ to $S$ which is the identity map outside
of $A$ and on $A$ it has the form $g\circ h\circ g^{-1}:A\to A$, where
$h(re^{2 \zp i \zq})=re^{2 \zp i(\zq-r)}$.  Thus for a fixed $\zq$, as
$r$ traverses the closed interval $[1,2]$ in either direction,
$h(re^{2 \zp i \zq})$ bends to the right as it winds once around the
annulus $\{z\in \bC:1\le \left|z\right|\le 2\}$.

Now let $\zg$ be a simple closed curve in $S^2\setminus P_2$, and let
$t\co S^2\setminus P_2\to S^2\setminus P_2$ be a right-handed Dehn
twist about $\zg$.  We next determine the induced map $\zS_t\co \bH\to
\bH$ on Teichm\"{u}ller space.  If $\zg$ is inessential or peripheral,
then $t$ is homotopic to the identity map, and so $\zS_t$ is the
identity map.  So suppose that $\zg$ is essential with slope
$\frac{p}{q}$.  Then $\zg$ is homotopic to the image in $S^2\setminus
P_2$ under $p_2\circ q_2$ of a line segment.  Such a line segment is
drawn with dashes in Figure~\ref{fig:gammatwist}.  Here $\zl=q \zl_2+p
\zm_2$ and $\zm=s \zl_2+r \zm_2$, where $r$ and $s$ are integers with
$ps-qr=1$, so that $\zl$ and $\zm$ form a basis of $\zL_2$.  Let
$\zs_t\co \widehat{\mathbb{Q}}\to \widehat{\mathbb{Q}}\cup
\{o\}$ be the induced map on slopes.

We first consider the case for which $\frac{p}{q}=0$.  Then
$\zs_t(0)=0$.  This and the results of Section~\ref{sec:horoballs} up
to and including Lemma~\ref{lemma:mod} imply that $\zS_t$ maps every
horocycle based at $\infty$ to itself.  By considering
Figure~\ref{fig:gammatwist}, we see that $\zs_t(\infty)=-\frac{1}{2}$
and $\zs_t(1)=-1$.  This and the results of
Section~\ref{sec:horoballs} imply that $\zS_t$ maps horocycles based
at 0, respectively $-1$, with Euclidean diameter $D$ to horocycles
based at 2, respectively 1, with Euclidean diameter $D$.  Now let
$z\in \mathbb{H}$.  The element $z$ determines horocycles based at
$\infty$, 0 and $-1$.  The map $\zS_t$ takes the first horocycle to
itself, it takes the second horocycle to a horocycle based at 2
maintaining Euclidean diameter and it takes the third horocycle to a
horocycle based at 1 maintaining Euclidean diameter.  Hence
$\zS_t(z)=z+2$ for every $z\in \mathbb{H}$.

Now consider a general value of $\frac{p}{q}\in
\widehat{\mathbb{Q}}$.  Let $r$ and $s$ be integers such that
$ps-qr=1$.  The induced map on $\mathbb{H}$ of a Dehn twist about a
simple closed curve with slope $\frac{p}{q}$ is conjugate to the
induced map on $\mathbb{H}$ of a Dehn twist about a simple closed
curve with slope 0.  In terms of matrices, this conjugation has the
form
  \begin{equation*}
\begin{aligned}
\left[\begin{matrix}-q & -s \\ p & r \end{matrix}\right]
  \left[\begin{matrix}1 & 2 \\ 0 & 1 \end{matrix}\right]
  \left[\begin{matrix}-q & -s \\ p & r \end{matrix}\right]^{-1}
& = \left[\begin{matrix}-q & -s-2q \\ p & r+2p \end{matrix}\right]
  \left[\begin{matrix}r & s \\ -p & -q \end{matrix}\right]\\
& = \left[\begin{matrix}1+2pq & 2q^2 \\ -2p^2 & 1-2pq \end{matrix}\right].
\end{aligned}
  \end{equation*}
The matrix $\left[\begin{smallmatrix} -q& -s \\ p &
r \end{smallmatrix}\right]$ in $\text{SL}(2,\bZ)$ represents a
M\"{o}bius transformation which maps $\infty$ to $-\frac{q}{p}$.  The
matrix $\left[\begin{smallmatrix} 1& 2 \\ 0 &
1 \end{smallmatrix}\right]$ represents the second power of a
generator of the stabilizer of $\infty$ in $\text{PSL}(2,\bZ)$, and it
translates horocycles at $\infty$ in the counterclockwise direction.
So $\zS_t$ is the second power of a generator of the stabilizer of
$-\frac{q}{p}$ in $\text{PSL}(2,\bZ)$, and it translates horocycles at
$-\frac{q}{p}$ in the counterclockwise direction.

\begin{figure}\begin{center} \includegraphics{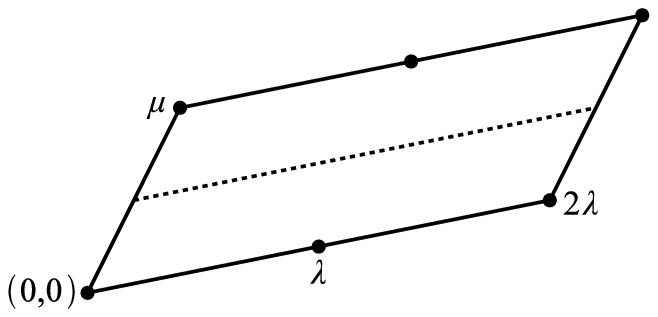}
\caption{ Determining $\zS_t$.}
\label{fig:gammatwist}\end{center}
\end{figure}

Now we consider the effect of Dehn twists on a NET map $f$.  Let $\zg$
be an essential, nonperipheral simple closed curve in $S^2\setminus
P_2$.  Let $t$ be a right-handed Dehn twist about $\zg$.  Let $d$ be
the degree with which $f$ maps every connected component of
$f^{-1}(\zg)$ to $\zg$.  The number of connected components of
$f^{-1}(\zg)$ is $d'=\text{deg}(f)/d$.  Let $t_1,\dotsc,t_{d'}$ be
right-handed Dehn twists about these simple closed curves in
$S^2\setminus P_2$.  By considering Figure~\ref{fig:degscmpts}, we see
that $t^d\circ f$ is homotopic to $f\circ t_1\circ \cdots \circ
t_{d'}$.  Hence
  \begin{equation*}
\zS_f\circ \zS_t^d=\zS_{t_{d'}}\circ \cdots \circ \zS_{t_1}\circ \zS_f.
  \end{equation*}
The Dehn twists among $t_1,\dotsc,t_{d'}$ which correspond to
inessential or peripheral connected components act trivially on $\bH$,
and the rest induce the same map.  We have proved the following
theorem.

\begin{thm}\label{thm:fnleqns} Let $f$ be a NET map.  Let
$\frac{p}{q}\in \widehat{\bQ}$, and let $\zg$ be an essential simple
closed curve in $S^2\setminus P_2$ with slope $\frac{p}{q}$.  Let $c$
be the number of connected components of $f^{-1}(\zg)$ which are
essential and nonperipheral in $S^2\setminus P_2$, and suppose that
$f$ maps each of these components to $\zg$ with degree $d$.  Suppose
that each of these essential connected components has slope
$\frac{p'}{q'}\in \widehat{\bQ}$.  Let
  \begin{equation*}
\zv(z)=\frac{(1+2pq)z+2q^2}{-2p^2z+1-2pq}\quad\text{and}\quad
\zj(z)=\frac{(1+2p'q')z+2q'^2}{-2p'^2z+1-2p'q'}.
  \end{equation*}
Then
  \begin{equation*}
\zS_f\circ \zv^d=\zj^c\circ \zS_f.
  \end{equation*}
\end{thm}

Theorem~\ref{thm:fnleqns} includes the possibility that $f^{-1}(\zg)$
contains no essential components, in which case $c=0$ and $\zS_f\circ
\zv^d=\zS_f$.

\begin{remark}\label{remark:virtual} Every Thurston map $f\co
(S^2,P_f)\to (S^2, P_f)$ defines a virtual endomorphism
  \begin{equation*}
\zf_f\co \text{Mod}(S^2,P_f)\dashrightarrow \text{Mod}(S^2,P_f)
  \end{equation*}
as in \cite{P}.  For NET maps, computation of the slope function
$\zs_f$ together with the mapping degree and number of preimages
amounts to computation of $\zf_f$ on powers of Dehn twists. This is
what Lodge exploits in \cite{L} to calculate the slope
function for the NET map of Example~\ref{ex:BEKP}: he finds (in the
terminology of Theorem~\ref{thm:fnleqns}) $\zj^c$ in terms of $\zv^d$,
which allows him to find $\frac{p'}{q'}$ in terms of $\frac{p}{q}$.

\end{remark}

\begin{ex}\label{ex:fnleqns}  We apply
Theorem~\ref{thm:fnleqns} to the main example with slope $\infty$.  It
is easy to see that $\zs_f(\infty)=\infty$; the methods of
Section~\ref{sec:slopefn} are not needed.  We have that
$\zv(z)=\zj(z)=\frac{z}{-2z+1}$.  Table~\ref{tab:mainex} shows that
$c=2$ and $d=5$.  The equation $\psi^{-c} \circ \Sigma_f = \Sigma_f
\circ \zv^{-d}$ yields that
  \begin{equation*}
\zS_f\left(\frac{z}{10z+1}\right)=\frac{\zS_f(z)}{4 \zS_f(z)+1}.
  \end{equation*}
\end{ex}

\section{Reflections }\label{sec:rflns}\nosubsections

By a reflection $\zr$ of a 2-sphere about a simple closed curve $\zg$
we mean a map which is topologically conjugate to the map of
$\{(x,y,z)\in \bR^3:x^2+y^2+z^2=1\}$ given by $(x,y,z)\mapsto
(x,y,-z)$ and which has $\gamma$ as the set of points fixed by $\zr$.
We are interested in maps of Teichm\"{u}ller space induced by
reflections.

We consider these maps of Teichm\"{u}ller space in this paragraph.  We
maintain the setting of Section~\ref{sec:defns}.  The boundary of the
parallelogram with vertices 0, $\zl_2$, $\zm_2$ and $\zl_2+\zm_2$ maps
under $p_2\circ q_2$ to a simple closed curve $\zg$ in $S^2$.  Let
$\zr$ be a reflection of $S^2$ about $\zg$.  Let $\zs_\zr\co
\widehat{\bQ}\to \widehat{\bQ}\cup \{o\}$ be the induced map on
slopes, and let $\zS_\zr\co \bH\to \bH$ be the induced map on
Teichm\"{u}ller space.  Let $\zt\in \bH$, and let $\zt'=\zS_\zr(\zt)$.
We see that $\zs_\zr(0)=0$.  So
$\text{mod}_\zt(0)=\text{mod}_{\zt'}(0)$.  This and
Lemma~\ref{lemma:mod} and the discussion at the beginning of
Section~\ref{sec:horoballs} imply that $\zt$ and $\zt'$ are on the
same horocycle of $\bH$ at $\infty$.  In the same way $\zt$ and $\zt'$
are on the same horocycle of $\bH$ at 0.  But distinct horocycles meet
in at most two points.  So from this alone we conclude that either
$\zt'=\zt$ or $\zt'=-\overline{\zt}$.  Since in general
$\zs_\zr(x)=-x$, it follows that $\zt'=-\overline{\zt}$.  So
$\zS_\zr(\zt)=-\overline{\zt}$ for every $\zt\in \bH$.

The discussion of the previous paragraph can be generalized as
follows.  Let $\zg$ be a simple closed curve in $S^2$, and let $\zr$
be a reflection of $S^2$ about $\zg$.  Suppose that $\zr(P_2)=P_2$.
The number of elements of $P_2$ fixed by $\zr$ is either 4, 2 or 0.
The three possibilities are illustrated in Figure~\ref{fig:fixrho}
with $\zg$ drawn as an equator.  In every case there exist two
essential simple closed curves in $S^2\setminus P_2$ with distinct
slopes which are fixed by $\zr$.  Two such curves are drawn in the
first two parts of Figure~\ref{fig:fixrho}.  In the third part $\zg$
is one of these two curves.  If these two curves have slopes
$\frac{p}{q}$ and $\frac{r}{s}$, then $\zs_\zr$ is the reflection of
$\widehat{\bQ}$ which fixes $\frac{p}{q}$ and $\frac{r}{s}$.
Similarly, $\zS_\zr$ is the reflection of $\bH$ which fixes the
geodesic with endpoints $-\frac{q}{p}$ and $-\frac{s}{r}$.

\begin{figure}
\begin{center} \includegraphics{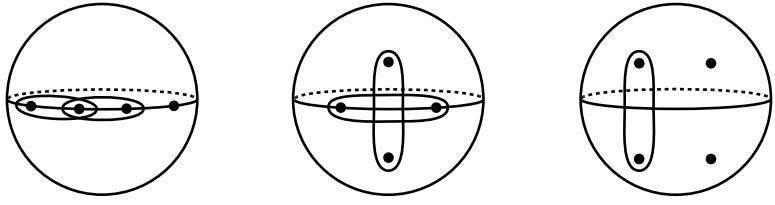}
\caption{ The fixed point set $\zg$ and $P_2$.}
\label{fig:fixrho}
\end{center}
\end{figure}

\begin{thm}\label{thm:rflns} Let $f$ be a NET map in the setting of
Section~\ref{sec:defns}.  Let $\frac{p}{q},\frac{r}{s}\in
\widehat{\bQ}$.  Suppose that $\zs_f(\frac{p}{q})\ne o$,
$\zs_f(\frac{r}{s})\ne o$ and $\zs_f(\frac{p}{q})\ne
\zs_f(\frac{r}{s})$.  Let $\zl=q \zl_2+p \zm_2$ and $\zm=s \zl_2+r
\zm_2$.  Let $d$, respectively $d'$, be the order of the image of
$\zl$, respectively $\zm$, in $\zL_2/\zL_1$.  Suppose that $(\zl,\zm)$
is a basis of $\zL_2$ and that $(d \zl,d' \zm)$ is a basis of $\zL_1$.
Suppose that $q_1^{-1}(p_1^{-1}(P_2))$ is invariant under the
reflection of $\bR^2$ given by $x \zl+y \zm\mapsto (2d-x)\zl+y \zm$,
where $x$ and $y$ are real numbers.  Let $\zr_1$ be the reflection of
$\bH$ about the geodesic whose endpoints are
$-\zs_f(\frac{p}{q})^{-1}$ and $-\zs_f(\frac{r}{s})^{-1}$. Let $\zr_2$
be the reflection of $\bH$ about the geodesic whose endpoints are
$-\frac{q}{p}$ and $-\frac{s}{r}$.  Then
  \begin{equation*}
\zS_f\circ \zr_2=\zr_1\circ \zS_f.
  \end{equation*}
\end{thm}
  \begin{proof} The situation is as in Figure~\ref{fig:degscmpts},
except that now the vertices of the large parallelogram are all
elements of $\zL_1$.  The discussion preceding Theorem~\ref{thm:rflns}
shows that $\zr_2$ is induced by a reflection $r_2$ of $S^2$ about the
simple closed curve which is the image under $p_2\circ q_2$ of the
boundary of the parallelogram whose vertices are 0, $\zl$, $\zm$ and
$\zl+\zm$.  Likewise let $r_1$ be a reflection of $S^2$ about the
image under $p_1\circ q_1$ of the boundary of the parallelogram whose
vertices are 0, $d \zl$, $d' \zm$ and $d \zl+d' \zm$.  In other words,
we may take $r_1$ to be the reflection of $S^2$ induced by the
reflection of $\bR^2$ given by $x \zl+y \zm\mapsto (2d-x)\zl+y \zm$,
where $x$ and $y$ are real numbers.  The assumptions imply that $r_1$
restricts to a homeomorphism of $S^2\setminus P_2$.  It fixes
essential simple closed curves with slopes $\zs_f(\frac{p}{q})$ and
$\zs_f(\frac{r}{s})$, so the Teichm\"{u}ller map which it induces is
$\zr_1$.  It is clear that $r_2\circ f=f\circ r_1$.  Thus $\zS_f\circ
\zr_2=\zr_1\circ \zS_f$.  This proves Theorem~\ref{thm:rflns}.

\end{proof}

\begin{cor}\label{cor:rflns}  In the situation of
Theorem~\ref{thm:rflns}, the Teichm\"{u}ller map $\zS_f$ maps the
geodesic in $\bH$ with endpoints $-\frac{q}{p}$ and $-\frac{s}{r}$ to
the geodesic with endpoints $-\zs_f(\frac{p}{q})^{-1}$ and
$-\zs_f(\frac{r}{s})^{-1}$.
\end{cor}

\section{$\text{Aff($f$)}$ and functional equations}
\label{sec:fnleqns}\nosubsections

Let $f$ be a NET map in the setting of Section~\ref{sec:defns}.  The
results of Sections~\ref{sec:dehn} and \ref{sec:rflns} show that
$\zS_f$ satisfies certain ``functional equations''.  By this we mean
that there are choices of M\"{o}bius transformations $\zv$ and $\zj$
(which might reverse orientation) such that $\zS_f\circ \zv=\zj\circ
\zS_f$.  The aim of this section is to unify and clarify these
results.

Let $i\in \{1,2\}$.  Every homeomorphism from $S^2\setminus P_i$ to
$S^2\setminus P_i$ is isotopic to a homeomorphism which lifts via
$p_i\circ q_i$ to an affine map from $\bR^2$ to $\bR^2$ taking $\zL_i$
bijectively to $\zL_i$.  This is essentially the content of
Proposition 2.7 of \cite{FM} by Farb and Margalit.  This leads us to
define the \emph{affine group} of $f$ to be the set $\text{Aff}(f)$
of all affine transformations of $\bR^2$ which bijectively stabilize
$\zL_1$, $\zL_2$ and $q_1^{-1}(p_1^{-1}(P_2))$.  We define the
\emph{general linear group} $\text{GL}(f)$ and the \emph{special
linear group} $\text{SL}(f)$ of $f$ analogously.  If $\zd\in
\text{Aff}(f)$ and if $\zg$ is a rotation of order 2 in $\zG_1$ which
fixes $\zl\in \zL_1$, then $\zd \zg \zd^{-1}$ is a rotation of order 2
which fixes $\zd(\zl)$.  Since these rotations generate $\zG_1$, it
follows that $\text{Aff}(f)$ normalizes $\zG_1$.  Likewise
$\text{Aff}(f)$ normalizes $\zG_2$.  Because $\text{Aff}(f)$
normalizes $\zG_1$ and $\zG_2$, its action on $\bR^2$ induces actions
on $\bR^2/\zG_1$ and $\bR^2/\zG_2$.  Let $\zd\in \text{Aff}(f)$, let
$\zd_1$ be the map which it induces on $\bR^2/\zG_1$ and let $\zd_2$
be the map which it induces on $\bR^2/\zG_2$.  Since the lift of $f$
to $\bR^2$ is the identity map, this lift commutes with $\zd$.  Thus
$f\circ \zd_1=\zd_2\circ f$.  The assumptions imply that both $\zd_1$
and $\zd_2$ stabilize $P_2$, and so both $\zd_1$ and $\zd_2$ induce
maps on the Teichm\"{u}ller space of $S^2\setminus P_2$.  This leads
to the functional equation $\zS_f\circ \zS_{\zd_2}=\zS_{\zd_1}\circ
\zS_f$.  Of course, both $\zS_{\zd_1}$ and $\zS_{\zd_2}$ are
M\"{o}bius transformations.  We have proved the first assertion of the
following theorem.  The second assertion concerning $\zS_{\zd_1}$ and
$\zS_{\zd_2}$ follows from the discussion between here and
Example~\ref{ex:fnleqna}.

\begin{thm}\label{thm:fnleqn} Let $f$ be a NET map in the setting of
Section~\ref{sec:defns}.  Let $\zd\in \text{Aff}(f)$.  Let $\zd_i$ be
the map which $\zd$ induces on $\bR^2/\zG_i$ for $i\in \{1,2\}$.  Then
$\zS_f\circ \zS_{\zd_2}=\zS_{\zd_1}\circ \zS_f$.  If $\zd_i$
preserves, respectively reverses, orientation, then $\zS_{\zd_i}$ is a
M\"{o}bius transformation from $\text{PSL}(2,\mathbb{Z})$,
respectively $\text{PGL}(2,\mathbb{Z})\setminus
\text{PSL}(2,\mathbb{Z})$, for $i\in \{1,2\}$.
\end{thm}

In this and the next two paragraphs we discuss computations involving
Theorem~\ref{thm:fnleqn}.  Let $f$ be a NET map in the setting of
Section~\ref{sec:defns}.  Let $\zd\in \text{Aff}(f)$.  We want
explicit forms for the M\"{o}bius transformations $\zd_1$ and $\zd_2$.
For this it suffices to understand how $\zd_1$ and $\zd_2$ act on the
boundary of $\bH$.

We first consider $\zd_2$.  We express $\zd$ in terms of the chosen
basis $(\zl_2,\zm_2)$ of $\zL_2$, and we calculate slopes of lines in
$\bR^2$ with respect to this basis.  Let $p$, $q$, $p'$, $q'$ be
integers as usual so that $\zd^{-1}$ maps lines in $\bR^2$ with slope
$\frac{p}{q}$ to lines with slope $\frac{p'}{q'}$.  Let $\zt\in \bH$,
and suppose that $\zS_{\zd_2}(\zt)=\zt'$.  Then
$\text{mod}_\zt(\frac{p}{q})=\text{mod}_{\zt'}(\frac{p'}{q'})$.  Using
Lemma~\ref{lemma:mod} and Corollary~\ref{cor:horocycle}, we see that
$\zS_{\zd_2}(-\frac{q}{p})=-\frac{q'}{p'}$.  If $\zd$ preserves
orientation and if $\left[\begin{smallmatrix}a & b \\ c &
d \end{smallmatrix}\right]$ is the matrix of the linear part of $\zd$,
then the inverse of the M\"{o}bius transformation $z\mapsto
\frac{az+b}{cz+d}$ maps $\frac{q}{p}$ to $\frac{q'}{p'}$.  Combining
this with the previous statement yields that
$\zS_{\zd_2}(z)=\frac{dz+b}{cz+a}$.  If $\zd$ reverses orientation,
then $\zS_{\zd_2}(z)=\frac{d \overline{z}+b}{c \overline{z}+a}$.  This
determines $\zS_{\zd_2}$ in terms of $\zd$.  In particular,
$\zS_{\zd_2}$ comes from $\text{PGL}(2,\mathbb{Z})$.

Whereas computing $\zS_{\zd_2}$ is easy, computing $\zS_{\zd_1}$ is
usually not, although see Examples~\ref{ex:fnleqna} and
\ref{ex:fnleqnb} for two easy special cases.  Its computation has much
in common with the computation of $\zs_f$ using spin mirrors.  In
general, we begin with an essential simple closed curve $\zg$ in
$S^2\setminus P_2$.  We lift it to $\bR^2$ using $p_1\circ q_1$.
Although we may assume that this lift is piecewise linear, we may not
assume that it is a line segment.  We wish to understand how $\zd_1$
acts on slopes, so we apply $\zd$ to this lift.  We then use spin
mirrors as in Section~\ref{sec:slopefn} to compute the slope of
$\zd_1(\zg)$.  Doing this for curves $\zg$ with slopes 0 and $\infty$
obtains two corresponding slopes $\frac{c}{a}$ and $\frac{d}{b}$,
where $(a,c)$ and $(b,d)$ form a basis of $\mathbb{Z}^2$.  We multiply
one of these vectors by $-1$ if necessary so that
$\left|\begin{smallmatrix}a & b \\ c & d \end{smallmatrix}\right|=1$.
If $\zd$ preserves orientation, then its action on slopes is given by
the matrix $\left[\begin{smallmatrix}a & b \\ c & d
\end{smallmatrix}\right]$.  If $\zd$ reverses orientation, then the
matrix is $\left[\begin{smallmatrix}a & -b \\ c & -d
\end{smallmatrix}\right]$.  As for $\zd_2$, it follows that
$\zS_{\zd_1}(z)=\frac{dz+b}{cz+a}$ if orientation is preserved and
$\zS_{\zd_1}(z)=-\frac{d \overline{z}+b}{c \overline{z}+a}$ if
orientation is reversed.  In particular, $\zS_{\zd_1}$ comes from
$\text{PGL}(2,\mathbb{Z})$.

In the following examples, we apply Theorem~\ref{thm:fnleqn} to our
main example of a NET map.

\begin{ex}\label{ex:fnleqna} For our main example of a NET map, we
have that $\zL_2=\bZ^2$ and $\zL_1=\left<(2,-1),(0,5)\right>$.
Furthermore, $q_1^{-1}(p_1^{-1}(P_2))$ consists of the six cosets of
$2\zL_1$ represented by $(0,0)$, $(0,5)$, $(2,0)$, $(2,-2)$, $(2,3)$
and $(2,5)$.  Let $\zd\co \bR^2\to \bR^2$ be the linear map with
matrix $\left[\begin{smallmatrix}1 & 0 \\ 5 & 1
\end{smallmatrix}\right]$ with respect to the standard basis.  Then
$\zd\in \text{SL}(f)$.  The discussion following
Theorem~\ref{thm:fnleqn} shows that $\zS_{\zd_2}(z)=\frac{z}{5z+1}$.
Let $B_1$ be the set of usual spin mirrors relative to $p_1\circ q_1$
for the main example.  Because $\zd$ fixes $(0,1)$ and stabilizes
every vertical line, it stabilizes $B_1$.  Hence $\zd$ restricts to a
homeomorphism of $\bR^2\setminus B_1$ to itself.  It follows that the
matrix which $\zd_1$ determines for its action on slopes of essential
curves in $S^2\setminus P_2$ relative to the basis $(\zl_2,\zm_2)$ is
the same matrix which $\zd_1$ determines for its action on $\zL_1$
relative to the basis $(\zl_1,\zm_1)$.  This matrix is
$\left[\begin{smallmatrix}1 & 0 \\ 2 & 1 \end{smallmatrix}\right]$.
We conclude that
  \begin{equation*}
\zS_f\left(\frac{z}{5z+1}\right)=\frac{\zS_f(z)}{2 \zS_f(z)+1}.
  \end{equation*}
The action of $\zS_f$ and $\zs_f$ are related by conjugation by
$z\mapsto -\frac{1}{z}$.  Hence $\zs_f(z-5)=\zs_f(z)-2$, and so
$\zs_f(z+5)=\zs_f(z)+2$.  The square of $\zd_1$ is a Dehn twist about
an essential simple closed curve in $S^2\setminus P_2$ with slope
$\infty$.  Compare this functional equation with that of
Example~\ref{ex:fnleqns}.
\end{ex}

\begin{ex}\label{ex:fnleqnb}  We continue with the main
example.  Let $\zd\co \bR^2\to \bR^2$ be the linear map with matrix
$\left[\begin{smallmatrix}-1 & 0 \\ 1 & 1 \end{smallmatrix}\right]$.
Then $\zd(2,-1)=(-2,1)$ and $\zd(0,1)=(0,1)$.  So $\zd$ reverses
orientation and it fixes the lines generated by $(2,-1)$ and $(0,1)$.
It stabilizes $\zL_1$ and $\zL_2$.  Furthermore a glance at
Figure~\ref{fig:maintess} shows that it stabilizes
$q_1^{-1}(p_1^{-1}(P_2))$.  So $\zd\in \text{GL}(f)$.  We have that
$\zS_{\zd_2}(z)=\frac{\overline{z}}{\overline{z}-1}$.  The map on
slopes of essential simple closed curves in $S^2\setminus P_2$ induced
by $\zd_1$ fixes 0 and $\infty$ and takes 1 to $-1$.  Hence
$\zS_{\zd_1}(z)=-\overline{z}$.  We conclude that
  \begin{equation*}
\zS_f\left(\frac{\overline{z}}{\overline{z}-1}\right)=-\overline{\zS}_f(z),
  \end{equation*}
and
  \begin{equation*}
\zs_f(-z-1)=-\zs_f(z).
  \end{equation*}
\end{ex}

\begin{remark}\label{remark:graph} We return to the graph of the slope
function $\zs_f$ for the main example which appears at the end of
Section~\ref{sec:slopefn}.  A prominent fuzzy line appears with
positive slope less than 1.  Suppose that this line is given by
$y=mx+b$.  Example~\ref{ex:fnleqna} implies that
$\zs_f(x+5)=\zs_f(x)+2$.  The line given by $y=mx+b$ should also
satisfy this functional equation, which implies that
$m=\frac{2}{5}$. Example~\ref{ex:fnleqnb} implies that
$\zs_f(-x-1)=-\zs_f(x)$.  The line should also satisfy this
functional equation, and so $b=\frac{1}{5}$.  So the special line
with positive slope less than 1 is given by
$y=\frac{2}{5}x+\frac{1}{5}$.
\end{remark}

\section{Constant Teichm\"{u}ller maps }\label{sec:constant}
\nosubsections

This section deals with NET maps whose associated Teichm\"{u}ller maps
are constant.  We aim to find an algebraic formulation of what it
means for the Teichm\"{u}ller map of a NET map to be constant.  The
results of this section are extended in Saenz Maldonado's thesis
\cite{SM}.  We begin with the following lemma.

\begin{lemma}\label{lemma:twogenr}  Let $\zv\co \bZ^2\to A$
be a surjective group homomorphism from $\bZ^2$ to a finite Abelian
group $A$.  Let $a\in A$, and let $B$ be a cyclic subgroup of $A$.
Then the quotient group $A/B$ is cyclic and the image of $a$ in $A/B$
generates $A/B$ if and only if there exists a basis of $\bZ^2$
consisting of elements $\za$ and $\za'$ with $\zv(\za)=a$ and
$\zv(\za')$ a generator of $B$.
\end{lemma}
  \begin{proof} It is clear that if $\za$ and $\za'$ form a basis of
$\bZ^2$ with $\zv(\za)=a$ and $\zv(\za')$ a generator of $B$, then
$A/B$ is cyclic and the image of $a$ in $A/B$ generates $A/B$.

To prove the converse, suppose that $A/B$ is cyclic and that the image
of $a$ in $A/B$ generates $A/B$.  Since the kernel $K$ of $\zv$ has
finite index in $\bZ^2$, there exists a basis of $\bZ^2$ consisting of
elements $\zb_1$ and $\zb_2$ and positive integers $m$ and $n$ with
$m| n$ such that $m \zb_1$ and $n \zb_2$ form a basis of $K$.
Without loss of generality we assume that $\zb_1=(1,0)$ and
$\zb_2=(0,1)$, so that $(m,0)$ and $(0,n)$ form a basis of $K$.  Hence
$A\cong \bZ/m \bZ\oplus \bZ/n \bZ$.  Let $C=\bZ/n \bZ\oplus \bZ/n
\bZ$.  Let $\zv_1\co \bZ^2\to C$ be the canonical group homomorphism.
By the standard homomorphism theorems of group theory, there exists a
group homomorphism $\zv_2\co C\to A$ such that $\zv=\zv_2\circ \zv_1$.
Let $c$ and $c'$ be elements of $C$ such that $\zv_2(c)=a$ and the
image of $\zv_2(c')$ in $A/B$ generates $A/B$.  Suppose that $c$ and
$c'$ are given in coordinates by $c=(c_1,c_2)$ and $c'=(c'_1,c'_2)$.
The determinant $d=\left|\begin{smallmatrix} c_1& c'_1 \\ c_2 &
c'_2 \end{smallmatrix}\right|$ is then an element of $\bZ/n \bZ$.  We
aim to prove that it is possible to choose $c$ and $c'$ so that $d=1$.

To begin this, let $p$ be a prime such that $p|  m$.  The
assumptions imply that $C/pC\cong A/pA\cong \bZ/p \bZ\oplus \bZ/p \bZ$
and that the images of $c$ and $c'$ in $A/pA$ generate $A/pA$.  So the
images of $c$ and $c'$ in $C/pC$ are linearly independent elements of
this vector space.  Thus the image of $d$ in $\bZ/p \bZ$ is not 0.

Now let $p$ be a prime such that $p| n$ but $p\nmid m$.  The
assumptions imply that $C/pC\cong \bZ/p \bZ\oplus \bZ/p \bZ$, that
$A/pA\cong \bZ/p \bZ$ and that the images of $c$ and $c'$ in $A/pA$
generate $A/pA$.  In this case the images of $c$ and $c'$ in $C/pC$
generate a nonzero subspace of $C/pC$, the image of the kernel of
$\zv_2$ is a 1-dimensional subspace of $C/pC$ and $C/pC$ is the sum of
these two subspaces.  So we may modify $c$ and $c'$ by elements of the
kernel of $\zv_2$ if necessary to make the images of $c$ and $c'$ in
$C/pC$ linearly independent.  Once this is done, the image of $d$ in
$\bZ/p \bZ$ is not 0.

We use the Chinese remainder theorem to modify $c$ and $c'$ as in the
previous paragraph for every prime $p$ such that $p| n$ but $p\nmid
m$. Once this is done, $d$ is a unit modulo $n$.  Now we replace $c'$
by $d^{-1}c'$.  It is still true that $\zv_2(b)=a$ and $\zv_2(c')$
generates $B$.  Furthermore the new determinant is 1.  Now we use the
fact that the canonical group homomorphism from $\text{SL}(2,\bZ)$ to
$\text{SL}(2,\bZ/n \bZ)$ is surjective.  So there exists a basis of
$\bZ^2$ consisting of elements $\za$ and $\za'$ with $\zv_1(\za)=c$
and $\zv_1(\za')=c'$.  Hence $\zv(\za)=a$ and $\zv(\za')$ generates
$B$.

This proves Lemma~\ref{lemma:twogenr}.

\end{proof}

Now we find an algebraic formulation of what it means for the
Teichm\"{u}ller map of a NET map to be constant.  Let $f$ be a NET map
in the setting of Section~\ref{sec:defns}.  Combining statements 1 and
4 of Theorem 5.1 of \cite{BEKP} implies that the Teichm\"{u}ller map
of $f$ is constant if and only if for every essential, nonperipheral
simple closed curve $\zd$ in $S^2\setminus P_2$ every connected
component of $f^{-1}(\zd)$ is either null or peripheral.  This and
Theorem~\ref{thm:degscmpts} lead us to consider the following.  Let
$\zl$ and $\zm$ be elements of $\zL_2$ which form a basis of $\zL_2$.
Let $c_1$, $c_2$, $c_3$, $c_4$ be the coset numbers for
$q_1^{-1}(p_1^{-1}(P_2))$ relative to $\zl$ and $\zm$.  Statement 2 of
Theorem~\ref{thm:degscmpts} now shows that the Teichm\"{u}ller map of
$f$ is constant if and only if $c_2=c_3$ for every choice of $\zl$ and
$\zm$.

This leads us to make a definition.  Let $A$ be a finite Abelian
group.  We say that a subset $H$ of $A$ is \emph{nonseparating} if
and only if it satisfies the following conditions.  First, $H$ is a
disjoint union of the form $H=\{\pm h_1\}\amalg\{\pm h_2\}\amalg\{\pm
h_3\}\amalg\{\pm h _4\}$.  (It is possible that $h_i=-h_i$.)  Let $B$
be a cyclic subgroup of $A$ such that $A/B$ is cyclic.  Let $c_1$,
$c_2$, $c_3$, $c_4$ be the coset numbers for $H$ relative to $B$ and
some generator of $A/B$.  The main condition is that $c_2=c_3$ for
every such choice of $B$ and generator of $A/B$.  We say that $H$ is
nonseparating because it never separates $c_2$ from $c_3$.
Lemma~\ref{lemma:twogenr} and the intervening discussion yield the
following theorem.

\begin{thm}\label{thm:algcformn} Let $f$ be a NET map in the setting
of Section~\ref{sec:defns}.  Then the Teichm\"{u}ller map of $f$ is
constant if and only if $p_1^{-1}(P_2)$ is a nonseparating subset of
$\zL_2/2 \zL_1$.
\end{thm}

Theorem~\ref{thm:algcformn} provides a strategy for constructing NET
maps whose Teichm\"{u}ller maps are constant.  The first step, and
only step which is not straightforward, is to construct a finite
Abelian group $A$ generated by two elements with $A/2A\cong \bZ/2
\bZ\oplus \bZ/2 \bZ$ such that $A$ has a nonseparating subset $H$.  We
then construct lattices $\zL_1\subseteq \zL_2\subseteq \bR^2$ such
that $\zL_2/2 \zL_1\cong A$ and use this isomorphism to identify
$\zL_2/2 \zL_2$ with $A$.  We construct an isomorphism from $\zL_2$ to
$\zL_1$, which in effect constructs a Euclidean Thurston map $g$
corresponding to $\zL_1$ and $\zL_2$.  We then construct an
orientation-preserving homeomorphism $h\co S^2\to S^2$ such that
$h(P_2)=p_1(H)$.  As in Section~\ref{sec:twists}, it follows
that $f=h\circ g$ is a NET map if it has four postcritical points
(which it usually does), and Theorem~\ref{thm:algcformn} shows that
its Teichm\"{u}ller map is constant.  Since
$\left|\zL_2/\zL_1\right|=\deg(f)$, we have that $\left|A\right|=4
\deg(f)$.

Examples~\ref{ex:degtwo} and \ref{ex:degnine} give examples of finite
Abelian groups with nonseparating subsets.

\begin{ex}\label{ex:degtwo}  In this example we consider
$A=\bZ/4 \bZ\oplus \bZ/2 \bZ$.  We show that $H=\{(0,0),\pm
(1,0),(2,0),\pm (1,1)\}$ is a nonseparating subset of $A$.  Let $B$ be
a cyclic subgroup of $A$ such that $A/B$ is cyclic.  Then either
$\left|B\right|=4$ or $\left|B\right|=2$.  If $\left|B\right|=4$, then
either $B=\left<(1,0)\right>$ or $B=\left<(1,1)\right>$.  One verifies
in these cases that $c_1=c_2=c_3=0$ and $c_4=1$, and so $c_2=c_3$, as
desired.  If $\left|B\right|=2$, then either $B=\left<(0,1)\right>$ or
$B=\left<(2,1)\right>$.  One verifies in these cases that $c_1=0$,
$c_2=c_3=1$ and $c_4=2$.  Thus $H$ is a nonseparating subset of $A$.

\end{ex}

\begin{ex}\label{ex:degnine}  In this example we show that
the set $H$ of elements of order 3 is a nonseparating subset of
$A=\bZ/6 \bZ\oplus \bZ/6 \bZ$.  The 3-torsion subgroup of $A$ is
isomorphic to $\bZ/3 \bZ\oplus \bZ/3 \bZ$.  It has eight elements of
order 3, which are paired by inversion.  So the set $H$ contains four
pairs of elements which are mutually inverse.  To verify that $H$ is a
nonseparating subset for $A$, let $B$ be a cyclic subgroup of $A$ such
that $A/B$ is cyclic.  Then $B\cong A/B\cong \bZ/6 \bZ$.  Elements of
order 3 in $A$ map to elements of order either 1 or 3 in $A/B$.  Both
$B$ and $A/B$ contain exactly one pair of mutually inverse elements of
order 3.  It follows that $c_1=0$ and $c_2=c_3=c_4=2$.  Hence
$c_2=c_3$, and so $H$ is a nonseparating subset of $A$.

\end{ex}

The next lemma provides a simple but limited way to produce
nonseparating subsets from known ones.

\begin{lemma}\label{lemma:translate}  Let $A$ be a finite
Abelian group, and let $H=\{\pm h_1,\pm h_2,\pm h_3,\pm h_4\}$
be a nonseparating subset of $A$.  Let $h$ be an element of order 2
in $A$, and let $H'=H+h=\{\pm (h_1+h),\pm (h_2+h),\pm (h_2+h),\pm
(h_4+h)\}$.  Then $H'$ is a nonseparating subset of $A$.
\end{lemma}
  \begin{proof} Let $B$ be a cyclic subgroup of $A$ such that $A/B$ is
cyclic.  Let $c_1$, $c_2$, $c_3$, $c_4$ be the coset numbers for $H$
relative to $B$ and a generator of $A/B$.  Let $c'_1$, $c'_2$,
$c'_3$, $c'_4$ be the corresponding coset numbers for $H'$.  If $h\in
B$, then $c'_k=c_k$ for $k\in \{1,2,3,4\}$.  If $h\notin B$, then the
order of $A/B$ is even, say, $2m$.  In this case $c'_k=m-c_k$ for
$k\in \{1,2,3,4\}$.  It follows that if $c_2=c_3$, then $c'_2=c'_3$.
This proves Lemma~\ref{lemma:translate}.

\end{proof}

The next lemma shows that a nonseparating subset for a subgroup is a
nonseparating subset for the group.

\begin{lemma}\label{lemma:subgroup}  If $A$ is a finite
Abelian group and if $A'$ is a subgroup of $A$, then every subset of
$A'$ which is nonseparating for $A'$ is nonseparating for $A$.
\end{lemma}
  \begin{proof} Let $A$ be a finite Abelian group, let $A'$ be a
subgroup of $A$ and suppose that $H$ is a subset of $A'$ which is
nonseparating for $A'$.  Let $B$ be a cyclic subgroup of $A$ such that
$A/B$ is cyclic.  Then both $A'\cap B$ and $A'/A'\cap B$ are cyclic.
Choosing a generator for $A/B$ determines an ordering of the cosets of
$B$ in $A$.  This determines an ordering of the cosets of $A'\cap B$ in
$A'$, and this ordering of the cosets of $A'\cap B$ is determined by a
generator of $A'/A'\cap B$.  It follows from this and the definition
that $H$ is nonseparating for $A$.  This proves
Lemma~\ref{lemma:subgroup}.
\end{proof}

\begin{ex}\label{ex:double} In this example we construct a rational
function with degree 4 which is a NET map whose Teichm\"{u}ller map is
constant.  We follow the strategy outlined immediately after
Theorem~\ref{thm:algcformn}.

We first construct an Abelian group $A$ with order $4\cdot 4=16$ which
has a nonseparating subset.  We take $A=\bZ/4 \bZ\oplus \bZ/4 \bZ$.
The group $A$ contains the subgroup $A'=\left<(1,0),(0,2)\right>$,
which is isomorphic to $\bZ/4 \bZ\oplus \bZ/2 \bZ$.
Example~\ref{ex:degtwo} implies that $H=\{(0,0),\pm (1,0),(2,0),\pm
(1,2)\}$ is a nonseparating subset of $A'$.  Hence
Lemma~\ref{lemma:subgroup} implies that $H$ is a nonseparating subset
of $A$.

Now we identify $\bR^2$ with $\bC$, and we let $\zL_2$ be any lattice
in $\bC$.  We let $\zL_1=2 \zL_2$, so that $\zL_2/2 \zL_1\cong A$.
For an isomorphism from $\zL_2$ to $\zL_1$, we choose the map
$z\mapsto 2z$.  The lattices $\zL_1$, $\zL_2$ and the map $z\mapsto
2z$ determine a Latt\`{e}s map $g\co \widehat{\bC}\to \widehat{\bC}$
from the Riemann sphere to itself up to analytic conjugation.

In this paragraph we identify $g$.  Let $\wp $ be the Weierstrass
function with group of periods $2 \zL_1$.  Figure~\ref{fig:double}
shows a fundamental domain for the action of $\zG_1$ on $\bC$.  The
dots are elements of $\zL_2$.  The lower left corner is 0.  Points are
labeled by their images in $\widehat{\bC}$ under $\wp $.  For our
usual branched covers $p_1\circ q_1\co \bC\to S^2$ and $p_2\circ
q_2\co \bC\to S^2$, we take $p_1(q_1(z))=\wp (z)$ and $p_2(q_2(z))=\wp
(2z)$.  It has long been known that $\wp (2z)$ is a rational function
of $\wp (z)$.  The map $g$ is this rational function.  Thus
$g(e_1)=g(e_2)=g(e_3)=g(\infty)=\infty$ and $g(E_k)=g(E'_k)=e_k$ for
$k\in \{1,2,3\}$.  The critical points of $g$ are $E_1$, $E'_1$,
$E_2$, $E'_2$, $E_3$, $E'_3$.  The postcritical set of $g$ is
$\{e_1,e_2,e_3,\infty\}$.  According to line 3.41 of \cite{W},
  \begin{equation*}
g(z)=\frac{(z^2-s_2)^2+8s_3z}{4(z^3+s_2z-s_3)},
  \end{equation*}
where $s_2=e_1e_2+e_1e_3+e_2e_3$ and $s_3=e_1e_2e_3$.  The only
restriction on $e_1$, $e_2$ and $e_3$ is that they are distinct and
that $e_1+e_2+e_3=0$.

\begin{figure}\begin{center}
\includegraphics{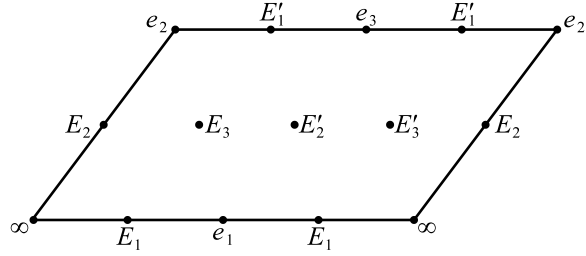} \caption{ A fundamental domain
for $\zG_1$.}
\label{fig:double}
\end{center}\end{figure}

We identify $A$ with a subgroup of $\bC/2 \zL_1$ so that the images in
$\widehat{\bC}$ of the elements $(0,0)$, $(1,0)$, $(2,0)$, $(1,2)$ of
$S$ are $\infty$, $E_1$, $e_1$, $E'_1$ in order.  We let $h\co
\widehat{\bC}\to \widehat{\bC}$ be an orientation-preserving
homeomorphism such that $h(e_1)=E'_1$, $h(e_2)=E_1$, $h(e_3)=e_1$ and
$h(\infty)=\infty$, so that $h$ takes the postcritical points of $g$ to
the points of $\widehat{\bC}$ corresponding to $H$.

Let $f=h\circ g$.  Then $f$ is a Thurston map.  Its critical points
are $E_1$, $E'_1$, $E_2$, $E'_2$, $E_3$, $E'_3$.  These points are
taken by $f$ to $E'_ 1$, $E_1$, $e_1$.  Moreover, $f(e_1)=\infty$.
Now we see that $f$ has exactly four postcritical points, and so it is
a NET map whose Teichm\"{u}ller map is constant.

To construct a specific rational map equivalent to $g$, we choose the
simple case in which $e_1=1$, $e_2=-1$ and $e_3=0$.  Then $s_2=-1$ and
$s_3=0$.  Hence
  \begin{equation*}
g(z)=\frac{(z^2+1)^2}{4(z^3-z)}.
  \end{equation*}
To find $E_1$ and $E'_1$, we solve the equation $g(z)=e_1=1$.  This
leads to the equation $(z^2+1)^2=4(z^3-z)$, then
$z^4-4z^3+2z^2+4z+1=0$ and finally $(z^2-2z-1)^2=0$.  We take
$E_1=1+\sqrt{2}$ and $E'_1=1-\sqrt{2}$.  Because $g$ is an odd
function, we may take $E_2=-1-\sqrt{2}$ and $E'_2=-1+\sqrt{2}$.  Set
$h(z)=-\sqrt{2}z+1$.  One verifies that $h(e_1)=E'_1$, $h(e_2)=E_1$,
$h(e_3)=e_1$ and $h(\infty)=\infty$.  Thus $f=h\circ g$ is a rational
function which is a NET map whose Teichm\"{u}ller map is constant.

Rather than explicitly calculating $f$, we explicitly calculate an
analytic conjugate of $f$ which is simpler.  For this, we observe that
there exists a unique orientation-preserving isometry $k$ of the
hyperbolic plane $\mathbb{H}$ which rotates about $i$ through the
angle $5 \zp/4$.  See Figure~\ref{fig:rotation}, where eight
hyperbolic sectors with angle $\zp/4$ are drawn at $i$.  The map $k$
has order 8 and stabilizes the set $\{0,\pm 1,\pm 1\pm
\sqrt{2},\infty\}$.  It acts on these points as follows.
  \begin{equation*}
0\mapsto -1-\sqrt{2}\mapsto 1\mapsto 1-\sqrt{2}\mapsto \infty
\mapsto -1+\sqrt{2}\mapsto -1\mapsto 1+\sqrt{2}\mapsto 0
  \end{equation*}

\begin{figure}\begin{center}
\includegraphics{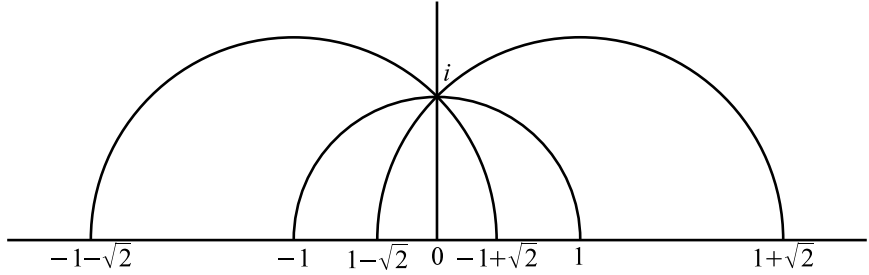} \caption{ Understanding the
map $k$.}
\label{fig:rotation}
\end{center}\end{figure}

Set $F= k\circ f\circ k^{-1}$.  One easily verifies that $F$ is a
rational function with degree 4 such that 0, $\pm 1$, $\infty$ are all
critical points of $F$ with $F(0)=F(\infty)=\infty$ and
$F(1)=F(-1)=0$.  Thus $F(z)=K(z-z^{-1})^2$ for some complex number
$K$.  To determine $K$, we note that $F(\sqrt{2}-1)=\sqrt{2}-1$.
Hence $K(\sqrt{2}-1-(\sqrt{2}+1))^2=\sqrt{2}-1$, and so
$K=\frac{\sqrt{2}-1}{4}$.  We conclude that
$F(z)=\frac{\sqrt{2}-1}{4}(z-z^{-1})^2$ is a NET map whose
Teichm\"{u}ller map is constant.

In this paragraph we find a functional equation which proves that
$\zS_F(z)=\zS_f(z)=i$ for every $z\in \bH$. This not only gives
another proof that the Teichm\"{u}ller map of $F$ is constant but also
determines the value of the constant.  We choose a basis
$(\zl_2,\zm_2)$ of $\zL_2$ in the straightforward way relative to the
fundamental domain in Figure~\ref{fig:double}, so that $\wp
(\zl_2)=E_1$ and $\wp (\zm_2)=E_2$.  The six $2 \zL_1$-cosets of
$q_1^{-1}(p_1^{-1}(P_2))=\wp ^{-1}(P_2)$ in $\zL_2$ are represented by
0, $\zl_2$, $2 \zl_2$, $3 \zl_2$, $\zl_2+2 \zm_2$ and $3 \zl_2+2
\zm_2$.  Let $\zd\co \bC\to \bC$ be the map defined by $\zd(z)=z+2
\zl_2$.  Then $\zd\in \text{Aff}(f)$.  Since $\zd\in \zG_2$, the map
$\zd_2$ from Section~\ref{sec:fnleqns} is the identity map, and so
$\zS_{\zd_2}$ is the identity map.  Similarly, $\zd_1^2$ and
$\zS_{\zd_1}^2$ are identity maps.  To better understand $\zd_1$, we
note that the closed interval $[e_2,e_3]=[-1,0]$ is a core arc for an
essential simple closed curve in $\widehat{\bC}\setminus \{0,\pm
1,\infty\}$ with slope 0.  Hence
$h([e_2,e_3])=[e_1,E_1]=[1,1+\sqrt{2}]$ is a core arc for an essential
simple closed curve in $\widehat{\bC}\setminus h(\{0,\pm 1,\infty\})$
with slope 0.  Similarly, $[\infty,e_2]=[-\infty,-1]$ is a core arc
for an essential simple closed curve in $\widehat{\bC}\setminus
\{0,\pm 1,\infty\}$ with slope $\infty$, and
$h([\infty,e_2])=[E_1,\infty]=[1+\sqrt{2},\infty]$ is a core arc for
an essential simple closed curve in $\widehat{\bC}\setminus h(\{0,\pm
1,\infty\})$ with slope $\infty$.  From this we see that the map which
$\zd_1$ induces on slopes relative to $\widehat{\bC}\setminus
h(\{0,\pm 1,\infty\})$ interchanges 0 and $\infty$.  So $\zS_{\zd_1}$
is an involution which interchanges 0 and $\infty$.
Theorem~\ref{thm:fnleqn} implies that $\zS_{\zd_1}$ comes from
$\text{PSL}(2,\mathbb{Z})$.  Hence $\zS_{\zd_1}(z)=-\frac{1}{z}$.
Because $\zS_f\circ \zS_{\zd_2}=\zS_{\zd_1}\circ \zS_f$, we obtain
that $\zS_f(z)=-\frac{1}{\zS_f(z)}$ for every $z\in \bH$.  Hence
$\zS_f(z)=i$ for every $z\in \bH$.  Similarly, $\zS_F(z)=i$ for every
$z\in \bH$.

This example illustrates the discussion at the beginning of
Section~\ref{sec:examples} concerning subdivision maps of finite
subdivision rules.  What follows is a brief description of this.  Let
$G=k\circ g\circ k^{-1}$ and $H=k\circ h\circ k^{-1}$ (not to be
confused with the nonseparating subset of $A$).  Then $G$ is a
Euclidean Thurston map, $H$ is an orientation-preserving homeomorphism
and $F=H\circ G$.  The critical points of $F$ and $G$ are 0, $\pm 1$,
$\pm i$ and $\infty$.  The postcritical set of $G$ is the image under
$k$ of the postcritical set of $g$, namely,
$k(\{e_1,e_2,e_3,\infty\})=\{\pm 1\pm \sqrt{2}\}$.  Including the
action of $k$, the image in $\widehat{\bC}$ of the boundary of the
fundamental domain in Figure~\ref{fig:double} is
$(-\infty,-1-\sqrt{2})\cup (1-\sqrt{2},\infty)\cup \{\infty\}$.  The
map $G$ maps this set into itself.  Thus this set together with
vertices at $\{\pm 1\pm \sqrt{2}\}$ make $\widehat{\bC}$ into a
2-complex for which $G$ is the subdivision map of a finite subdivision
rule.  The 1-skeleton of the first subdivision of $\widehat{\bC}$ is
the union of the real line, the unit circle and $\{\infty\}$.
Figure~\ref{fig:doublex} shows the 1-skeleton of the original complex
drawn with solid line segments and the remaining 1-skeleton of the
subdivision drawn with dashes.  Combinatorially, the tiles of this
finite subdivision rule are essentially squares which are subdivided
into four squares in the straightforward way.  The map $H$ maps
$(\infty,-1-\sqrt{2})\cup (1-\sqrt{2},\infty)\cup \{\infty\}$ to
$(1-\sqrt{2},\infty)\cup \{\infty\}$.  Furthermore
$H(1-\sqrt{2})=\infty$, $H(1+\sqrt{2})=0$, $H(-1-\sqrt{2})=1-\sqrt{2}$
and $H(-1+\sqrt{2})=-1+\sqrt{2}$.  Thus the set
$(1-\sqrt{2},\infty)\cup \{\infty\}$ together with vertices at
$\{0,\pm (1-\sqrt{2}),\infty\}$ make $\widehat{\bC}$ into a 2-complex
for which $F$ is the subdivision map of a finite subdivision rule.
The 1-skeleton of its first subdivision is the same as that for $G$.
Although the first subdivisions of these finite subdivision rules are
identical, these finite subdivision rules are quite different.  The
first has bounded valence, while the second does not.

\begin{figure}\begin{center}
\includegraphics{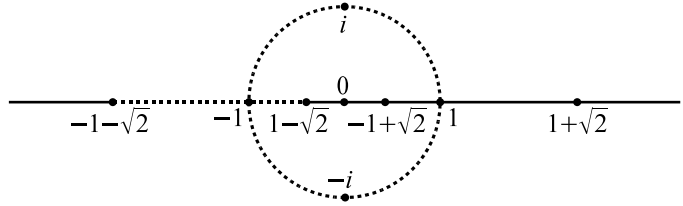} \caption{ The first subdivision
of the subdivision complexes for $F$ and $G$.}
\label{fig:doublex}
\end{center}\end{figure}

This concludes Example~\ref{ex:double}.

\end{ex}

We next prove a general existence theorem.

\begin{thm}\label{thm:twonine} If $d$ is an integer with $d>2$ such
that $d$ is divisible by either 2 or 9, then there exists a NET map
with degree $d$ whose Teichm\"{u}ller map is constant.
\end{thm}
  \begin{proof} Let $d$ be an integer such that $d>2$ and $d$ is
divisible by either 2 or 9.  We prove that there exists a NET map with
degree $d$ and constant Teichm\"{u}ller map.  Example~\ref{ex:double}
provides a degree 4 NET map with a constant Teichm\"{u}ller map.
Hence we may assume that $d>4$.

First suppose that $d>4$ and $2| d$.  Let $A=\bZ/2d \bZ\oplus \bZ/2
\bZ$.  Then $A$ contains a subgroup $A'$ isomorphic to $\bZ/4
\bZ\oplus \bZ/2 \bZ$.  Example~\ref{ex:degtwo} shows that $A'$ has a
nonseparating subset.  Hence Lemma~\ref{lemma:subgroup} implies that
$A$ has a nonseparating subset.  As in the discussion following
Theorem~\ref{thm:algcformn}, we construct a degree $d$ Thurston map
$f$ which is nearly Euclidean with a constant Teichm\"{u}ller map if
its postcritical set contains at least four points.  But since $d>4$,
the discussion in Section~\ref{sec:twists} shows that $f$ does have at
least four postcritical points.  This proves Theorem~\ref{thm:twonine}
if $2| d$.  If $9| d$, then we argue in the same way using
Example~\ref{ex:degnine} with $A=\bZ/2d' \bZ\oplus \bZ/6 \bZ$, where
$d'=d/3$.

This proves Theorem~\ref{thm:twonine}.

\end{proof}

\begin{remark}\label{remark:reducible}  In this remark we
present another view of Theorem~\ref{thm:twonine}.  We begin with some
general observations.

Instead of having two lattices as usual, suppose that we have three
lattices $\zL_1\subseteq \zL_2\subseteq \zL_3$.  We also have
corresponding maps $p_i$ and $q_i$, and $S^2$ is identified with
$p_i\circ q_i(\mathbb{R}^2)$ so that the four-element sets $p_i\circ
q_i(\zL_i)$ are equal.  Suppose also that $h\co S^2\to S^2$ is an
orientation-preserving homeomorphism which maps this four-element
subset of $S^2$ to $p_1\circ q_1(\zL_2)$.

Let $f\co S^2\to S^2$ be the map induced by $p_1\circ q_1$ and
$p_2\circ q_2$, so that $f\circ p_1\circ q_1=p_2\circ q_2$.
Similarly, let $g\co S^2\to S^2$ be the map for which $g\circ p_2\circ
q_2=p_3\circ q_3$.  Then $f$ and $g$ are Euclidean Thurston maps.  The
map $F=h\circ f$ is a Thurston map, and it is nearly Euclidean if it
has at least four postcritical points.  The map $G=h\circ g\circ
h^{-1}$ is a Euclidean Thurston map, being topologically conjugate to
the Euclidean Thurston map $g$.  The map $E=h\circ g\circ f=G\circ F$
is a NET map if it has at least four postcritical points.  The
postcritical set of $F$ is contained in $h(P)$.  The postcritical set
of $G$ is exactly $h(P)$.  It follows that if $\zS_E$, $\zS_F$ and
$\zS_G$ are the Teichm\"{u}ller maps of $E$, $F$ and $G$, then
$\zS_E=\zS_F\circ \zS_G$.  This shows that if $\zS_F$ is constant,
then $\zS_E$ is constant.  Moreover, since $G$ is a Euclidean Thurston
map, $\zS_G$ is a linear fractional transformation.  Thus $\zS_E$ and
$\zS_F$ differ only by a linear fractional transformation.

Now we return to Theorem~\ref{thm:twonine}.  As in the discussion
following Theorem~\ref{thm:algcformn}, we construct lattices $\zL_1$
and $\zL_2$ so that $\zL_2/\zL_1\cong \mathbb{Z}/2\mathbb{Z}$.
Choosing an identification map from $\mathbb{R}^2/\zG_2$ to
$\mathbb{R}^2/\zG_1$ obtains a Euclidean Thurston map $f$.  We choose
$h$ so that $h(P_f)=p_1(H)$, where $H$ is the nonseparating subset of
$\zL_2/2\zL_1\cong \mathbb{Z}/4\mathbb{Z}\oplus
\mathbb{Z}/2\mathbb{Z}$ as in Example~\ref{ex:degtwo}.  The map
$F=h\circ f$ is a Thurston map and $\zS_F$ is constant, but $F$ is not
nearly Euclidean because it has only three postcritical points
(consistent with Theorem~\ref{thm:nottwo}).  Let $\zL_3$ be any
lattice containing $\zL_2$.  Choosing an identification map from
$\mathbb{R}^2/\zG_3$ to $\mathbb{R}^2/\zG_2$ gives rise to a Euclidean
Thurston map $g$.  Let $G=h\circ g\circ h^{-1}$, and let $E=h\circ
g\circ f=G\circ F$.  Then $E$ is a NET map if it has at least four
postcritical points.  The discussion earlier in this remark shows that
$\zS_E$ is constant.  Because the index of $\zL_2$ in $\zL_3$ is
arbitrary, this gives another proof of Theorem~\ref{thm:twonine} if
$2|d$.  This way of showing that $\zS_E$ is constant by expressing $E$
as a composition $G\circ F$ is in the spirit of Proposition 5.1 of
\cite{BEKP}.  A major difference between the treatment here and that
in Proposition 5.1 of \cite{BEKP} is that here $\zS_E$ factors with
one of the factors a constant function, while there $\zS_E$ factors
through a trivial Teichm\"{u}ller space.  Note, in particular, that
the rational functions in Example~\ref{ex:double} with constant
pullback maps on Teichm\"{u}ller space factor as in the previous
paragraph.

When 2 is replaced by 9, the earlier observations in this remark again
show that once we obtain an example with degree 9, then we obtain an
example for every degree divisible by 9.

\end{remark}

For our first nonexistence result, we prove that there does not exist
a NET map with degree 2 whose Teichm\"{u}ller map is constant.

\begin{thm}\label{thm:nottwo} There does not exist a NET map with
degree 2 whose Teichm\"{u}ller map is constant.
\end{thm}
  \begin{proof} Let $f\co S^2\to S^2$ be a Euclidean Thurston map with
degree 2 whose Teichm\"{u}ller map is constant.  We seek a
contradiction.

As in Section~\ref{sec:defns}, there exist lattices $\zL_1\subseteq
\zL_2\subseteq \bR^2$ such that the canonical group homomorphism
$\widetilde{f}\co \bR^2/2 \zL_1\to \bR^2/2 \zL_2$ lifts $f$.
Theorem~\ref{thm:algcformn} implies that $p_1^{-1}(P_2)$ is a
nonseparating subset of $\zL_2/2 \zL_1$.  Since $\deg(f)=2$, it
follows that $\zL_2/2 \zL_1\cong A=\bZ/4 \bZ\oplus \bZ/2 \bZ$.  We fix
such an isomorphism, and let $H$ be the subset of $A$ corresponding to
$p_1^{-1}(P_2)$.  Group inversion generates an equivalence relation on
$A$ whose equivalence classes have either one or two elements.  Here
they are.
  \begin{equation*}
\{(0,0)\}\quad\{\pm (1,0)\}\quad\{(2,0)\}\quad\{(0,1)\}\quad\{\pm
(1,1)\}\quad\{(2,1)\}
  \end{equation*}
The set $H$ contains exactly four of these six subsets.

In this paragraph we prove by contradiction that $H$ contains the
elements of order 4 in $A$.  Suppose that $(1,0)\notin H$.  Let
$B=\left<(1,0)\right>=\{(0,0),\pm (1,0),(2,0)\}$, a cyclic subgroup of
$A$ such that $A/B$ is cyclic.  Let $c_1$, $c_2$, $c_3$, $c_4$ be the
coset numbers for $H$ relative to $B$ and the generator of $A/B$.
Since $H$ contains four of the six subsets in the last display, it
contains either $(0,0)$ or $(2,0)$.  If $H$ contains both of these
elements, then $c_1=c_2=0$ and $c_3=c_4=1$, contrary to the assumption
that $H$ is nonseparating.  So $H$ contains exactly one of $(0,0)$ and
$(2,0)$.  Hence it contains $(1,1)$.  Now applying the same argument
with $B'=\left<(1,1)\right>$ obtains a contradiction.  This
contradiction implies that $(1,0)\in H$.  By symmetry, $(1,1)\in H$.
Thus $H$ contains the elements of order 4 in $A$.

The elements of order 4 in $\zL_2/2 \zL_1$ are precisely the elements
of $\bR^2/2 \zL_1$ whose images under $p_1$ are critical points of
$f$.  Since $H$ contains the elements of order 4 in $A$, the critical
points of $f$ are postcritical.  So in order for the postcritical set
$P_2$ of $f$ to have four points, the restriction of $f$ to $P_2$ must
be surjective.  This implies that $p_1^{-1}(P_2)$ contains a
representative from every coset of $\ker(\widetilde{f})=2 \zL_2/2
\zL_1$ in $\zL_2/2 \zL_1$.  It follows that $H$ contains exactly one
of the elements in $2A=\{(0,0),(2,0)\}$.  Now we obtain a contradiction
as in the previous paragraph using $B=\left<(1,0)\right>$.

This proves Theorem~\ref{thm:nottwo}.

\end{proof}

Now we prove a general nonexistence theorem.

\begin{thm}\label{thm:notcyclic}  Let $A$ be a finite
Abelian group such that $A/2A\cong \bZ/2 \bZ\oplus \bZ/2 \bZ$ and $2A$
is a cyclic group with odd order.  Then $A$ does not contain a
nonseparating subset.
\end{thm}
  \begin{proof} We proceed by contradiction.  Suppose that $A$
contains a nonseparating subset $H=\{\pm h_1,\pm h_2,\pm h_3,\pm
h_4\}$.  The assumptions imply that $A\cong 2A\oplus (A/2A)$.  Every
subgroup of order 2 in $A$ is a cyclic subgroup $B$ such that $A/B$ is
cyclic, and every subgroup of order 2 in $A$ is contained in a cyclic
subgroup $B$ such that $A/B$ has order 2, and so is cyclic.

We consider the restriction to $H$ of the canonical group homomorphism
from $A$ to $A/2A$.  Each of the inverse pairs $\{\pm h_1\}$, $\{\pm
h_2\}$, $\{\pm h_3\}$, $\{\pm h_4\}$ determines an element of $A/2A$.
There are five possible forms for this map: 1) its image contains all
four elements of $A/2A$; 2) its image contains three elements of
$A/2A$; 3) its image contains two elements of $A/2A$, each the image
of two of the inverse pairs; 4) its image contains two elements of
$A/2A$, one the image of one inverse pair and the other the image of
three inverse pairs; 5) its image consists of one element of $A/2A$.

In this paragraph we assume that the map from $H$ to $A/2A$ has forms
either 1, 2 or 3.  A case analysis shows that there exists a subgroup
of order 2 in $A/2A$ which contains the image of exactly two of the
inverse pairs in $H$.  The inverse image in $A$ of this subgroup of
order 2 is a cyclic subgroup $B$ of $A$ such that $A/B$ has order 2,
hence is cyclic, and $B$ contains exactly two of the inverse pairs of
$H$.  Hence the coset numbers for $H$ relative to $B$ and the
generator of $A/B$ are $c_1=c_2=0$ and $c_3=c_4=1$.  This contradicts
the assumption that $H$ is nonseparating.

In this paragraph we assume that the map from $H$ to $A/2A$ has forms
either 4 or 5.  By translating $H$ by an element of order 2 if
necessary as in Lemma~\ref{lemma:translate}, we may assume that the
image of $H$ in $A/2A$ is contained in a subgroup of order 2.  The
inverse image in $A$ of this subgroup of order 2 is a cyclic subgroup
$C$ of $A$ such that $A/C$ has order 2, and $C$ contains $H$.  Let $B$
be a subgroup of order 2 in $A$ not contained in $C$.  Then $B$ is a
cyclic subgroup of $A$ such that $A/B\cong C$ is cyclic, and the coset
numbers for $H$ relative to $B$ and any generator of $C$ are distinct
because $H$ is a subset of $C$.  This contradiction completes the
proof of Theorem~\ref{thm:notcyclic}.

\end{proof}

\begin{thm}\label{thm:notsqfree} There does not exist a NET map with
degree an odd squarefree integer and constant Teichm\"{u}ller map.
\end{thm}
  \begin{proof} Suppose that $f\co S^2\to S^2$ is a NET map, and
suppose that the degree $d$ of $f$ is an odd squarefree integer.  We
maintain the notation of Section~\ref{sec:defns}.  Let $A=\zL_2/2
\zL_1$.  Then $A/2A\cong \bZ/2 \bZ\oplus \bZ/2 \bZ$, and since
$\left|A\right|=4d$, it follows that $\left|2A\right|=d$.  Since every
finite Abelian group with squarefree order is cyclic, $2A$ is cyclic.
Theorem~\ref{thm:notcyclic} now implies that $A$ does not contain a
nonseparating subset.  Theorem~\ref{thm:algcformn} now implies that
the Teichm\"{u}ller map of $f$ is not constant.  This proves
Theorem~\ref{thm:notsqfree}.

\end{proof}

It should be true that there exists a NET map with degree $d$ and
constant Teichm\"{u}ller map if and only if $d>2$ and $d$ is divisible
by either 2 or 9.  See Saenz Maldonado's thesis \cite{SM} for further
progress in this direction.

\end{document}